\documentclass{amsart}

\usepackage[margin=1in]{geometry} %change to 1 before posting
\usepackage{lipsum}

\usepackage{tikz-cd}
\usepackage[unicode]{hyperref}
\usepackage{graphicx,color}
\usepackage{amssymb,amsmath}
\usepackage{mathtools}
\usepackage{mathrsfs}
\usepackage{enumerate}
\usepackage{tkz-graph}
\usepackage{ragged2e}
\usepackage{enumitem,pifont}
\usepackage{xypic}
\usepackage{tikz}
\usetikzlibrary{arrows,decorations.pathmorphing,automata,backgrounds}
\usetikzlibrary{backgrounds,positioning}
\tikzset{plain/.style={circle,draw,very thick,
		inner sep=0pt,minimum size=6mm}}
\tikzset{empty/.style={circle,inner sep=0pt,minimum size=6mm}}
\tikzset{emptyvt/.style={circle,inner sep=0pt,minimum size=0mm}}
\usepackage{caption}
\usepackage{subcaption}
\usepackage{float}

\usepackage{egothic}
\usepackage[T1]{fontenc}
\usepackage{yfonts}

\usepackage{dbnsymb}

\definecolor{ao}{rgb}{0.0, 0.5, 0.0}

\newcommand{\Z}{\mathbb{Z}}

\newcommand{\V}{\mathsf{V}}
\newcommand{\W}{\mathsf{W}}

\newcommand{\bV}{\mathbf{V}}
\newcommand{\bW}{\mathbf{W}}
\newcommand{\bE}{\mathbf{E}}

\newcommand{\sV}{\mathsf{V}}
\newcommand{\sE}{\mathsf{E}}

\newcommand{\sW}{\mathsf{W}}

\newcommand{\tin}{\textnormal{in}}
\newcommand{\tout}{\textnormal{out}}
\newcommand{\trace}{\textgoth{t}}

\newcommand{\calA}{\mathcal{A}}

\newcommand{\calI}{\mathcal{I}}

\newcommand{\calF}{\mathcal{F}}
\newcommand{\calG}{\mathcal{G}}
\newcommand{\calS}{\mathcal{S}}

\newcommand{\id}{\operatorname{id}}
\newcommand{\Id}{\operatorname{Id}}
\newcommand{\colim}{\operatorname{colim}}
\newcommand{\Hom}{\operatorname{Hom}}
\newcommand{\End}{\operatorname{End}}
\newcommand{\Alg}{\operatorname{Alg}}

\newcounter{dummy} \numberwithin{dummy}{section}

\newtheorem{thm}[dummy]{Theorem}
\newtheorem{prop}[dummy]{Proposition}
\newtheorem{lemma}[dummy]{Lemma}

\newtheorem*{thm*}{Theorem}
\newtheorem*{prop*}{Proposition}

\theoremstyle{definition}
\newtheorem{definition}[dummy]{Definition}
\newtheorem{cons}[dummy]{Construction}
\newtheorem{example}[dummy]{Example}
\newtheorem{remark}[dummy]{Remark}

\numberwithin{equation}{section}

\usepackage{multicol}

\title{Circuit Algebras are wheeled props}

\author[Z. Dancso]{Zsuzsanna Dancso}
\address{School of Mathematics and Statistics\\ The University of Sydney\\ Sydney, NSW, Australia}
\email{zsuzsanna.dancso@sydney.edu.au}

\author[I. Halacheva]{Iva Halacheva}
\address{Department of Mathematics \\ Northeastern University \\ Boston, Massachusetts, USA}
\email{i.halacheva@northeastern.edu}

\author[M. Robertson]{Marcy Robertson}
\address{School of Mathematics and Statistics \\ The University of Melbourne \\ Melbourne, Victoria, Australia}
\email{marcy.robertson@unimelb.edu.au}

\date{\today}

\begin{document}

\begin{abstract} 
Circuit algebras, introduced by Bar-Natan and the first author, are a generalization of Jones's planar algebras, in which one drops the planarity condition on ``connection diagrams''. They provide a useful language for the study of virtual and welded tangles in low-dimensional topology. In this note, we present the circuit algebra analogue of the well-known classification of planar algebras as pivotal categories with a self-dual generator. Our main theorem is that there is an equivalence of categories between circuit algebras and the category of linear wheeled props -- a type of strict symmetric tensor category with duals that arises in homotopy theory, deformation theory and the Batalin-Vilkovisky quantization formalism. 
	
\end{abstract}
	
\maketitle\bibliographystyle{amsalpha}

\section{Introduction}

In~\cite{Jones:PA}, Jones introduced the notion of a planar algebra as an axiomatization of the standard invariant of a finite index subfactor. A planar algebra is an algebraic structure whose operations are parametrized by {\em planar tangles}. A planar tangle is a $1$-manifold with boundary, embedded in a ``disc with $r$ holes'', where the boundary points of the $1$-manifold lie on the boundary circles of 
the disc with holes: an example is shown in Figure~\ref{fig:WiringDiagram} on the left.
Planar tangles form a coloured operad where composition is defined by gluing the outer circle of one tangle into an inner circle of another, as long as the tangle endpoints match (eg. \cite{Jones:PA}, \cite[Definition 2.2]{BHP12}). 
Given a field $\Bbbk$ of characteristic zero, a \emph{planar algebra} is a sequence of $\Bbbk$-vector spaces, which admit an action by the operad of planar tangles (\cite{Jones:PA},\cite[Definition 2.4]{BHP12}).

\begin{figure}[h]
	\includegraphics[width=9cm]{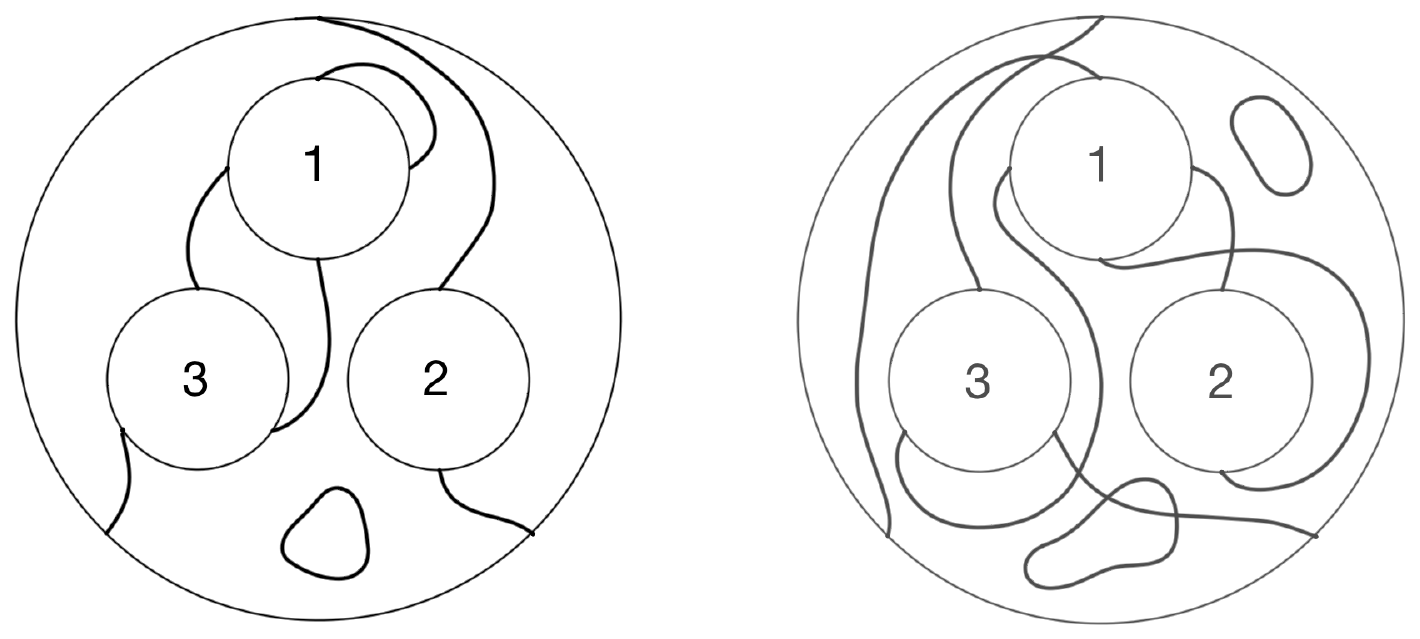}
	\caption{A planar tangle on the left, a wiring diagram on the right.}\label{fig:WiringDiagram}
\end{figure}

Planar algebras arise in many contexts where a tensor category with a ``good'' notion of duals is involved, and have played an important role in the theories of subfactors, conformal and quantum field theories and knot and tangle invariants. There is a well-known classification of planar algebras as \emph{pivotal categories} with a symmetrically self-dual generator (\cite{MR2559686}, \cite{MR3624399}, \cite{BHP12}). Pivotal categories are rigid tensor categories in which every object is isomorphic to its double dual (Section~\ref{section: categorical description}).

{\em Circuit algebras} were defined by Bar-Natan and the first author in \cite{BND:WKO2} as a generalization of planar algebras which provides a convenient language for \emph{virtual} and \emph{welded} tangles in low-dimensional topology. The term, inspired by electrical circuits, was coined by Bar-Natan. Circuit algebras are defined similarly to planar algebras, but their operations are parametrized by not necessarily planar {\em wiring diagrams} (Definition~\ref{def:WD}). Wiring diagrams can be defined similarly to planar tangles, but with embedded submanifolds replaced by abstract $1$-manifolds whose boundary points are identified with points on the boundary circles (Figure~\ref{fig:WiringDiagram}). This results in a purely combinatorial -- no longer topological -- structure; we discuss this distinction in detail in Section~\ref{sec:CA}. 

Wiring diagrams can be composed in the same manner as planar tangles, making the collection of wiring diagrams into a coloured operad. A circuit algebra is an algebra over this operad: a collection of vector spaces along with linear maps between them, parametrized by wiring diagrams. In Section~\ref{sec:CA} we present the definition of circuit algebras in detail. In Section~\ref{sec:Ex} we present an example from knot theory, and explain the relationship to planar algebras in more detail. 

Circuit algebras are more recent and, so far, not as widely studied as planar algebras. In the last few years a number of authors have used circuit algebras to study invariants for virtual and welded tangles (\cite{BND:WKO2}, \cite{DF}, \cite{Hal16},\cite{Tub}). In this paper we expand the definition of \cite{BND:WKO2} to include more detail and improve accessibility for a wider audience.   In Section~\ref{sec:Ex}, we present the example of virtual tangles, which can be defined as both a planar algebra and a circuit algebra. Indeed, every circuit algebra has an underlying planar algebra, as we prove in Proposition~\ref{prop: adjunction PA and CA} :
\begin{prop*}
	There exists a pair of adjoint functors $\begin{tikzcd} \mathsf{CA} \arrow[r, hook, shift right =1] & \mathsf{PA}\arrow[l, shift right =1]\end{tikzcd}$.  
\end{prop*}
\noindent There is no expectation, however, that this adjunction should be an equivalence. As just one example, classical tangles from knot theory admit a planar algebra structure, but not a circuit algebra structure.

The main result of this paper is a classification result for circuit algebras in terms of {\em linear wheeled props}, analogous to that of planar algebras via pivotal categories. A linear prop is a strict symmetric tensor category whose objects are generated by a single object. \emph{Wheeled props} \cite{mms}, are rigid props or, equivalently, strict symmetric tensor categories with duals which are generated by a single object. Wheeled props arise naturally in deformation theory and the Batalin--Vilkovisky quantization formalism of theoretical physics, invariant theory, and other abstract settings where a generalized trace operation plays a role (\cite{MR2648322}, \cite{MR2762550}, \cite{CFP20}, \cite{DM19}). Our main theorem (Theorem~\ref{thm:main}) is the following:
\begin{thm*}
	There is an equivalence of categories between circuit algebras and linear wheeled props $$\mathsf{CA}\cong\mathsf{wPROP}.$$
\end{thm*} 
In Section~\ref{sec:WP} we give a formal introduction to linear wheeled props  as algebras over a monad on diagrams of vector spaces indexed by directed graphs. The key step in this identification is understanding that wiring diagrams can be identified with directed graphs (Lemma~\ref{lemma: graphs are wiring diagrams}).  In order to make this paper readable to the largest possible audience, however, we have also included an appendix containing an equivalent, axiomatic, definition of wheeled props (Definition~\ref{def:baised modular operad}). To keep our comparison with the algebraic classification of planar algebras in mind, we point out in Proposition~\ref{prop: wheeled props are pivcat} that linear wheeled props embed into the category of linear pivotal categories. In summary, we have the following diagram in which the two horizontal adjunctions are equivalences of categories:  \[\begin{tikzcd} \mathsf{CA} \arrow[d, hook, shift left = 1]\arrow[r, "\cong"] & \mathsf{wProp} \arrow[l]\arrow[d, hook, shift right=1]\\
\mathsf{PA}\arrow[r, "\cong"] \arrow[u, shift left=1.5] & \mathsf{PivCat} \arrow[l]\arrow[u, shift right =1.5]
\end{tikzcd}.\] We further discuss the conjectural commutativity of this diagram at the end of Section~\ref{section: categorical description}.

Throughout this paper we focus our attention on linear wheeled props, that is, wheeled props enriched in $\Bbbk$-vector spaces. This choice was made to simplify exposition, but all objects can be defined in any closed, symmetric monoidal category and all arguments still hold.  We further note that one goal of this paper is to provide a bridge between the tensor categories, knot theory and category theory communities. As such, the level of detail is intended to make each section accessible to mathematicians working in the other areas.

\medskip 
\noindent\textbf{Acknowledgements.} 
Part of this work was completed while the first and third authors were in residence at MSRI for the program ``Higher Categories and Categorification'' in 2020. In addition, we would like to thank Dror Bar-Natan, Scott Morrison and Sophie Raynor for suggestions of key references and Arun Ram for many helpful comments. 
\newpage
\tableofcontents
%Informally, a PROP can be described as a generalisation of a category where morphisms have $n$ inputs and $m$ outputs (\cite{}). A nice ``visual'' example is the Segal PROP whose morphisms are elements of the moduli space of Riemann surfaces bounding $m+n$ labeled non-overlapping punctures. Composition is given by ``sewing'' Riemann surfaces together along boundaries \cite{Segal}.  This algebraic structure allowed Segal to describe certain field theories as ``representations''\footnote{In the literature one uses the term algebra over a PROP for what is intuitively a representation of the PROP.} of the PROPs whose morphisms are cobordisms.  

%\emph{Wheeled PROPs} were introduced by Markl, Merkulov, and Shadrin \cite{mms} as an extension of PROPs equipped with a contraction, or ``trace,'' operation. These structures usually arise in deformation theory and the Batalin--Vilkovisky quantization formalism of theoretical physics. If we are picturing the morphisms in a PROP as a Riemann surface with $m+n$-punctures the ``trace'' has the effect of gluing two of these punctures together, creating a new handle on the surface.  The main theorem of this paper establishes that circuit algebras are equivalent to $\Bbbk$-linear wheeled PROPs: 

%Circuit algebras were introduced in \cite{BND:WKO2} as a language for welded and virtual tangles.

%%%%%%%%%%%%%%%%%%%%%%%%%%%%%%%%%%%%%%%%%%%%%%%%%%%%%%%%%%%%%%%%%%

\section{Circuit algebras}\label{sec:CA}

A circuit algebra, much like a planar algebra, is a family of vector spaces with operations indexed by \emph{wiring diagrams}: abstract $1$-manifolds with boundary whose boundary points are identified with endpoints along the circles of a ``disc with holes''.  We will develop the comparison with planar algebras further in Section~\ref{sec:Ex}, but for now we alert the reader to the fact that we are using a definition of planar algebra without shading, such as that in \cite[Definition 2.3]{HPT16} or \cite[Section 2]{BHP12}.  The key difference between wiring diagrams and planar tangles is that while planar tangles are inherently topological objects, wiring diagrams are purely combinatorial, and their topological description below is merely for convenience (see Remark~\ref{rmk:Matching}). 

\medskip 

Throughout this paper, let $\calI$ denote a countable alphabet, the set of {\em labels}. The following definition is an expanded version of the definition given in \cite[Section 2]{BND:WKO2}.

\begin{definition}\label{def:WD}
	An \emph{oriented wiring diagram} is a triple $D=(\calA,M,f)$ consisting of:
	\begin{enumerate}
		\item A set $\calA=\{A_0^{\text{out}}, A_0^{\text{in}},A_1^{\text{out}},A_1^{\text{in}},\hdots,A_r^{\text{out}},A_r^{\text{in}}\}$ of  sets of labels, for some non-negative integer $r$. That is, $A_i^{\text{out}}, A_i^{\text{in}} \subseteq \calI$ for each $0\leq i\leq r$. The elements of the sets $A^{\text{out}}_i$ are referred to as \emph{outgoing labels} and the elements of $A^{\text{in}}_i$ are \emph{incoming labels}. The sets $A_0^{\tout}$ and $A_{0}^{\tin}$ play a distinguished role: their elements are called the \emph{output labels} of the diagram, while the sets $A_1^{{\text{out}}}, A_{1}^{\text{in}},...,A_r^{\text{out}},  A_r^{\text{in}}$ contain {\em input labels} of the diagram. We write $A_{i}^{\text{out}/\text{in}}$ to mean  ``$A_{i}^{\text{out}}$ and $A_{i}^{\text{in}}$, respectively''. 
		
		\item An oriented compact $1$-manifold $M$, with boundary $\partial M$, regarded up to orientation-preserving homeomorphism. The connected components of $M$ are homeomorphic to either an oriented circle (with no boundary) or an oriented interval with one beginning and one ending point. We write $\partial M^{\tout}$ for the set of {\em beginning} boundary points of $M$, and  $\partial M^{\text{in}}$ for the set of {\em ending}  boundary points, so $\partial M = \partial M^{\text{out}} \sqcup \partial M^{\tin}$.
		
		\item Bijections\footnote{If the sets $\{A_i^{\text{out}/\text{in}}\}$ are not pairwise disjoint, replace the unions $(\cup_{i=0}^r A_i^{\text{out}})$ and  $(\cup_{i=0}^r A_i^{\text{in}})$ by the set of triples $\{(a,i,\text{out}/\text{in})\, | \, a\in A_i^{\text{out}/\text{in}}, 0\leq i \leq r\}$.} \label{fn:Triples}
		$$\partial M^{\text{out}} \xrightarrow{f} \cup^r_{i=0} {A^{\text{out}}_i} \quad \textrm{and}\quad  \partial M^{\text{in}} \xrightarrow{f} \cup^r_{i=0} {A^{\text{in}}_i}.$$
	\end{enumerate}
\end{definition}

\medskip

Wiring diagrams have a convenient pictorial representation shown in Figure~\ref{fig:WiringDiagramOri}, which illuminates their relationship to planar algebras.  A {\em disc with $r$ holes}, $$D_0 \setminus (\mathring{D}_1 \sqcup \mathring{D}_2 \sqcup  \hdots \sqcup \mathring{D}_r),$$ is obtained by removing $r$ disjoint numbered open discs with disjoint boundaries from the interior of a bigger disc. The boundaries of the removed discs are called the {\em input circles}, while the boundary of the big disc is numbered zero and called the {\em output circle}. 

Assume $\calI$ is ordered, which is often the case with the labels we use in examples, such as natural numbers or Roman letters. Arrange the elements of $A_0^{\text{out}}$, then $A_{0}^{\text{in}}$ in the order induced by the ordering on $\calI$ at uniform intervals along the output circle, and the elements of $A_i^{\text{out}}$, then $A_{i}^{\text{in}}$ for $i=1,..., r$, in order at uniform intervals along the $i$th input circle. Represent the manifold $M$ and the identification $f$ as immersed curves in $D_0 \setminus (\mathring{D}_1 \sqcup \mathring{D}_2 \sqcup  \hdots \sqcup \mathring{D}_r)$. Note that the specific immersion is not part of the data of the wiring diagram.

\begin{figure}[h]
	\includegraphics[height=.4\linewidth]{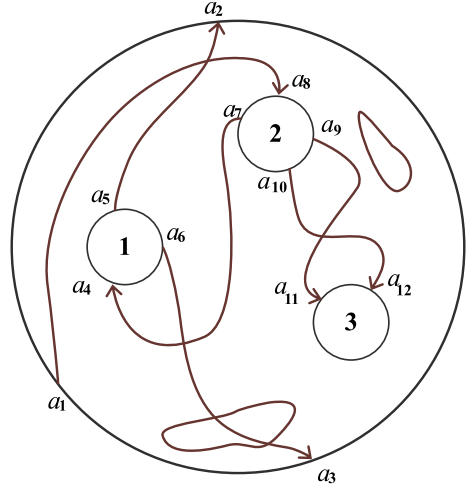}
	\caption{An example of an oriented wiring diagram. The labels sets are $A_0^{\tout}=\{a_1\}$, $A_0^{\tin}=\{a_2,a_3\}$, $A_1^\tout=\{a_5, a_6\}$, $A_1^\tin=\{a_4\}$, $A_2^\tout=\{a_7, a_9, a_{10}\}$, $A_2^\tin=\{a_8\}$, $A_3^\tout=\emptyset$, $A_3^\tin=\{a_{11},a_{12}\}$. The manifold, drawn in brown lines, is a disjoint union of six oriented intervals and two oriented circles. In the picture we don't draw arrows for circles, since they are abstract, not embedded: there is only one homeomorphism type of an oriented circle.}\label{fig:WiringDiagramOri}
\end{figure}

\medskip 

\begin{definition}\label{def:WiringComp} Given two wiring diagrams
	\[\begin{array}{c} D=\left(\mathcal{A}=\{A_0^{\text{out}/\text{in}},A_1^{\text{out}/\text{in}},\hdots, A_r^{\text{out}/\text{in}}\},M,f\right) \\ D'=\left(\mathcal{B}=\{B_0^{\text{out}/\text{in}},B_1^{\text{out}/\text{in}},\hdots,B_s^{\text{out}/\text{in}}\},N,g\right)\end{array}\] 
	the {\em composition} $D \circ_i D'$, for $0 \leq i \leq r$, is defined whenever $A_i^{\text{out}/\text{in}}={B_0}^{\text{in}/\text{out}}$ as sets.\footnote{It is important that $A_i^\tin$ is identified with $B_0^\tout$ and vice versa.}  The resulting \emph{composite wiring diagram} $D \circ_i D' =(\mathcal{A}\circ_i\mathcal{B},M \sqcup_{\varphi} N, f \sqcup_{\varphi} g)$ consists of:
	
	\begin{enumerate} 
		\item the label sets $$\mathcal{A}\circ_i\mathcal{B} =\{A_0^{\text{out}/\text{in}},A_1^{\text{out}/\text{in}},\hdots, A_{i-1}^{\text{out}/\text{in}}, B_1^{\text{out}/\text{in}},\hdots, B_s^{\text{out}/\text{in}}, A^{\text{out}/\text{in}}_{i+1}, \hdots, A^{\text{out}/\text{in}}_r\}$$
		
		\item a compact oriented $1$-manifold $M \sqcup_{\varphi} N$, obtained by gluing $M$ and $N$ along the map $\varphi$, which identifies  the boundary points in $f^{-1}(A_i^{\text{out}/\text{in}})$ and $g^{-1}(B_0^{\text{in}/\text{out}})$: 	
		\[\begin{tikzcd}\partial M \supseteq f^{-1}(A_i^{\text{out}/\text{in}})\arrow[rr, bend left =15, "\varphi"] \arrow[r, "f"] & A_i^{\text{out}/\text{in}} = B_0^{\text{in}/\text{out}} \arrow[r, "g^{-1}"] & g^{-1}(B_0^{\text{in}/\text{out}}) \subseteq \partial N; \end{tikzcd}\] %\zdnote{Could Marcy work some diagram magic in $(2)$ here and put a curved arrow with $\varphi$ on it, from $f^{-1}(A_i^{\text{out}/\text{in}})$ to $g^{-1}(B_0^{\text{in}/\text{out}})$?}

		\item a bijection $f \sqcup_{\varphi} g$ defined to be $f$ on $\partial M \setminus  f^{-1}(A_i^{\text{out}/\text{in}})$ and $g$ on $\partial N \setminus  g^{-1}(B_0^{\text{out}/\text{in}})$. 
		
	\end{enumerate} 
\end{definition}

%\begin{equation*}\label{eq:wiring_diagram_composition}D\circ_i D':=(\{A_0^{\text{out}/\text{in}},A_1^{\text{out}/\text{in}},\hdots, A_{i-1}^{\text{out}/\text{in}}, B_1^{\text{out}/\text{in}},\hdots, B_s^{\text{out}/\text{in}}, A^{\text{out}/\text{in}}_{i+1}, \hdots, A^{\text{out}/\text{in}}_r, \},  M \sqcup_{\varphi} N, f\sqcup_{\varphi} g). \end{equation*} The map $\varphi$ in the definition is the identification of the boundary points in $f^{-1}(A_i^{\text{out}/\text{in}})$ and $g^{-1}(B_0^{\text{in}/\text{out}})$:	\[\partial M \supseteq f^{-1}(A_i^{\text{out}/\text{in}}) \xrightarrow{f} A_i^{\text{out}/\text{in}} = B_0^{\text{in}/\text{out}} \xrightarrow{g^{-1}}g^{-1}(B_0^{\text{in}/\text{out}}) \subseteq \partial N.\]	The manifold $M \sqcup_{\varphi} N$ is the oriented 1-manifold obtained by gluing M and N along $\varphi$ and the map $f \sqcup_{\varphi} g$ is defined to be $f$ on $\partial M \setminus  f^{-1}(A_i^{\text{out}/\text{in}})$ and $g$ on $\partial N \setminus  g^{-1}(B_0^{\text{out}/\text{in}})$.

\medskip 

Composition can be pictorially represented by shrinking the wiring diagram $D'$ and gluing it into the $i$th input circle of $D$ so that the labels match. Then delete the outer circle of $D'$, as shown in Figure~\ref{fig:WDComp}. The ordering of the input discs of the composite diagram follows the ordering prescribed in $(1)$ of Definition~\ref{def:WiringComp}. Note that composition may create closed components (circles) in $M\sqcup_{\varphi} N$.

\begin{figure}[h]
	\includegraphics[width=13cm]{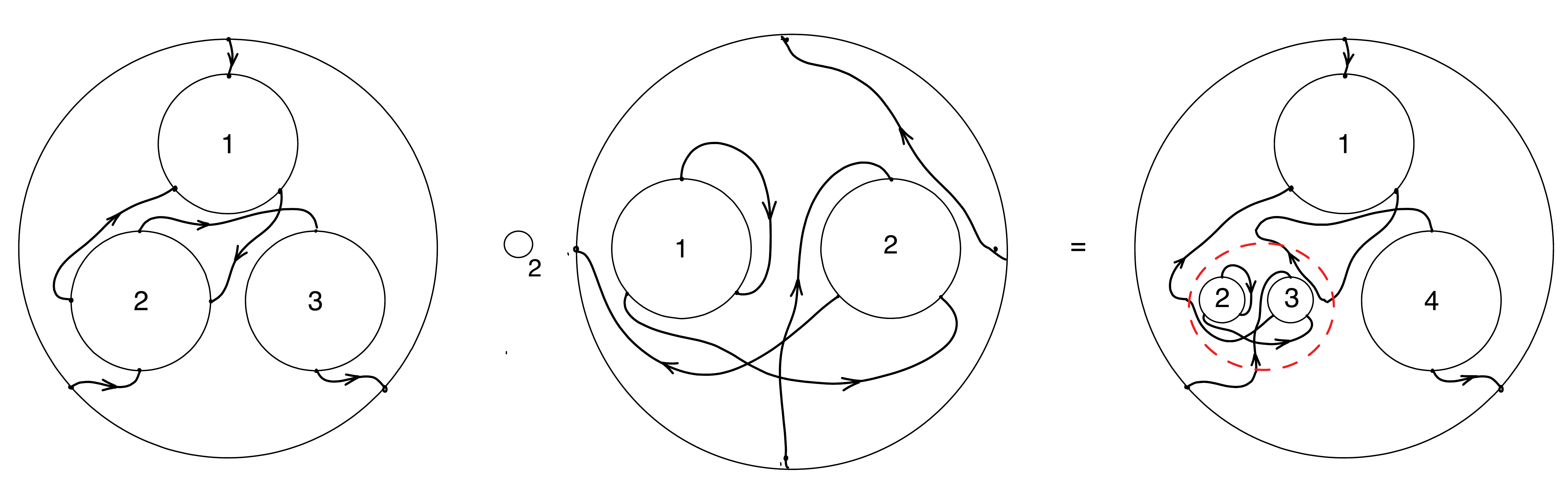}
	\caption{An illustration of oriented wiring diagram composition; labels are suppressed for simplicity, but must match. The deleted outer disc of $D'$ is shown as a broken line red circle.}\label{fig:WDComp}
\end{figure}

\begin{prop}\label{rmk:Matching} A wiring diagram $(\calA, M, f)$ is equivalent to a triple $(\calA, p, l),$ where $\calA$ is as above, $p$ is a perfect matching (set bijection) between\footnote{If the sets $A_i^{\tout/\tin}$ are not disjoint, replace them in the disjoint unions with sets of triples, as before.}
	$\bigsqcup_{i=0}^r 
	A^{\text{out}}_i$ and $\bigsqcup_{i=0}^r A^{\text{in}}_i$, and $l\in \Z_{\geq 0}$ is a non-negative integer, 
	``the number of circles in $M$''.
\end{prop}

\begin{proof} The oriented $1$-manifold $M$ in Definition~\ref{def:WD} is a disjoint union of a finite number of oriented intervals and 
	circles.  The endpoints of the intervals are identified with the labels, and as such, the only role of the intervals is 
	to define a perfect matching -- that is, a set bijection -- between the sets of incoming and outgoing labels. Since there is only one homeomorphism type of an oriented circle, one can equivalently simply remember the number of circle components of $M$.
\end{proof}

This alternative definition illuminates that wiring diagrams are combinatorial -- as opposed to topological -- objects. This is the main difference between wiring diagrams and planar tangles. We make this combinatorial description even more explicit by identifying wiring diagrams with a combinatorial formalism of directed graphs in Lemma~\ref{lemma: graphs are wiring diagrams}. Our main reason for presenting the definition using $1$-manifolds is the composition of wiring diagrams: describing the perfect matching and number of circles resulting from a composition in purely combinatorial terms is possible, but a headache (we encourage the reader to try). The definition of composition through gluing 1-manifolds is much more elegant and concise.

\medskip

As with planar tangles, one can show that the collection of oriented wiring diagrams 
%$\mathsf{WD}=\{\mathsf{WD}(\mathcal{A})\}_{\mathcal{A}\in\calI}$ 
together with the $\circ_i$ compositions assembles into a (coloured) {\em operad}. A proof of this is a simple exercise along the same lines as the description of the operad of planar tangles in \cite{Jones:PA} or the construction of an operad of wires in \cite[Section 3]{polyak2010alexanderconway}. A concise definition for an oriented circuit algebra is an algebra over this operad; we explain the notion of operads and algebras over them in more detail in Section~\ref{sec:Ex}. Here we unwind this concept and arrive at the following definition; an expanded version of that in \cite[Definition 2.10]{BND:WKO2}: 

\begin{definition}\label{def:CA}
	An \textbf{oriented circuit algebra} $\V$ consists of a collection of vector spaces indexed by pairs of label sets, $\{\V[S^{\tout};S^{\tin}]\}_{S^{\text{out}}, S^{\text{in}}\subseteq \calI}$ , together with a family of linear maps between these, parametrised by oriented wiring diagrams. Namely, for each wiring diagram $D=(\calA,M,f)$, there is a corresponding linear map 
	$$F_D:\V[A^{\tout}_1;A^{\tin}_1] \otimes \hdots \otimes \V[A^{\tout}_r;A^{\tin}_{r}] \rightarrow \V[A^{\tin}_0;A^{\tout}_0].$$ This data must satisfy the following axioms: 
	
	\begin{enumerate}
		\item The composition of wiring diagrams corresponds to composition of linear maps in the following sense.  Let \[D=(\{A^{\tout/\tin}_0,\hdots,A^{\tout/\tin}_r\},M,f), \quad D'=(\{B^{\tout/\tin}_0,\hdots,B^{\tout/\tin}_s\},N,g)\] be two wiring diagrams composable as $D\circ_i D'$. Then the map of vector spaces corresponding to the composition $D\circ_i D'$ is 
		\[ F_{D\circ_{i}D'}= F_D\circ (\Id\otimes \dots \otimes \Id \otimes F_{D'}\otimes \Id \otimes \dots\otimes \Id), \]
		where $F_{D'}$ is inserted in the $i$th tensor component.
		\item There is an action of the symmetric groups $\mathcal S_r$ on wiring diagrams with $r$ input sets, which permutes (re-numbers) the input sets. 
		The assigment of linear maps to wiring diagrams is equivariant under this action in the following sense. Let $D=(\{A^{\tout/\tin}_0, A^{\tout/\tin}_1,\hdots,A^{\tout/\tin}_{r}\},M,f)$ be a wiring diagram, $\sigma\in \mathcal S_r$, and let
		$\sigma D= (\{A^{\tout/\tin}_0, A^{\tout/\tin}_{\sigma(1)},\hdots,A^{\tout/\tin}_{\sigma(k)}\},M, f)$ be the wiring diagram $D$ with the input sets re-ordered; note that the output set $A^{\tout/\tin}_0$ is fixed.
		Then the induced linear map
		$F_{\sigma \! D}$ is $F_D\circ \sigma^{-1}$, where $\sigma^{-1}$ acts on  
		$\V[A^{\tout/\tin}_{\sigma(1)}] \otimes \hdots \otimes \V[A^{\tout/\tin}_{\sigma(r)}]$ by permuting the tensor factors.  
	\end{enumerate}
\end{definition}

\begin{definition}	A \emph{morphism} of circuit algebras $\Phi:\mathsf{V}\rightarrow\mathsf{W}$ is a family of linear maps 
	$\{\Phi_{S^{\tout};S^\tin}:\mathsf{V}[S^\tout;S^\tin]\rightarrow \mathsf{W}[S^\tout;S^\tin]\}_{S^{\text{out}}, S^{\text{in}} \subseteq \calI}$
	which commutes with the action of wiring diagrams. That is, for any wiring diagram $D=(\calA, M, f)$ we have a commutative diagram:
	\[\xymatrixcolsep{5pc}\xymatrix{
		\V[A^{\tout}_1;A^{\tin}_1] \otimes \hdots \otimes \V[A^{\tout}_r;A^{\tin}_{r}] \ar[r]^-{(F_\V)_D} \ar[d]_{\Phi_{A_1^\tout;A_1^\tin}\otimes ... \otimes \Phi_{A_r^\tout;A_r^\tin} }& \V[A^{\tin}_0;A^{\tout}_0] \ar[d]^{\Phi_{A_0^\tin;A_0^\tout}} \\
		\W[A^{\tout}_1;A^{\tin}_1] \otimes \hdots \otimes \W[A^{\tout}_r;A^{\tin}_{r}] \ar[r]^-{(F_\W)_D} & \W[A^{\tin}_0;A^{\tout}_0]
	}
	\]
	More concisely put, a morphism of circuit algebras is a map of algebras over the operad of wiring diagrams. The category of all circuit algebras is denoted $\mathsf{CA}$. 
\end{definition}

\begin{example}\label{ex:IdAndPerm}
	For every pair of sets of labels $S,T \subseteq \calI$, there are left and right {\em identity wiring diagrams}. In the {\em perfect matching} notation introduced in Remark~\ref{rmk:Matching}:

	The left identity is  $\Id^L_{S,T}=(\{S,T,T,S\}, \Id, 0)$, where: 
	\begin{itemize}
		\item $A^{\tout}_0=S$, $A^{\tin}_0=T$, $A^{\tout}_1=T$, $A^{\tin}_1=S$;  
		\item and $\Id$ refers the perfect matching induced by the set identities $\Id_S$ and $\Id_T$.
	\end{itemize}
	The right identity is $\Id^R_{S,T}=\Id^L_{T,S}$.
	Then, for any wiring diagram $D$ with $A^{\tout}_i=S$, $A_i^{\tin}=T$, we have $D\circ_i \Id^R_{S,T} = D$. On the other hand if $D$ is such that $A^{\tout}_0=S$, $A^{\tin}_0=T$ then $\Id^L_{S,T} \circ_1 D=D$. Consequently, the corresponding linear maps are the identity maps $F_{\Id_{S,T}^R}=\Id_{V[S;T]}$ and $F_{\Id_{S,T}^L}=\Id_{V[T;S]}$.  See Figure~\ref{IDandReLabel}. 
	
\end{example}
\begin{example}\label{example:relabellingWD}
	In a similar vein to the left and right identities, given any pair of permutations $\sigma \in \calS_{S}$ and $\tau \in \calS_{T}$, there are {\em label permuting wiring diagrams} $D_{\sigma,\tau}=({S,T,T,S}, (\sigma,\tau), 0)$ where:
	\begin{itemize}
		\item $A^{\tout}_0=S$, $A^{\tin}_0=T$, $A^{\tout}_1=T$, $A^{\tin}_1=S$ and  
		\item the perfect matching $(\sigma,\tau)$ is given by the permutations $\sigma: S=A_1^\tin \to A_0^\tout =S$ and $\tau: T= A_1^\tout \to A_0^\tin =T$.
	\end{itemize}
	If $B$ is a wiring diagram with $A_0^\tin=T$, $A_0^\tout=S$ then $D_{\sigma,\tau}\circ_1 B$ is the wiring diagram $B$ with the labels in $A_0^\tin$ and $A_0^\tout$ permuted by $\sigma$ and $\tau$ respectively. Similarly, for a wiring diagram $C$ with $A_i^\tin=S$, $A_i^\tout=T$,  the composition $C\circ_i D_{\sigma,\tau}$ is the wiring diagram $C$ with the permutations $\sigma^{-1}$ and $\tau^{-1}$ applied to the labels in $A_i^\tout$ and $A_i^\tin$, respectively. The linear maps $F_{D_{\sigma,\id}}:\V[S;T]\to\V[S;T]$ give a left $\calS_S$ action on $\V[S;T]$, and $F_{D_{\id, \tau^{-1}}}$ gives a commuting right action by $\calS_{T}$. See Figure~\ref{IDandReLabel}. 
	
	{\em Relabelling} wiring diagrams can be constructed the same way given a pair of set bijections for subsets of $\calI$.
\end{example} 

\begin{figure}[h]
	\centering
	\begin{subfigure}{.5\textwidth}
		\centering
		\includegraphics[height=4.5cm]{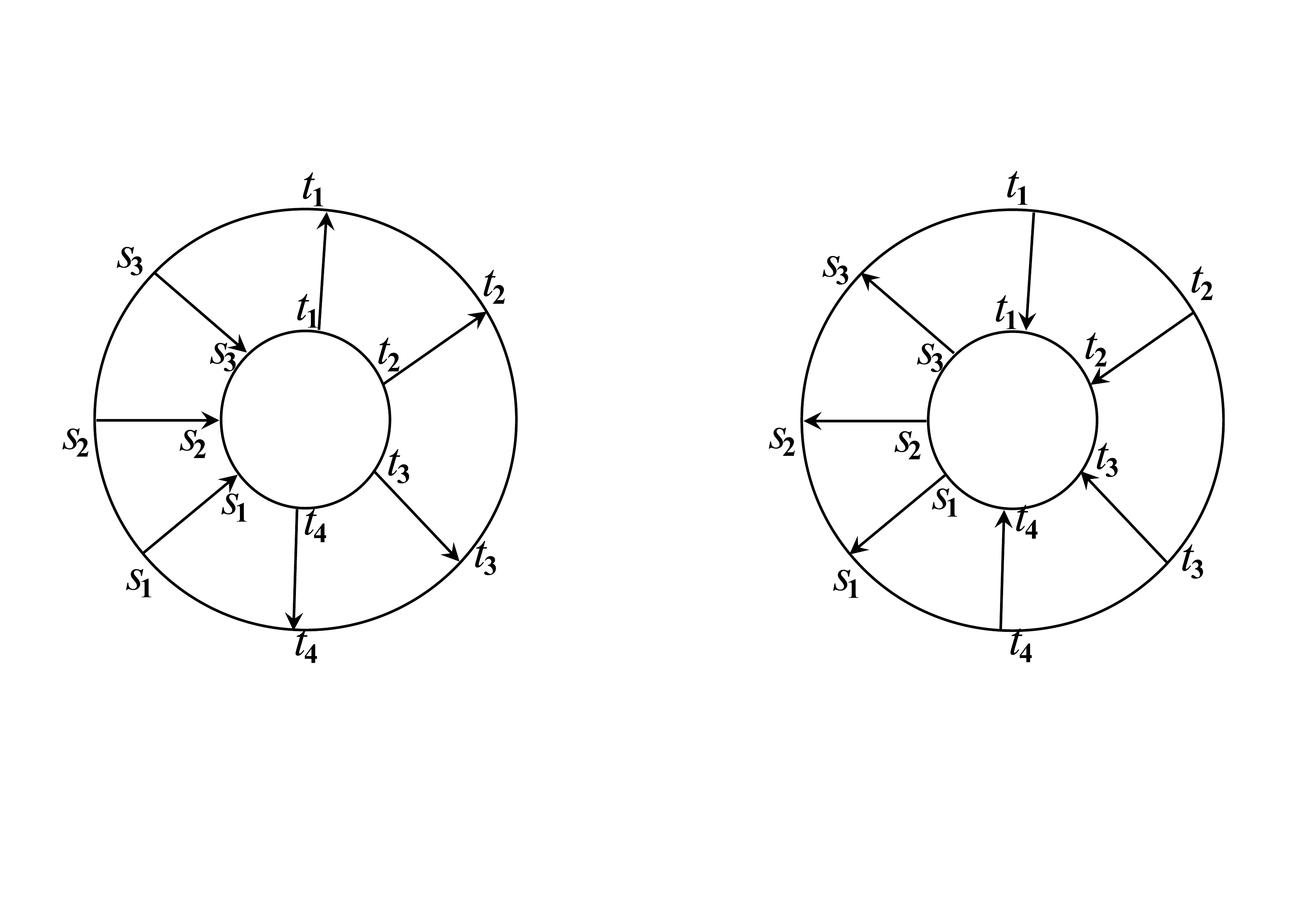}
	\end{subfigure}%
	\begin{subfigure}{.5\textwidth}
		\centering
		\includegraphics[height=4.5cm]{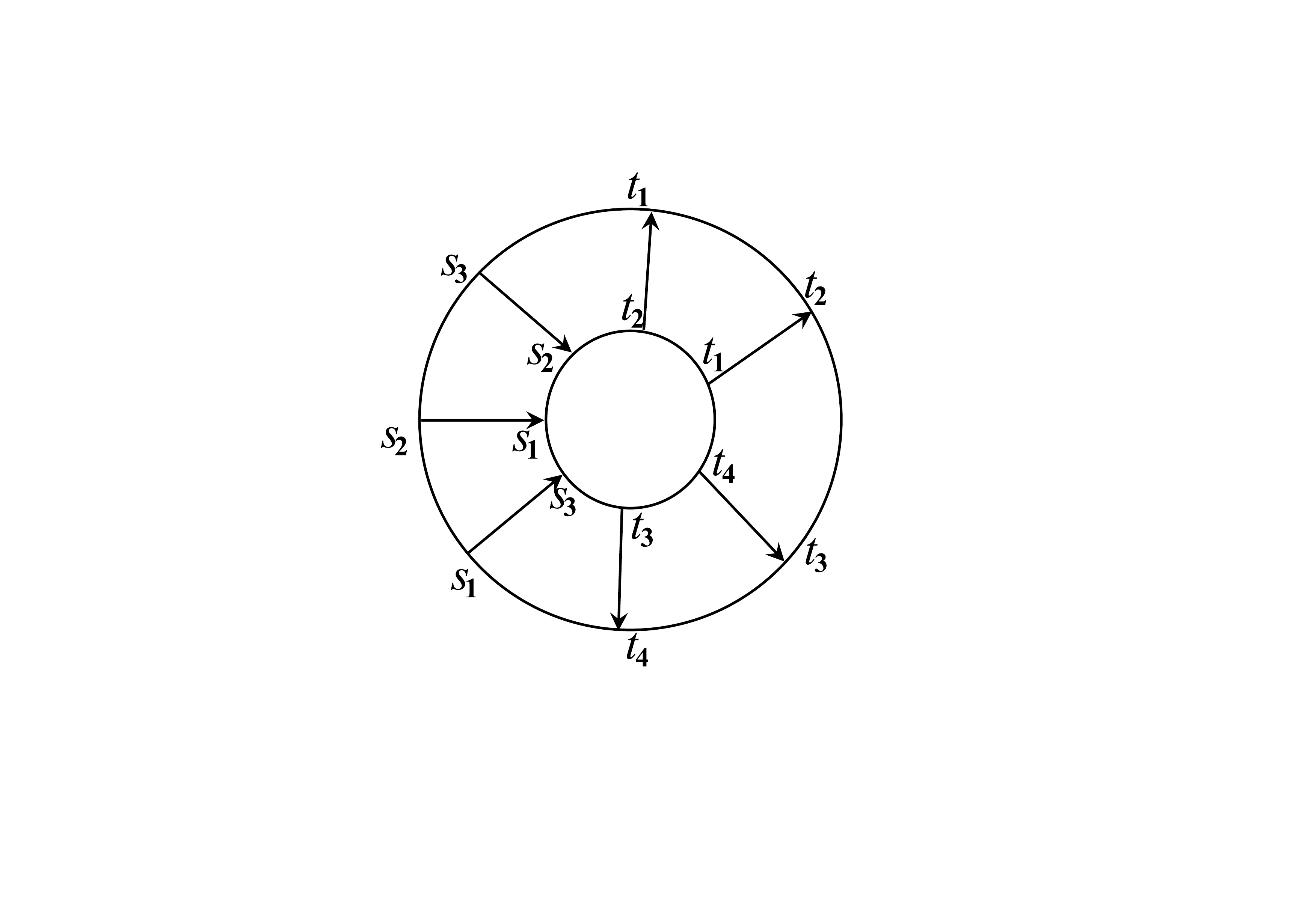}
	\end{subfigure}
	\caption{Examples of identity and relabelling wiring diagrams, respectively.}\label{IDandReLabel}
\end{figure}

\begin{remark}\label{remark: unoriented}
	In this paper we have focused on {\em oriented circuit algebras}, as they are most useful in the topological examples and applications that the authors have in mind (see Section~\ref{sec:Ex}). However, both planar algebras and circuit algebras admit many variations including oriented and un-oriented versions, coloured versions and various enrichments. 
	
	For example, while Definition~\ref{def:CA} of a circuit algebra describes $\V$ as a sequence of vector spaces and linear maps, the definition makes sense in any closed, symmetric monoidal category.  For the unoriented and coloured versions, one only need modify the definition of wiring diagrams as appropriate. Specifically, to define unoriented wiring diagrams simply drop the notion of inputs and outputs so that in a wiring diagram $D=(\calA,M,f)$ the labels sets are $\calA=\{A_0,A_1,\hdots,A_r\}$ with no $\tin/\tout$ distinction. 
\end{remark}

%%%%%%%%%%%%%%%%%%%%%%%%%%%%%%%%%%%%%%%%%%%%%%%%%%%%%%%%%%%%%%%%%%%%%%%%%%%%%%%%%

\section{Planar algebras and virtual tangles}\label{sec:Ex}

In this section we make two short detours: one to clarify the relationship between circuit algebras and planar algebras (Subsection~\ref{subsec:CAPA}), and another to present the example of virtual tangles (Subsection~\ref{subsec:VT}), where circuit algebras provide a useful and simple algebraic framework. We contrast the circuit algebra approach to virtual tangles with a planar algebra approach, which also illustrates the point of Section~\ref{subsec:CAPA}. Section~\ref{subsec:CAPA} is likely most interesting to readers already familiar with planar algebras; nothing in the latter part of the paper depends on it. Section~\ref{subsec:VT} is perhaps most interesting to topologists, though the authors believe it is quite self-contained.

\subsection{Circuit algebras and planar algebras}\label{subsec:CAPA}
The main statement of this section is that every circuit algebra is, in particular, a planar algebra, since planar tangles in discs with holes may be viewed as wiring diagrams. 
In more technical terms, there exists a pair of adjoint functors between the category of circuit algebras $\mathsf{CA}$ and the category of planar algebras $\mathsf{PA}$ \[\begin{tikzcd}\mathsf{CA} \arrow[r, hook, shift left=.5ex]& \arrow[l, shift left=.5ex]\mathsf{PA}.\end{tikzcd}\] 

To make the adjunction precise, we give a brief definition of an {\em oriented planar algebra}.
There are many variations of planar algebras in the literature, for simplicity, we work with oriented planar algebras, which are algebras over the {\em coloured operad} of oriented planar tangles (without shading), in the vein of Definition 2.3 of \cite{HPT16}.

Recall that, given a set of colours $\mathfrak{C}$, a \emph{$\mathfrak{C}$-coloured operad} $\mathsf{P}=\{\mathsf{P}(c_1,\ldots,c_r;c_0)\}$ consists of a collection of vector spaces $\mathsf{P}(c_1,\ldots,c_r;c_0)$, one for each sequence $(c_1,\ldots,c_r;c_0)$ of colours in $\mathfrak{C}$, which is equipped with an $\calS_r$-action permuting $c_1,\ldots,c_r$, together with an equivariant, associative and unital family of partial compositions
\[\begin{tikzcd}\circ_i:\mathsf{P}(c_1,\ldots,c_r;c_0)\times\mathsf{P}(d_1,\ldots,d_s;d_0)\arrow[r]& \mathsf{P}(c_1,\ldots,c_{i-1}, d_1,\ldots, d_s, c_{i+1},\ldots, c_r;c_0) \end{tikzcd},\] whenever $d_0=c_i$. For full details see \cite[Definition 1.1]{bm_resolutions}.

The \emph{operad of oriented planar tangles} $\mathsf{PT} =\{\mathsf{PT}(s_1,\ldots,s_r; s_0)\}$ is a coloured operad where $\mathsf{PT}(s_1,\ldots,s_r; s_0)$ is the vector space spanned by isotopy classes of {\em oriented planar tangles} of {\em type} $(s_1,\ldots,s_r; s_0)$. Here $s_i$ refers to a finite sequence of signs $\pm 1$, and a planar tangle of type $(s_1,\ldots,s_r; s_0)$ lives in a disc with $r$ holes $D_0 \setminus (\mathring{D}_1 \sqcup \mathring{D}_2 \sqcup  \hdots \sqcup \mathring{D}_r),$ where along each boundary circle, there is an equidistant sequence of marked points, labeled with the sign sequence $s_i$. The planar tangle is an oriented $1$-manifold $M$ embedded in such a disc with holes, such that the embedding maps $\partial M$ bijectively to the set of marked points with incoming boundary points mapping to positive marked points, and outgoing boundary points to negative marked points; cf.\ \cite[Definition  2.1]{HPT16}. 

Planar tangles are composed by the shrinking-and-gluing procedure described in the previous sections, where the gluing of the 1-manifolds must be orientation-respecting, meaning the sign sequences must match. This is a partial operadic composition which makes the set of oriented planar tangles a coloured operad.

An \emph{algebra} over an $\mathfrak{C}$-coloured operad $\mathsf{P}$ is a $\mathfrak{C}$-indexed family of vector spaces $\mathsf{A}=\{\mathsf{A}(c)\}_{c\in\mathfrak{C}}$ together with an action by $\mathsf{P}$ (\cite[Definition 1.2]{bm_resolutions}). 
An {\em oriented planar algebra} is an algebra over the operad of oriented planar tangles: a collection of vector spaces which carry actions of planar tangles, which are compatible with compositions and the symmetric group action, just like the definition of a circuit algebra based on the notion of wiring diagrams. A map of oriented planar algebras is a morphism of algebras over the operad of planar tangles -- this can be unpacked just like we did for circuit algebras. Denote the category of planar algebras by $\mathsf{PA}$.

As above, we have that $\mathsf{PT}(s_1,\ldots,s_r;s_0)$ denotes the space of oriented planar tangles of type $(s_1,\ldots,s_r;s_0)$. Let $\mathsf{WD}(s_1,\ldots,s_r;s_0)$ denote the vector space of oriented wiring diagrams where, if $s_i=(s_{i,1},...,s_{i,d_i})$, then $A_i^\tin =  \{j: s_{i,j}=1\}$ and $A_i^\tout =  \{j: s_{i,j}=-1\}$.

\begin{lemma}\label{lem:PTWD}
	For every signed sequence of integers $(s_1,\ldots, s_r;s_0)$, the space of oriented planar tangles $\mathsf{PT}(s_1,\ldots,s_r;s_0)$ is a subspace of $\mathsf{WD}(s_1,\ldots,s_r;s_0)$. 
\end{lemma}

\begin{proof}
	Given a planar tangle of type $(s_1,\ldots, s_r;s_0)$, forget the embedding information on the interior of $M$ to obtain a wiring diagram, where only $\partial M$ is identified with the label sets.
\end{proof}

Now assume that $\mathsf{P}$ and $\mathsf{Q}$ are $\mathfrak C$-coloured operads with inclusions $\mathsf{Q}(c_1,\ldots,c_r;c_0) \subseteq \mathsf{P}(c_1,\ldots,c_r;c_0)$ for each $(c_1,\ldots,c_r;c_0)$. Recall that such data defines a $\mathfrak C$-coloured \emph{sub-operad} of $\mathsf{P}$, if the restriction of the symmetric group actions and $\circ_i$ partial compositions of $\mathsf{P}$ agrees with the operad structure of $\mathsf{Q}$.

\begin{prop} \label{prop:suboperad}
	The operad of oriented planar tangles is a sub-operad of the operad of wiring diagrams. 
\end{prop}

\begin{proof} 
	The symmetric group $\calS_r$ acts on $\mathsf{WD}$ by permuting the indices of the input sets (internal discs), and therefore restricts to $\mathsf{PT}(s_1,\ldots,s_r;s_0)\subseteq \mathsf{WD}(s_1,\ldots,s_r;s_0)$. The $\circ_i$ partial composition of planar tangles, as described in Definition 2.1~\cite{HPT16}, is precisely the $\circ_i$ partial composition of wiring diagrams when restricted to the subspace of planar tangles. 
\end{proof}

Let $\mathsf{Q}$ be a sub-operad of $\mathsf{P}$, and let $\text{Alg}(\mathsf{Q})$ and $\text{Alg}(\mathsf{P})$ denote the categories of algebras over $Q$ and $P$, respectively. Then the inclusion $\phi:\mathsf{Q}\rightarrow \mathsf{P}$ induces an adjunction \[\begin{tikzcd}\text{Alg}(\mathsf{Q})\arrow[r, shift left=.5ex, "\phi_!"] & \arrow[l,shift left=.5ex,"\phi^*"] \text{Alg}(\mathsf{P}). \end{tikzcd}\] For the further details, see just below Definition 1.2 in \cite{bm_resolutions}. The functor $\phi^*$ is quite intuitive to construct, as follows. An algebra over $P$ is a structure that carries an action by the elements of $P$; an algebra over $Q$ is a structure that carries an action by elements of $Q$; both of which are compatible with compositions, and symmetric group actions. Since every element of $Q$ is in particular an element of $P$, an algebra over $P$ is automatically an algebra over $Q$. The left adjoint exists for formal reasons.

\begin{prop} \label{prop: adjunction PA and CA}
	There is an adjunction \[\begin{tikzcd}\mathsf{PA}\arrow[r, shift left=.5ex, "\phi_!"] & \arrow[l,shift left=.5ex,"\phi^*"] \mathsf{CA} \end{tikzcd}\] between the category of oriented planar algebras and the category of oriented circuit algebras. 
\end{prop}

\begin{proof} 
	A direct consequence of Proposition~\ref{prop:suboperad}, given the discussion above.
\end{proof} 

\begin{remark}
	A similar argument will hold for the reader's favourite variation of planar and circuit algebras (unoriented, coloured, enriched) by appropriate modifications to the underlying operads. 
\end{remark}

\subsection{Virtual tangles}\label{subsec:VT}
Virtual tangles -- a generalization of classical tangles -- were the motivating example for the definition of circuit algebras, and so far they have been their main area of application \cite{Brochier, polyak2010alexanderconway}.  

An oriented (classical) tangle is a smooth embedding of an oriented $1$-manifold M into a $3$-dimensional ball $M\hookrightarrow B^3$, such that the boundary is mapped to the boundary of the ball: $\partial M \hookrightarrow S^2=\delta B^3$. Such embeddings are considered up to boundary-preserving ambient isotopy. In knot theory, tangles are often studied via their Reidemeister theory: project the ball onto a disc to obtain a {\em tangle diagram} where the image of $M$ has only transverse double-points in the interior, called {\em crossings}, and no double points on the boundary. Ambient isotopy is captured by three diagrammatic relations, called the Reidemeister $1$, $2$ and $3$ moves. 

\begin{figure}[h]
	\includegraphics[height=4cm]{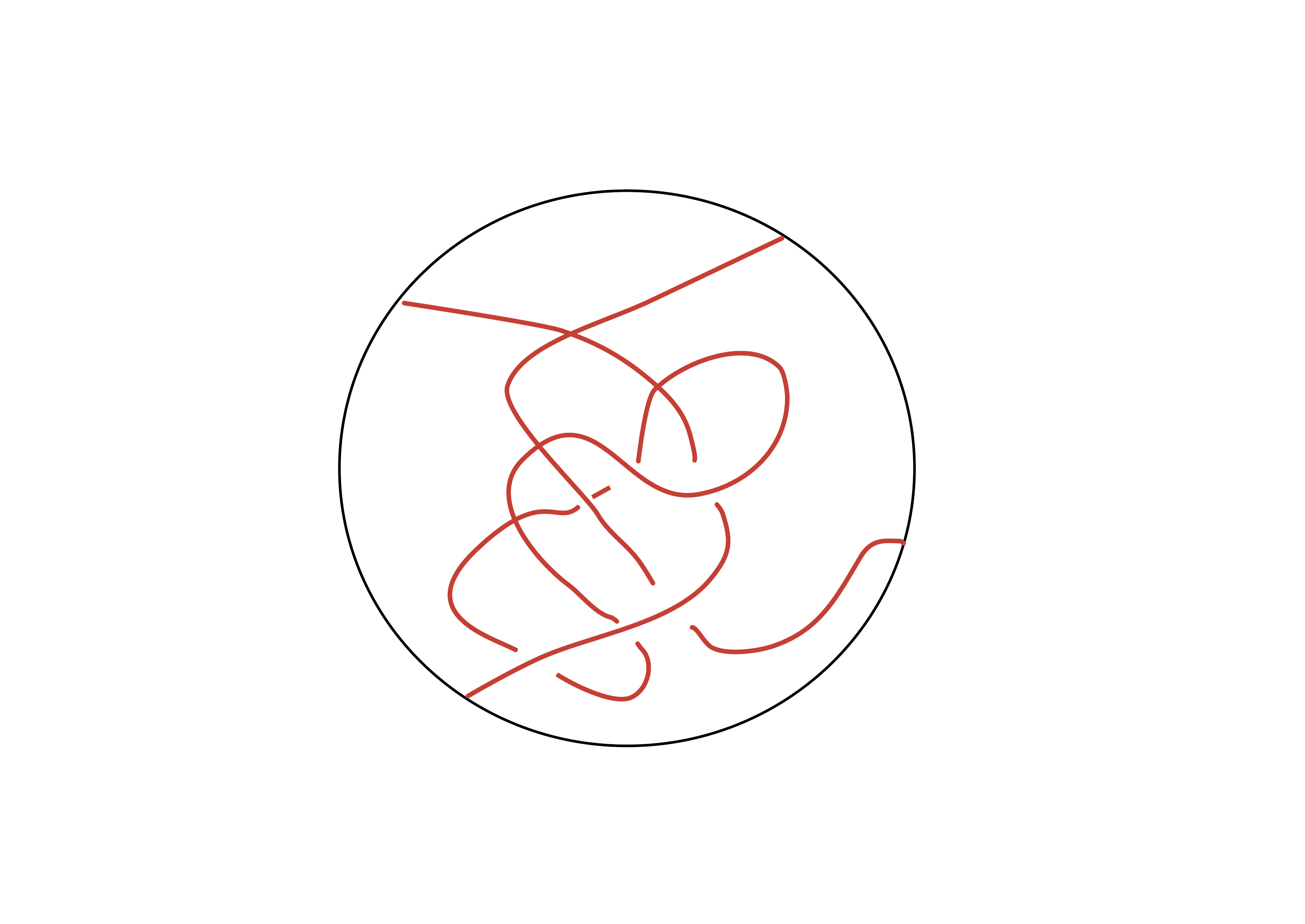}
	\caption{An example of a virtual tangle diagram.} 
\end{figure} 

This leads naturally to a presentation of tangles as a planar algebra \cite[Section 5]{BN:KHTangles}, generated\footnote{Generation in a planar algebra means, as one would expect, all possible applications of planar connection diagrams.} by an overcrossing $\overcrossing$ and an undercrossing $\undercrossing$, modulo the Reidemeister $1$, $2$ and $3$ moves shown in Figure~\ref{fig:Reidemeister} (\cite{BND:WKO2}). To summarize, tangles form a planar algebra $\mathcal{T}=PA\langle\overcrossing, \undercrossing~|~R1, R2, R3\rangle$. 

\begin{figure}[h]
	\centering
	\begin{subfigure}{.5\textwidth}
		\centering
		\includegraphics[height=4cm]{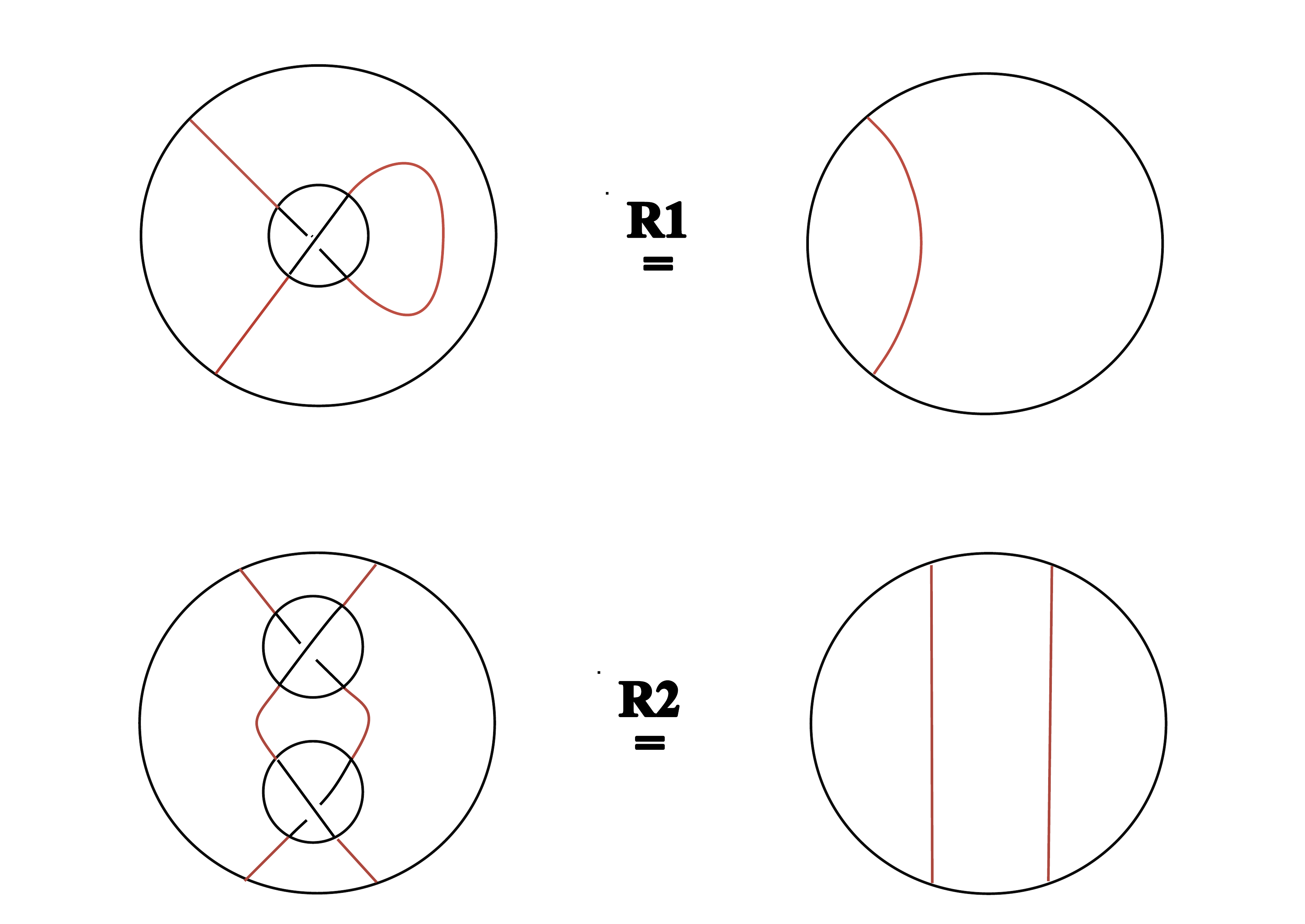}
	\end{subfigure}%
	\begin{subfigure}{.5\textwidth}
		\centering
		\includegraphics[height=4cm]{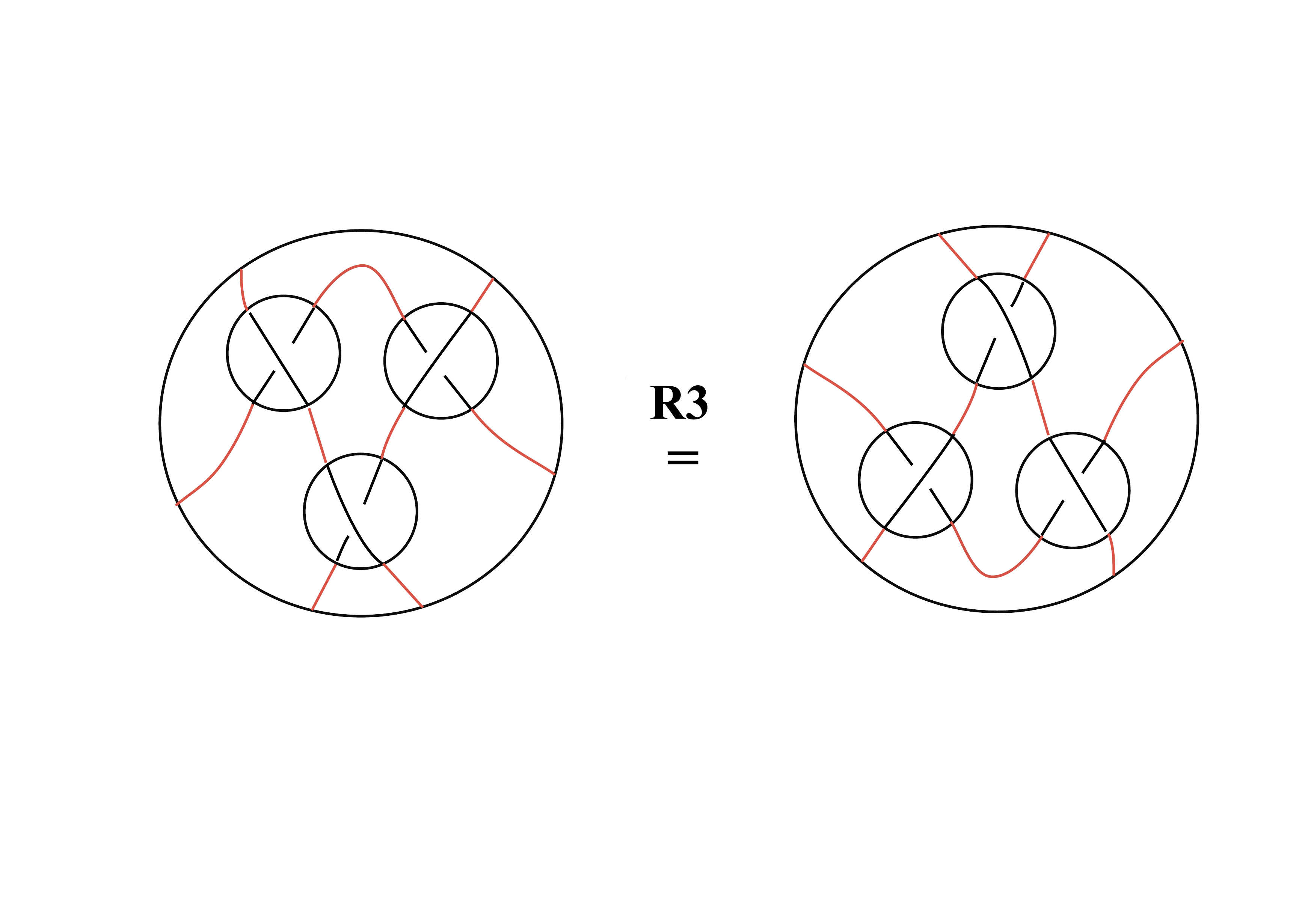}
	\end{subfigure}
	\caption{The classical Reidemeister moves, presented as planar algebra relations between the crossings (generators). The relations are imposed in all possible (consistent) strand orientations.}\label{fig:Reidemeister}
\end{figure}

Virtual tangles are a generalization of tangles to embeddings into thickened surfaces, rather than into $B^3$ -- see \cite{Kup} for a proof of this statement in the case of links. Virtual tangles are topologically interesting as a broader family of tangled objects, but they also possess interesting algebraic and combinatorial properties. As one example, their finite type invariants are conjecturally deeply connected with quantizations of Lie bialgebras, see \cite[Introduction]{BND:WKO1} for a brief overview.\footnote{This connection is further strengthened by the equivalence of circuit algebras and wheeled props in Theorem~\ref{thm:main}, as wheeled props play a role in formality theorems for Lie bialgebras \cite{MR3451536}. The direct comparison between circuit algebras and these wheeled props is part of the authors' motivation for this project.} 

Like classical tangles, virtual tangles can be described as virtual tangle diagrams modulo Reidemeister moves, and can be presented as a planar algebra in which one adds a {\em virtual crossing} as an additional generator $\virtualcrossing$ as well as additional virtual Reidemeister relations, shown in Figure~\ref{fig:VReidemeister}, which describe how virtual crossings interact with each other and with classical crossings. 

\begin{figure}[h]
	\centering
	\begin{subfigure}{.5\textwidth}
		\centering
		\includegraphics[height=4cm]{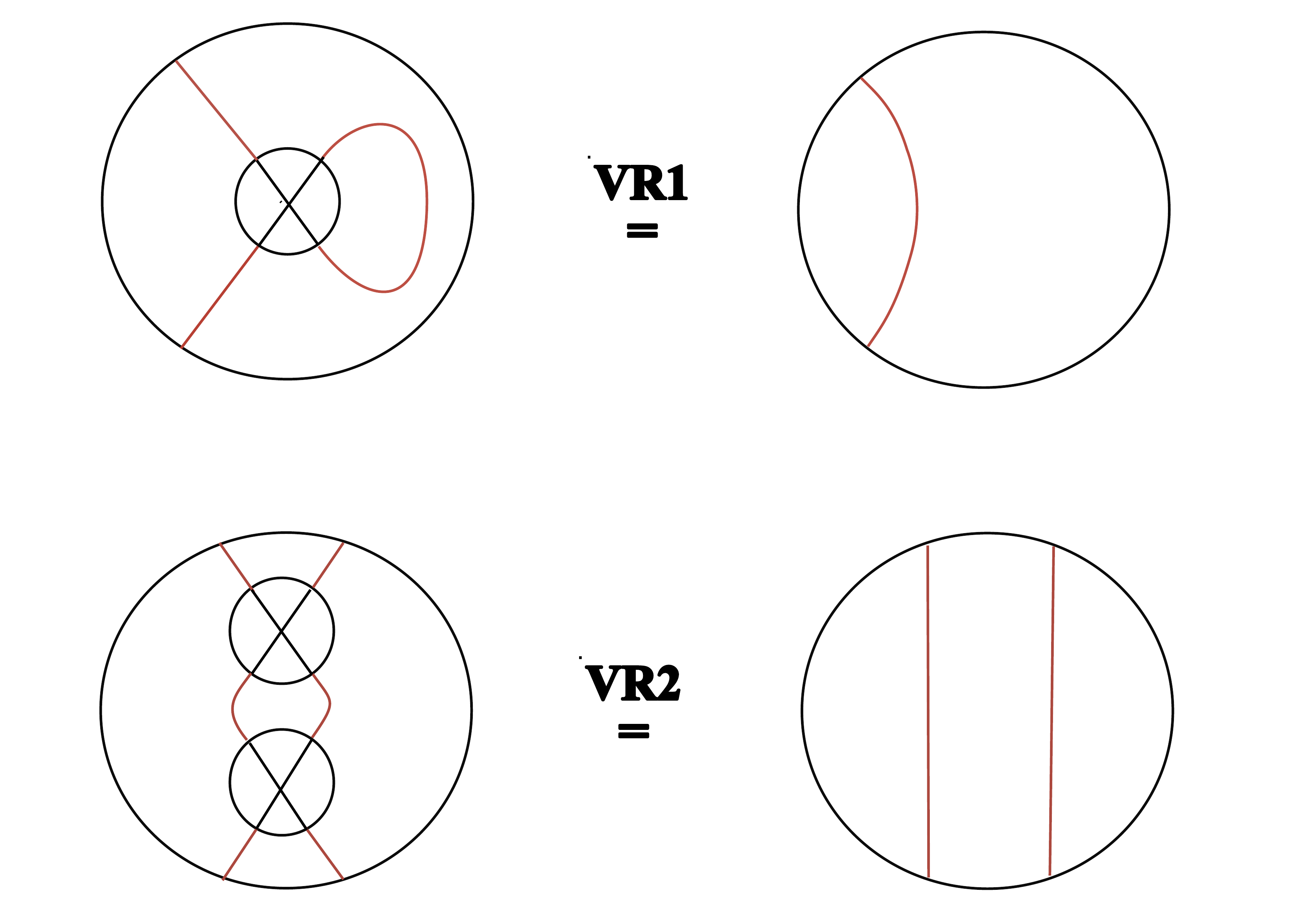}
	\end{subfigure}%
	\begin{subfigure}{.5\textwidth}
		\centering
		\includegraphics[height=4.5cm]{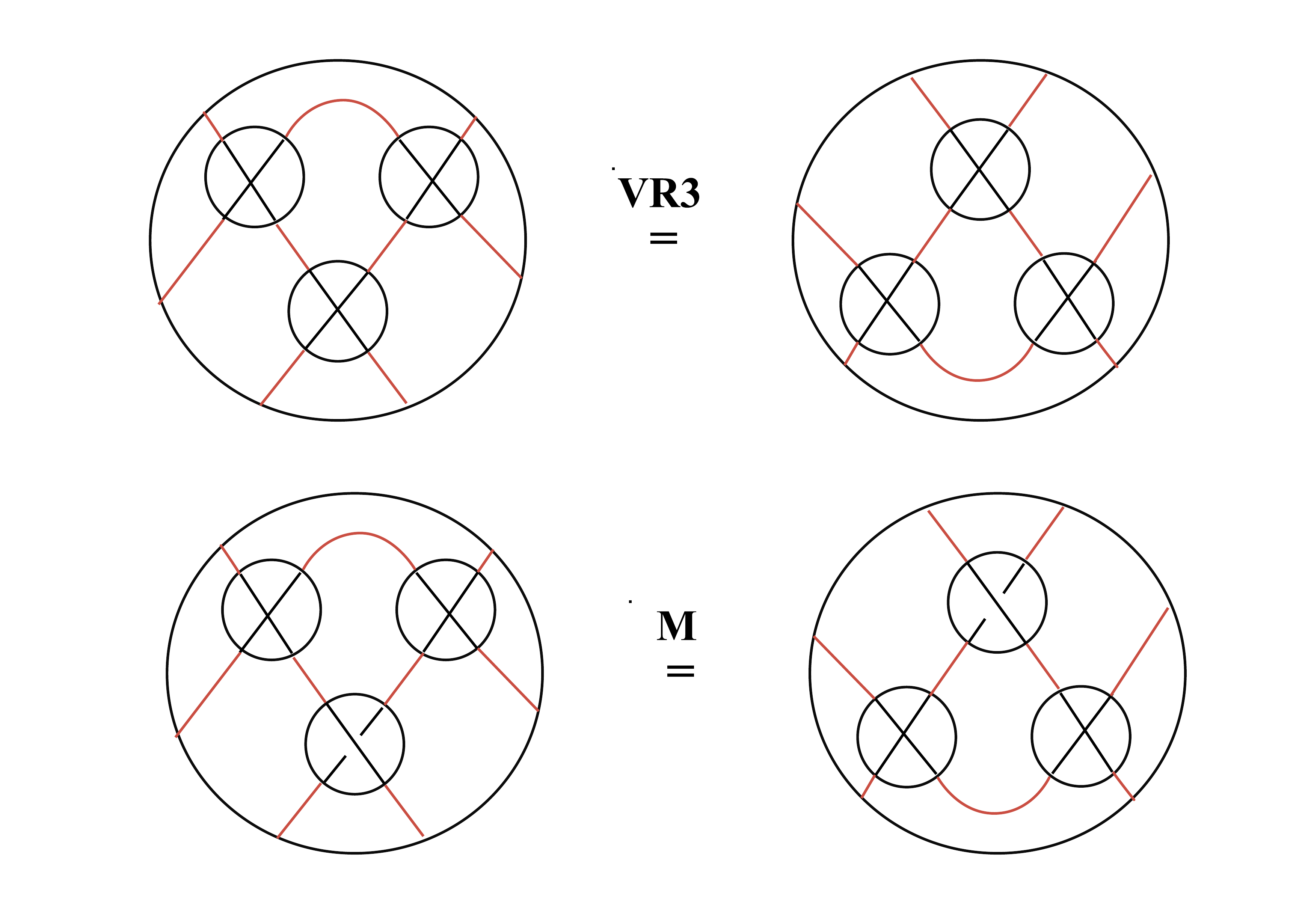}
	\end{subfigure}
	\caption{The "virtual" and "mixed" Reidemeister moves, as planar algebra relations. The relations are imposed in all possible (consistent) strand orientations.}\label{fig:VReidemeister}
\end{figure}

%One conventional [REFERENCES] presentation of virtual tangles is as a planar algebra similar to tangles, but with an additional generator $\virtualcrossing$ called a {\em virtual crossing}. This necessitates the introduction of a number of new relations, as in Figure~\ref{fig:VReidemeister} to describe how the virtual crossings interact with each other and the classical crossings.

It is a basic fact of virtual knot theory called the {\em detour move} (see for example \cite[Section 2]{DK:Virtual}) that any purely virtual part of a strand of a virtual tangle diagram (i.e. a strand that only intersects others in virtual crossings) can be re-routed in {\em any other purely virtual way}. This suggests that the virtual crossing, morally speaking, isn't a ``true generator'' of virtual tangles, but merely a structural, diagrammatic artifact. 

This motivates the description of virtual tangles as a circuit algebra \cite[Section 3]{BND:WKO2}:  the generators are simply two crossings $\{\overcrossing, \undercrossing\}$, and the relations are the ordinary Reidemeister moves $\{R1, R2, R3\}$: \[\mathsf{vT}=\mathsf{CA}\left<\overcrossing, \undercrossing \mid R1, R2, R3\right>.\] In this description, ``virtual crossings'' only exist in the pictorial representation of wiring diagrams, however, as wiring diagrams are fundamentally combinatorial objects (Remark~\ref{rmk:Matching} and Lemma~\ref{lemma: graphs are wiring diagrams}), it is understood that they don't hold any mathematical meaning. The {\em detour move} is so tautological that it doesn't even make sense as a statement in a circuit algebra context. 

This also gives an illustrative example of the relationship between circuit algebras and planar algebras: virtual tangles naturally form a circuit algebra, and all circuit algebras are also planar algebras by Lemma~\ref{lem:PTWD} and Proposition~\ref{prop: adjunction PA and CA}, hence virtual tangles also form a planar algebra. However, we see above that the circuit algebra description in terms of generators and relations is simpler and in the authors' opinion more elegant. Classical tangles, on the other hand, naturally form a planar algebra, with a simple, elegant description, and do not admit a (reasonable) circuit algebra structure.

%%%%%%%%%%%%%%%%%%%%%%%%%%%%%%%%%%%%%%%%%%%%%%%%%%%

\section{Wheeled Props}\label{sec:WP}

A (linear) prop\footnote{Some authors prefer to capitalise the word PROP to emphasise that prop refers to ``PROduct and Permutation category."} is a strict symmetric tensor category in which the {\em monoid of objects} is freely generated by a single object. 
%In other words, every object of a prop is of the form $x^{\otimes n}$, $n\geq 0$, for some fixed $x$. 
Props are often used to ``encode'' a class of algebraic structures. For example, there exists a prop $\mathsf{LieB}$ with the property that strict symmetric monoidal functors from $\mathsf{LieB}$ to the category of vector spaces are in one-to-one correspondence with Lie bialgebras \cite[Definition 2.1]{MR3451536}. 

Wheeled props are an extension of props which ``encode'' algebraic structures with a notion of \emph{trace}. They arise naturally in 
geometry, deformation theory and theoretical physics. For example, 
in the Batalin-Vilkovisky quantization formalism, formal germs of SP-manifolds are in one-to-one correspondence with representations of a certain wheeled prop \cite[Theorem 3.4.3]{mms}. More detailed examples of wheeled props are given in Section~\ref{section:examples of wheeled props}.

In this section we give a monadic definition for wheeled props (Definition~\ref{def:monad} and ~\ref{def:WP}) using the notion of {\em oriented graphs} and the operation of {\em graph substitution}. This is the most intuitive route -- in the authors' opinion -- to describing the relationship between wheeled props and circuit algebras in Section~\ref{sec:Eq}. An alternate, axiomatic, definition of wheeled props is presented in Section~\ref{section: categorical description}.

\subsection{Oriented graphs} For this purpose, a \emph{graph}\footnote{Also referred to as Borisov-Manin graphs. Sometimes defined as a quadruple of vertices, flags, attachment map and an involution. Our generalization allows for vertex-less loops.} is a {\em graph with open edges}: an edge may be adjacent to one, two or no vertices. Graphs may have free-floating loops with no vertices, as in Figure~\ref{fig:GraphExample}. There are several definitions of such graphs in the literature; in Definitions~\ref{def:Graph} and \ref{def: oriented graph} we adapt a combinatorial definition for {\em generalized graphs} from \cite{yj15}, which gives a direct comparison to wiring diagrams.

For a user-friendly preview before the technical definition, a graph consists of a collection $\calF$ of {\em flags}, or {\em half edges}, along with an (ordered) partition on $\calF$. The sets in the partition are the {\em vertices} of the graph, with the exception of one set, set aside for flags that are part of a free-floating edge or loop, and hence not incident to a vertex. This set in the partition is called the {\em exceptional cell}, denoted by $v_{\epsilon}$. 

An involution on flags $\iota: \calF\rightarrow \calF$ glues the flags together to form {\em edges}. Two flags $a, a' \in v_{\epsilon}$ with the property that $\iota(a)=a'$ form a {\em free-floating loop} detached from any vertex. In addition, there is a fixed point free involution $\pi$ on the $\iota$-fixed points of $v_{\epsilon}$, which assembles these flags into {\em free-floating edges} not attached to a vertex. %\footnote{The third author like to refer to these floating edges and loops as exceptional edges and exceptional loops, to remind us that they live in the exceptional cell $v_{\epsilon}$.} 
To summarise: 

\begin{definition}\label{def:Graph}
	Fix a countable alphabet $\calI$. A \emph{labelled graph} $G$ is a finite set $\calF=\calF(G)$, called the set of {\em flags}, or \emph{half-edges}, together with 
	\begin{itemize}
		\item a finite ordered partition $\calF=(\coprod_{\alpha\in V(G)}v_{\alpha}) \sqcup v_\epsilon$,
		\item an involution $\iota:\calF\to \calF$ with the property $\iota(v_\epsilon) = v_\epsilon$, 
		\item a fixed point free involution $\pi$ on the $\iota$-fixed points in $v_\epsilon$, and 
		\item a {\em labelling} function $\lambda: \coprod_{v_{\alpha}\in V(G)} v_{\alpha} \to \calI$, injective on each $v_\alpha$, 
		\item an injective {\em boundary labelling} function  $\beta: \partial G \to \calI$, where $\partial G \subseteq \calF$ is the set of $\iota$-fixed flags.
	\end{itemize}
\end{definition}

The vertices $v_\alpha$, and the exceptional cell $v_\epsilon$, are subsets of flags, and as such they can be empty.  If a vertex $v_{\alpha}$ is empty, it is an {\em isolated vertex} with no incident flags. If $v_{\epsilon}$ is empty, then the graph has no free-floating edges or loops. The set $V(G)$ is called the vertex set. In this paper, we assume for simplicity that vertices are numbered, i.e. there is a bijection $V(G)\rightarrow \{1,2,...,r\}$.  From now on we will refer to labelled graphs simply as \emph{graphs}.

\begin{example}
	The graph in Figure~\ref{fig:GraphExample} has flags $$\calF=\{1, 2, 3, 4, 5, 6, 7, 8, 9, 10,11, 12,13\}$$ partitioned into vertices: $v_1=\{1,2,3,4,5\}$ and $v_2=\{6,7, 8,9\}$ with exceptional cell $v_{\epsilon}=\{10,11,12,13\}$. The involution $\iota$ acts on $\calF$ with $\iota(5)=7$, $\iota(4)=6$ and $\iota(12)=13$, the last pair making up the floating circle.  The $\iota$-fixed points are $1$, $2$, $3$, $8$, $9$, $10$ and $11$. The fixed point free involution $\pi$ acts on the $\iota$-fixed points in $v_{\epsilon}$ and sets $\pi(10)=11$. The labelling $\lambda: v_1\cup v_2\rightarrow \calI$  labels each flag $k$ for $k=1,...,9$ by $a_k$. In this example, $\partial G= \{1, 2, 3, 8, 9, 10, 11\}$, and the boundary labelling assigns the label $\beta(k)=i_k$ to each $k \in \partial G$. Note that the flags in the floating loop remain un-labelled. 
	\begin{figure}\label{figure:graph with labels}
		\begin{centering}
			\includegraphics[height=4.5cm]{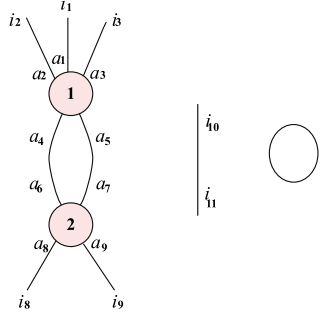}
		\end{centering}
		\caption{A labelled graph with $13$ flags partitioned into vertices: $v_1=\{1,2,3,4,5\}$ and $v_2=\{6,7, 8,9\}$ with exceptional cell $v_{\epsilon}=\{10,11,12,13\}$. } \label{fig:GraphExample} 
	\end{figure}
\end{example}

\begin{definition}\label{def: oriented graph}
	An {\em oriented graph} is a graph $G$ with an orientation function \linebreak $\delta: \calF_{0} \to \{-1,1\}$, where $\calF_{0}=\coprod_{\alpha\in V(G)}v_{\alpha} \sqcup \{\varphi \in v_{\epsilon}: \iota(\varphi)= \varphi\},$ such that $\delta(\iota x)=-\delta(x)$ whenever $\iota x\not=x$, and $\delta(\pi x)=-\delta(x)$ whenever $\pi$ is defined. 
	We write $$\tin(v_{\alpha})=\lambda(\delta |_{v_\alpha})^{-1}(1), \text{ and } \tout(v_{\alpha}) = \lambda(\delta |_{v_\alpha})^{-1}(-1),$$ 
	$$\tin(G)=\beta(\delta |_{\partial G})^{-1}(1), \text{ and } \tout(G) = \beta(\delta |_{\partial G})^{-1}(-1),$$ 
	where $\delta |_{v_\alpha}$ and $\delta |_{\partial G}$ are the restrictions of $\delta$ to a vertex $v_\alpha$ and the boundary $\partial G$, respectively. For an oriented graph we only require the labelling function $\lambda$ to be injective on the sets $(\delta |_{v_\alpha})^{-1}(1)$ and $(\delta |_{v_\alpha})^{-1}(-1)$ as opposed to all of $v_\alpha$. Similarly, the boundary labelling $\beta$ is only required to be injective on $(\delta |_{\partial G})^{-1}(1)$ and $(\delta |_{\partial G})^{-1}(-1)$.
\end{definition}

In pictures we indicate the direction function by drawing an arrow from a negative flag $a$ to its positive pair $\iota (a)$; or  from a positive flag $a$ to its negative pair $\pi(a)$ for free edges. Free loops do not have directions. See Figure~\ref{fig:Smod}.

For non-exceptional $\iota$-fixed flags  -- that is, half-edges attached to a vertex -- negative flags are drawn as outgoing\footnote{Outgoing edges are also commonly called outputs, and incoming edges are called inputs. We use the outgoing/incoming terminology to avoid confusion with circuit algebra inputs and outputs.} and positive flags are drawn as incoming. As an example, the oriented corolla in Figure~\ref{fig:corolla} has incoming, or positive, flags with (boundary) labels $i_1, i_2$ and $i_3$, and outgoing (negative) flags labelled $j_1$ and $j_2$. Figure~\ref{fig:directed graph with boundary} shows another example of a directed graph. Boundary labels are shown with incoming flags labelled $i_\alpha$ and outgoing flags labelled $j_\beta$ for $\alpha= 1,2,3$ and $\beta=1,2$. Vertex labels (i.e. the $\lambda$ labelling) are suppressed. Note that any complete edge has a well-defined beginning and end. The free loop has no direction. 

The labels of the incoming and outgoing flags adjacent to a vertex $v_{\alpha}$ are called the \emph{neighbourhood} of the vertex\footnote{Sometimes the vertex is said to be \emph{labelled} by the input and output flags contained in the neighbourhood, and the graph is {\em labelled} by its boundary flags.}:
$$\text{nbh}(v)=(\tin(v);\tout(v))=(\lambda(\delta|_{v}^{-1})(1);\lambda(\delta|_{v}^{-1})(-1)).$$ 
By a slight abuse of notation (using $\partial G$ to denote the set of boundary flags or their labels, depending on context) for oriented graphs we also write:
$$\partial G=(\tin(G);\tout(G))=(\beta(\delta|_{\partial G}^{-1})(1);\beta(\delta|_{\partial G}^{-1})(-1)).$$

\begin{figure}
	\centering
	\begin{subfigure}{.5\textwidth}
		\centering
		\includegraphics[height=4.5cm]{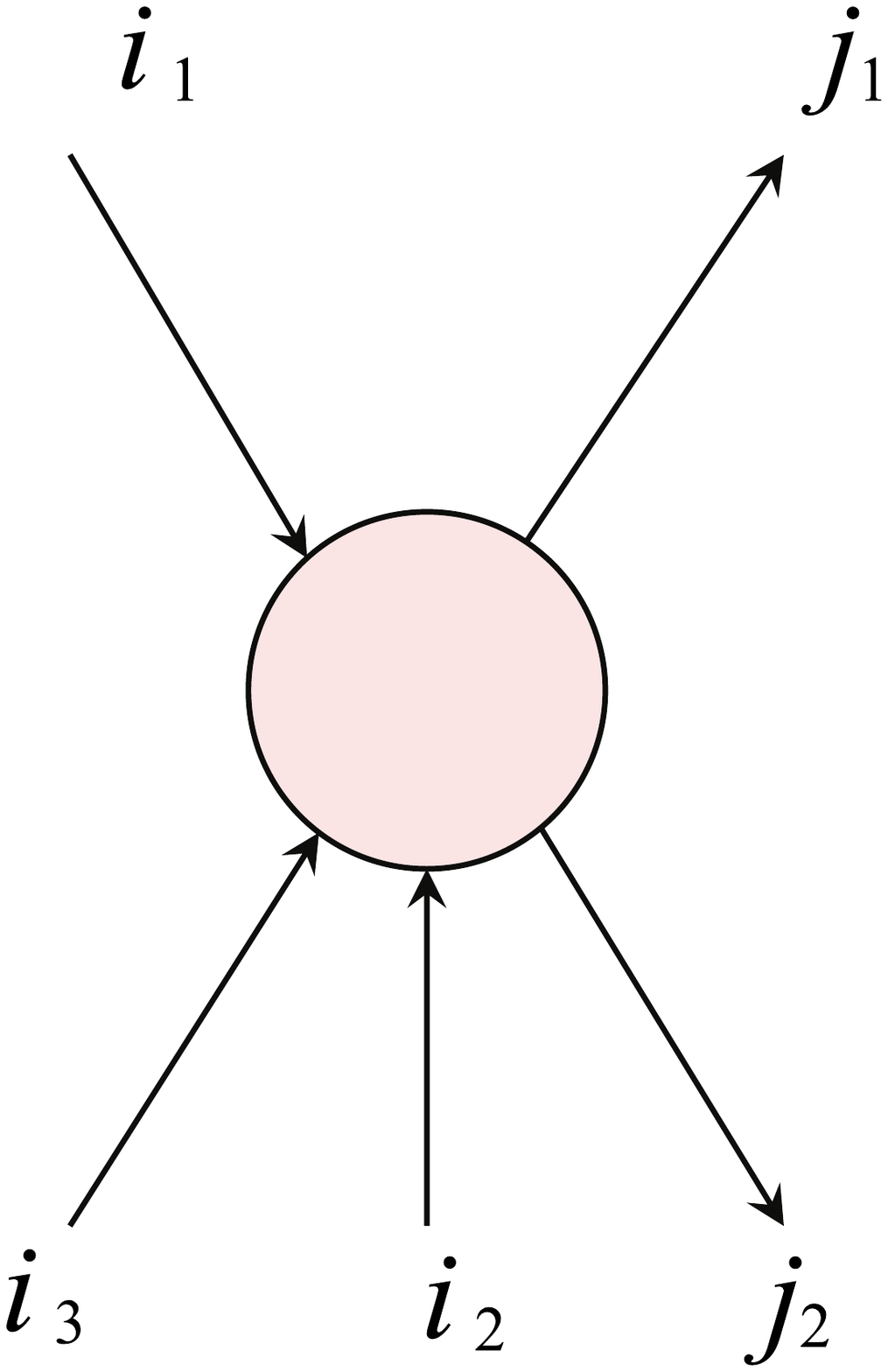}
		\caption{A corolla $C_{(I;J)}$.}\label{fig:corolla}
	\end{subfigure}%
	\begin{subfigure}{.5\textwidth}
		\centering
		\includegraphics[height=4.5cm]{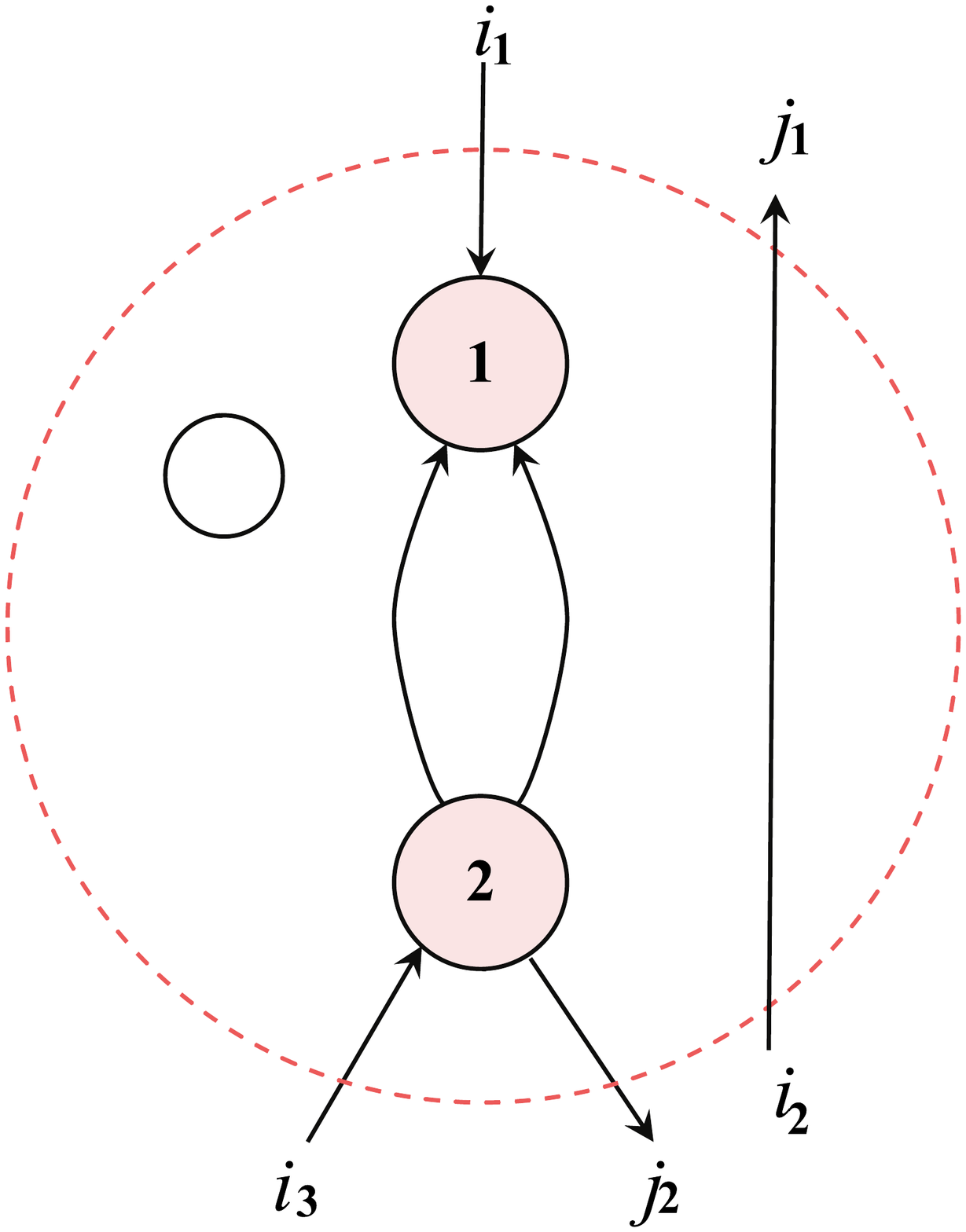}
		\caption{ A directed graph with inputs $\{i_1,i_2,i_3\}$ and outputs $\{j_1,j_2\}$. The dashed red circle indicates the boundary of $G$, $\partial{G}$. }\label{fig:directed graph with boundary}
	\end{subfigure}
	\caption{Two depictions of oriented graphs with incoming labels $\{i_1,i_2,i_3\}$ and outgoing labels $\{j_1,j_2\}$.}\label{fig:Smod}
\end{figure}

\begin{definition}\label{def:graph isomorphism}
	We say that two oriented labelled graphs $G_1$ and $G_2$ are \emph{isomorphic} if there is a bijection of flags $\phi:\calF(G_1)\rightarrow \calF(G_2)$, which preserves the ordered partitions, the involutions, the orientation function, and the labelling.
\end{definition}

The operation of {\em graph substitution} for oriented\footnote{The notion also exists for non-oriented graphs, by simply omitting the condition that orientations match.} graphs parallels the composition of wiring diagrams in a circuit algebra. Intuitively, graph substitution ``glues'' a graph $H_v$ into a vertex $v$ of a graph $G$ in such a way that $\partial H_v=\text{nbh}(v)$. The result is a new graph $G(H_v)$. While the intuition is clear, writing down the resulting graph in terms of involutions is tedious due to the possible creation of floating loops: see Figure~\ref{fig:graphsubloop} for an example. Below we give an intuitive definition and for full details refer the reader to \cite[Chapter 5]{yj15}. 

\begin{figure}%[h]
	\begin{centering}
		\includegraphics[height=4.5cm]{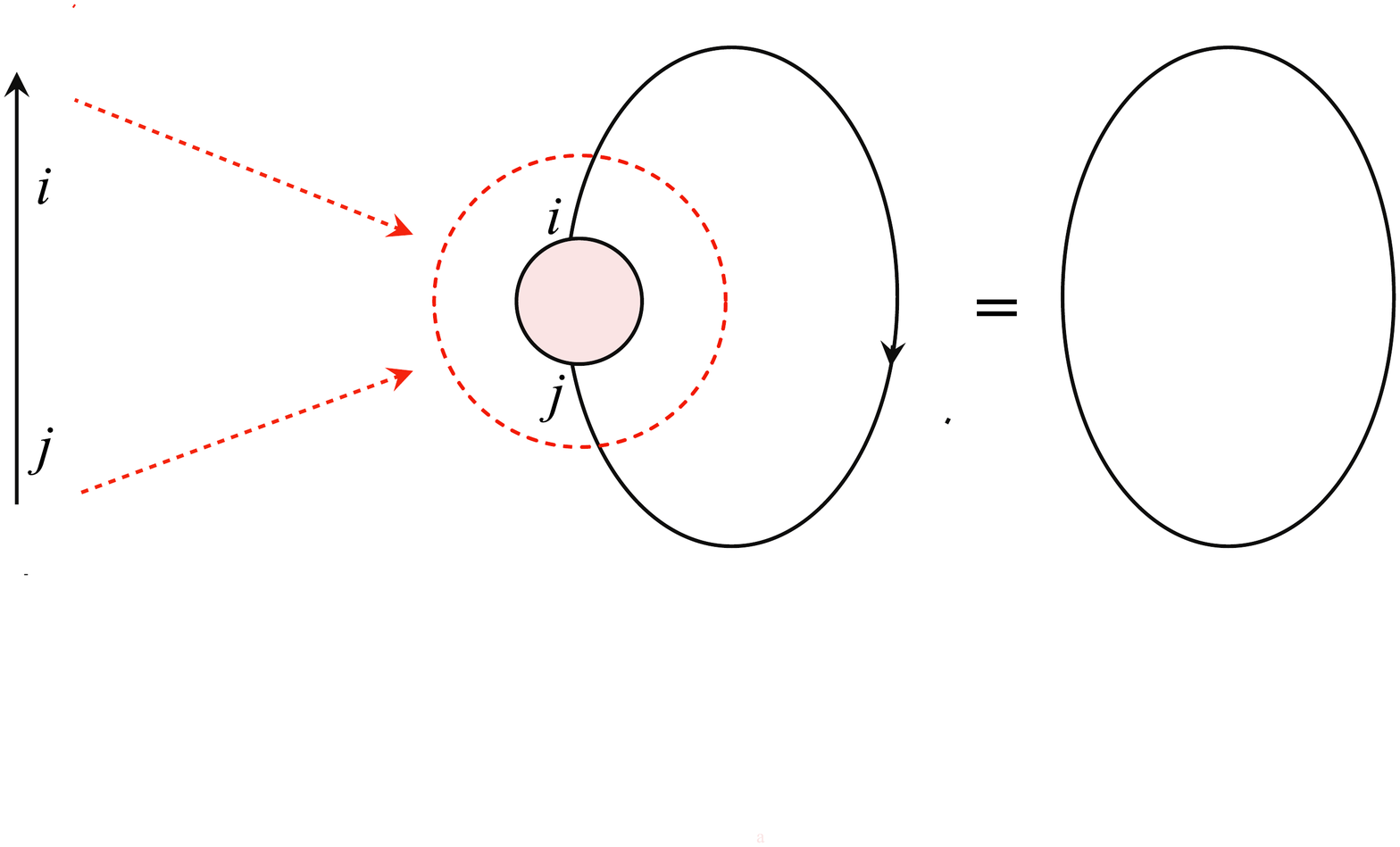}
	\end{centering}
	\caption{An example of a graph substitution which creates a free-floating loop.}\label{fig:graphsubloop}
\end{figure}

\begin{definition}\label{def: graph substitution}
	Let $G$ be an oriented graph, $v \in V(G)$ a vertex, and $H_{v}$ an oriented graph with $\partial H_v=(\tin(H_v);\tout(H_{v})) =(\tin(v);\tout(v))=\text{nbh}(v)$.   Define the \emph{graph substitution}
	$G(H_v)$
	as the graph obtained by
	\begin{itemize}
		\item
		replacing the vertex $v \in V(G)$ with the graph $H_{v}$, and
		\item
		identifying each leg (boundary flag) of $H_{v}$ with the flag of $v$ with the same label and same orientation.
	\end{itemize}
\end{definition}

The boundary of the graph $G(H_v)$ is identified with the boundary of $G$.  Moreover, there is a canonical identification of vertex sets $V(G(H_v))= (V(G)\setminus \{v\}) \sqcup V(H_v)$. The ordering (numbering) of the vertices of $G(H_v)$ is as follows: first follow the ordering of $V(G)$ before $v$, then the ordering of $V(H_v)$, then the ordering of $V(G)$ after $v$.

Graph substitution is \emph{associative} and \emph{unital}, see Theorem 5.32 and Lemma 5.31 in 
\cite{yj15}. The unit for substitution into a given vertex $v$ of a graph $G$ is a {\em corolla}: a single vertex with the 
same number of incoming and outgoing legs as $v$, with the same labelling. An example of a corolla is shown in Figure~\ref{fig:corolla}. 	

Associativity implies that graph substitution can be carried out {\em en masse}, given a substitutable 
graph $H_v$ for each of the vertices of a graph $G$. This operation is denoted $G(\{H_v\}_{v\in V(G)})
$, or $G(\{H_v\})$ for short. The boundary flags of $G(\{H_v\})$ are identified with the boundary flags of $G$, and there is a 
canonical identification

\begin{equation}\label{eq:GraphSubst}
V\left(G(\{H_v\})\right) = \coprod_{v\in V(G)} V(H_v).
\end{equation}  

\begin{example}\label{example: permuted corolla}
	Graph substitution captures important structural changes in graphs. A key example is the \emph{relabelling} substitution: one can change the boundary labelling of a graph $G$ arbitrarily, by substituting $G$ into a corolla whose vertex labels agree with the boundary labels of $G$, but whose boundary labels are the desired new labels. The example shown in Figure~\ref{figure:graph relabelling} is a graph substitution which implements a permutation of boundary labels.  
	
	\begin{figure}%[h]
		\begin{centering}
			\includegraphics[height=5cm, width=7cm]{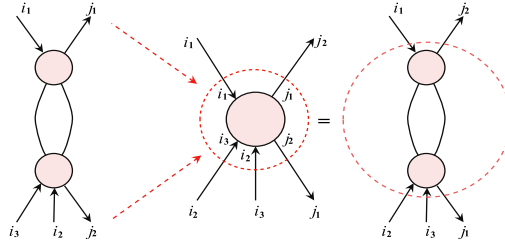}
		\end{centering}
		\caption{An example of a graph substitution which permutes the labels of a graph.}\label{figure:graph relabelling}
	\end{figure}
\end{example}

\subsection{Wheeled props}\label{subsec:WP}
In this section, wheeled props are defined as algebras over a monad as in \cite[Definition 2.1.7]{mms}. We first include some short reminders of the necessary categorical notions.  

An endofunctor $T:\mathcal{C}\rightarrow\mathcal{C}$ is called a \emph{monad} if it comes equipped with two natural transformations $\mu:T^2\rightarrow T$ and $\eta: \text{Id}_{\mathcal{C}}\rightarrow T$ called \emph{multiplication} and \emph{unit} which satisfy the associativity (\ref{monad ass}) and unit conditions of a monoid (\ref{monad id}). 

\noindent\begin{minipage}{.5\linewidth}
	\begin{equation}\label{monad ass}
	\begin{tikzcd}
	TTT\arrow[r, "T\mu"]\arrow[d, "\mu T"]& TT\arrow[d, "\mu"]\\
	TT\arrow[r, "\mu"] &T
	\end{tikzcd} 
	\end{equation} 
\end{minipage}%
\begin{minipage}{.5\linewidth}
	\begin{equation}\label{monad id}
	\begin{tikzcd}
	T\arrow[r, "T \eta"]\arrow[d,"\eta T", swap] \arrow[dr, "="] & T^2 \arrow[d, "\mu"]\\
	T^2 \arrow[r,"\mu", swap]& T
	\end{tikzcd}
	\end{equation}
\end{minipage}

Given a monad $(T,\mu,\eta)$ on a category $\mathcal{C}$, a \emph{$T$-algebra}  is a pair $(X,\gamma)$ where $X$ is an object in $\mathcal{C}$ together with a structure map $\gamma: T(X)\rightarrow X$ in $\mathcal{C}$ which commutes with the monad multiplication \eqref{alg over monad} and with the unit \eqref{alg id}. A morphism $\phi:X\rightarrow Y$ between $T$-algebras is a morphism in $\mathcal{C}$ which is compatible with the $T$-action \eqref{category of alg}. The category of all $T$-algebras in $\mathcal{C}$ is denoted $\Alg_{T}(\mathcal{C})$.

\begin{minipage}{.3\linewidth}
	\begin{equation}\label{alg over monad}
	\begin{tikzcd}
	T(T(X))\arrow[r, "T(\gamma)"]\arrow[d, "\mu_X"]& T(X)\arrow[d, "\gamma"]\\
	T(X)\arrow[r, "\gamma"] &X
	\end{tikzcd} 
	\end{equation} 
\end{minipage}%
\begin{minipage}{.3\linewidth}
	\begin{equation}\label{alg id}
	\begin{tikzcd}
	X\arrow[r,"\eta_X"]\arrow[dr,"id", swap] & T(X) \arrow[d, "\gamma"]\\
	& X
	\end{tikzcd}
	\end{equation}
\end{minipage}
\begin{minipage}{.3\linewidth}
	\begin{equation}\label{category of alg}
	\begin{tikzcd}
	T(X)\arrow[d,"\gamma_X"] \arrow[r, "T(\phi)"]& T(Y)\arrow[d,"\gamma_Y"]\\
	X\arrow[r,"\phi"] &Y
	\end{tikzcd}
	\end{equation} 
\end{minipage}

\begin{example}
	A standard example of a monad is the ``monad for monoids'' $T:\mathbf{Set}\rightarrow\mathbf{Set}$ given by $TX=\coprod_{n\geq 0}X^{n}$, that is, words in $X$. Monadic multiplication is concatenation of words. The unit is given by the inclusion $\eta_X:X=X^1\hookrightarrow TX$.
	An algebra over this monad is a choice of set $M$ together with an associative and unital multiplication $M\times M\rightarrow M$. In other words, an algebra over this monad is an associative monoid in $\mathbf{Set}$. 
\end{example}

The monad ``encoding'' wheeled props arises from graph substitution. In short, it is an {\em endofunctor on categories of equivariant vector space valued diagrams indexed by a category of graphs}. We explain this sentence in detail over the next few pages; the first step is to introduce the category of $\calS$-bimodules in $\mathsf{Vect}$. 

\begin{definition}
	Given a countable alphabet $\calI$, an {\em $\mathcal S$-bimodule} $\bE$ is a family of vector spaces $\{\sE[I;J]\}_{I;J}$, where $I$ and $J$ run over finite subsets of $\calI$. Each $\sE[I;J]$ is equipped with commuting left and right actions of the symmetric groups $S_{I}$  and $S_{J}$, respectively. 
	A morphism $f:\bE\rightarrow\mathbf{E}'$ of $\calS$-bimodules is an $S_I\times S_J$-equivariant family of linear maps $f_{I,J}:\mathsf{E}[I;J]\rightarrow\mathsf{E}'[I;J]$. The category of $\calS$-bimodules is denoted by $\mathsf{Vect}^{\calS}$. \end{definition} 

In the context of this paper $\mathcal S$-bimodules are used to \emph{decorate} vertices of graphs: if $\bE$ is an $\calS$-bimodule and $v$ is a vertex of a graph G with $(\tin(v);\tout(v))=(I;J)$, then $v$ can be {\em decorated} by an element of $\sE[I;J]$ and, to decorate the entire graph, these elements are tensor multiplied together in the order of the vertices of $G$:

\begin{definition}\label{def: decoration}
	Given an $\mathcal S$-bimodule $\bE=\{\sE[I;J]\}$ and an isomorphism class of directed graphs $[G]$, we define the {\em $\bE$-decorated graph} as the tensor product \[ \bE[G]=\bigotimes_{v\in V(G)} \sE[\text{in}(v);\text{out}(v)].\] If $G$ has no vertices, then $\bE[G]=\Bbbk$, the tensor unit. 
\end{definition}

A counter-intuitive aspect of this Definition~\ref{def: decoration} is that the {\em $\bE$-decorated graph} is a vector space which does not ``remember'' the isomorphism class of $G$. The following lemma makes this observation explicit, by noting that $\bE[G]$ only depends on the input and output sets of the vertices of $G$:
\begin{lemma}\label{lem:DecorationVertexLabels}
	Assume that $G$ and $G'$ are directed graphs with $V(G)=\{v_1,...,v_n\}$ and $V(G')=\{w_1,...,w_n\}$, and assume furthermore that $(\tin(v_i);\tout(v_i))=(\tin(w_i);\tout(w_i))$, for all $i=1,...,n$. Then there is a canonical isomorphism $\bE[G]\cong\bE[G']$.
\end{lemma}
\begin{proof}
	Immediate from the definition.
\end{proof}

The next lemma establishes the relationship between decorations and graph substitution:	
\begin{lemma}\label{lemma:decorations commute}
	Given an $\mathcal S$-bimodule $\bE$, an oriented graph $G$ and a collection $\{H_v\}_{v\in V(G)}$ so that the graph substitution $G(\{H_v\})$ is defined, there is a natural isomorphism $\bE[G(\{H_v\})]\cong\bigotimes_{v\in V(G)} \bE[H_v]$. 
\end{lemma}

\begin{proof}
	This is a direct consequence of Formula~(\ref{eq:GraphSubst}). 
\end{proof}

Next, we define a {\em category of graphs}:

\begin{definition}\label{def:GIJ}
	Let $\calG(I;J)$ denote the category whose 
	objects are strict isomorphism classes (as in Definition~\ref{def:graph isomorphism}) of graphs [G] with $\partial{G}=(\tin(G);\tout(G))=(I;J).$
	Morphisms in $\calG(I;J)$ are bijections of flags which preserve the 
	vertex sets, involutions, directions, and labels. Morphisms may permute the vertex order -- that is, permute ordering of the partition, while fixing the exceptional cell. The category $\calG$ is the coproduct $\coprod \calG(I;J)$, where $I$ and $J$ run over finite subsets of $\calI$. 
\end{definition} 

Note that morphisms in $\calG(I;J)$ are ``isomorphisms'' of graphs in a slightly looser sense than the strict isomorphisms of Definition~\ref{def:graph isomorphism}. 

To summarise, the category $\mathsf{Vect}^{\calS}$ is {\em indexed} by graphs, that is, graph decoration defines the object-level of a bifunctor $\mathsf{Vect}^{\calS}\times \calG\rightarrow \mathsf{Vect}$, were $\calG$ is the category of graphs.  

We're now ready to define the {\em monad of graph substitution}: this is an endofunctor \[\begin{tikzcd}\mathsf{F}:\mathsf{Vect}^{\calS}\arrow[r]&\mathsf{Vect}^{\calS}\end{tikzcd}.\] At the level of objects, $\mathsf{F}$ is defined by sending an $\calS$-bimodule $\bE$ to an $(I;J)$-indexed collection of vector spaces, with symmetric group actions to be defined afterwards. As a vector space, \[ \mathsf{F}\sE[I;J] := \colim\limits_{[G]\in\calG(I;J)} \bE[G],\]where $\bE[G]$ is the graph decoration as in Definition~\ref{def: decoration}. The colimit here is essentially a direct sum of vector spaces; for convenience, we briefly recall the definition in this context:

The colimit $\colim\limits_{[G]\in\calG(I;J)} \bE[G]$ is a vector space, equipped with, for every $[G]\in \calG(I;J)$, a structure map $\bE[G] \to \colim\limits_{[G]\in\calG(I;J)} \bE[G]$. Each morphism $[G] \to [G']$ of $\calG(I;J)$ induces a tensor factor permuting isomorphism $\bE[G] \to \bE[G']$, and the structure maps are compatible (form commutative triangles) with each of these isomorphisms. The colimit has the universal property that given any vector space $\sV$, with maps $f_{[G]}:\bE[G] \to \sV$ for all  $[G]\in \calG(I;J)$, all the maps $f_{[G]}$ factor through the colimit via the structure maps, and the colimit is the unique vector space, up to a unique isomorphism, with this property.  

We say that $\mathsf{F}\sE[I;J]$ is the \emph{space of $\bE$-decorations} of (isomorphism classes of) graphs $G$ with boundary labels $(I;J)$. 

To make $\mathsf{F}$ an endofunctor, it needs to take its values in $\mathsf{Vect}^{\calS}$, that is, we need to define commuting left $S_I$ and right $S_J$ actions on $\mathsf{F}\sE[I;J]$. Note that $S_I$ and $S_J$ have natural commuting actions on the indexing category $\calG(I;J)$, by permuting the incoming and outgoing boundary labels of a graph $G\in \calG(I;J)$. This can be accomplished by substituting into the appropriate label permuting corolla, as in Figure~\ref{figure:graph relabelling}. For $\sigma\in S_I$, $\tau\in S_J$, let $\sigma G\tau := C_{\sigma,\tau}(G)$ denote the graph with permuted boundary labels, where $C_{\sigma,\tau}$ is the label permuting corolla. Then, by Lemma~\ref{lem:DecorationVertexLabels}, there is a canonical isomorphism $\bE[G]=\bE[\sigma G\tau]$. The action of the pair $(\sigma,\tau)$ on $\mathsf{F}\sE[I;J]$ sends the summand $\bE[G]$ to $\bE[\sigma G \tau]$ via this canonical isomorphism. Hence, $\mathsf{F}$ is indeed an endofunctor on $\mathsf{Vect}^{\calS}$.

For the monad structure on $\mathsf{F}$, we need to define the monadic multiplication $\mu: \mathsf{F}^2 \to \mathsf{F}$.
Informally, $\mathsf{F}^2 \sE [I;J]$ is ``the space of graphs whose vertices are decorated by $\bE$-decorated graphs'', so one can picture a monadic multiplication $\mu:\mathsf{F}^2\sE[I;J]\rightarrow \mathsf{F}\sE[I;J]$ defined by substituting all the ``decorating graphs'' into the appropriate vertices. This is possible as decorations ``commute'' with graph substitution by Lemma~\ref{lemma:decorations commute}. The next proposition makes this paragraph precise.

\begin{prop}\label{prop: monad multiplication}
	Graph substitution induces a natural transformation $\begin{tikzcd} \mu: \mathsf{F}^2 \arrow[r]& \mathsf{F}.\end{tikzcd}$ 
\end{prop}

\begin{proof}
	Note that, for each graph $G$ in $\calG(I;J)$, graph substitution (Definition~\ref{def: graph substitution}) describes a functor 
	\[\mathbb{S}: \prod\limits_{v\in V(G)}\calG(\tin(v);\tout(v)) \to \calG(I;J) \] 
	which sends a family of graphs $\{H_v\}_{v\in V(G)}$ to $G(\{H_v\})$ at the level of objects. 
	
	A morphism $\prod_{v \in V(G)} \varphi_v : \{H_v\} \to \{H'_v\}$ is a permutation of the vertex sets of each $H_v$. Via the canonical bijection $V(G(\{H_v\}))= \bigsqcup_{v\in V(G)} V(H_v)$, the disjoint union of the permutations $\varphi_v$ induces a vertex permutation on $V(G)$, which is, in turn, a morphism in $\calG(I;J)$. This defines $\mathbb S$ at the level of morphisms, and it is clear that $\mathbb S$ is a functor.

	By definition, $\mathsf{F}^2 \sE [I;J] = \colim_{[G]\in\calG(I;J)} (\mathsf{F}\bE)[G]$. To define a natural transformation $\mu:\mathsf{F}^2\rightarrow \mathsf{F}$, it is sufficient to define, for each $G\in\calG(I;J)$, a map $m: (\mathsf{F}\bE)[G]\rightarrow \mathsf{F}\sE[I;J]$. Then the universal property of the colimit gives rise to the map $\mu$:
	\begin{equation}\label{eq:monad composite}
	\begin{tikzcd}
	(\mathsf{F}\bE)[G]\arrow[r]\arrow[rr, bend left=25, "m"]& \mathsf{F}^2\sE[I;J]\arrow[r, "\mu"]& \mathsf{F}\sE[I;J]. 
	\end{tikzcd} 
	\end{equation}  
	
	To define the map $m$, observe the following:
	\begin{align}\label{eq: equalities}(\mathsf{F}\bE)[G] &=  \bigotimes_{v\in V(G)} (\mathsf{F}\sE)[\tin(v);\tout(v)]= 
	\bigotimes_{v\in V(G)} \,\, \colim_{[H_v]\in\calG(\tin(v);\tout(v))}  \bE[H_v]  
	\\&\cong \colim\limits_{\calG(v)}\,\, \bigotimes_{v\in V(G)} \bE[H_v]\simeq  \colim\limits_{\calG(v)} \bE[G(\{H_v\})],
	\end{align} 
	where we have written $\calG(v) =\prod\limits_{v\in V(G)}\calG(\tin(v);\tout(v))$ to shorten notation. The equalities in \eqref{eq: equalities} follow from the definition of the functor $\mathsf{F}$. The first isomorphism is the fact that tensor products commute with colimits in $\mathsf{Vect}$, and the second isomorphism is by Lemma~\ref{lemma:decorations commute}. 
	
	Since $[G(\{H_v\})]\in\calG(I;J)$, the graph substitution functor $\mathbb{S}$ induces a map $\tilde{\mathbb{S}}$, which composes with the isomorphism above to define $m$:  \[\begin{tikzcd}(\mathsf{F}\bE)[G]\arrow[r,"\cong"] \arrow[rrr,bend left=15,"m"] &\colim_{\{H_v\}\in\calG(v)}\bE[G(\{H_v\})]\arrow[r, "\tilde{\mathbb{S}}"]&\colim\limits_{[K]\in \calG(I;J)}\bE[K]\arrow[r,"\cong"]&\mathsf{F} \sE[I;J].\end{tikzcd}\] 
	%since graph substitution is associative and unital. This makes the endofunctor $F$ into a monad.  For further reference see also \cite[Section 2]{mms} where it is first used to define wheeled props, and for a study in complete generality see \cite[Chapter 10]{yj15}.
\end{proof}

To define the monad unit, note that if $C_{(I;J)}$ is the corolla whose vertex labels agree with its boundary labels $(I;J)$ then $\bE[C_{(I;J)}]=\sE[I;J]$. 
\begin{definition}\label{def: monad unit}The map $$\begin{tikzcd}\bE[C_{(I;J)}]\arrow[r]&\colim_{\calG(I;J)}\bE[G]\end{tikzcd}$$ defines a natural transformation $$\begin{tikzcd}\eta_{\bE}: \{\sE[I;J]\}\arrow[r]& \{\mathsf{F}\sE[I;J]\}.\end{tikzcd}$$ 
\end{definition}

The following proposition is used to define wheeled props in \cite[Section 2]{mms}:
\begin{prop}\label{def:monad} The endofunctor $\mathsf{F}$ together with the natural transformations $\mu: \mathsf{F}^2 \rightarrow \mathsf{F}$ and  $\eta: \Id_{\mathsf{Vect}^{\calS}} \rightarrow \mathsf{F}$ is a monad on the category $\mathsf{Vect}^{\calS}$.% the {\em monad of graph substitutions}.% with the mondaic multiplication and unit defined as follows. The {\em monadic multiplication} $\mu: F^2\bE\rightarrow F\bE$ sends a summand $\bE[G(\{H_v\})]$ of $F^2\sE[I;J]$, indexed by the pair $([G], \{[H_v]\})$, to the summand $\bE[G(\{H_v\})]$ of $F\sE[I;J]$, indexed by the substituted graph $G(\{H_v\})$.  The {\em monadic unit} $\eta: \bE\rightarrow F\bE$ is defined by the composition $\sE[I;J]\xrightarrow{Id} \bE[C_{I;J}]\hookrightarrow F\sE[I;J],$ where $C_{I;J}$ is the unit corolla whose vertex labels agree with its boundary labels $(I;J)$, and the second map of the composition is the inclusion of $\bE[C_{I;J}]$ in $F\sE[I;J]$ as a summand, as $[C_{I;J}]\in \calG(I;J)$. 
\end{prop}
\begin{proof}
	The natural transformations $\mu$ (Definition~\ref{prop: monad multiplication}) and $\eta$ (Definition~\ref{def: monad unit}) are associative and unital since graph substitution is associative and unital.  
\end{proof}

\begin{definition}\label{def:WP}
	A linear \emph{wheeled prop} $\bE$ is an algebra over the monad $\mathsf{F}$ in the category of $\mathcal{S}$-bimodules $\mathsf{Vect}^{\calS}$.  The category of linear wheeled props is the category of $\mathsf{F}$-algebras and is denoted by $\mathsf{wProp}$.
	%. A map of wheeled props is a map of $\mathcal{S}$-bimodules $\phi:\bE\rightarrow\bE'$ which commutes with the action of $F$. We denote the category of wheeled props by $\mathsf{wProp}$.
\end{definition} 

In other words, a wheeled prop is an $\mathcal{S}$-bimodule $\bE$ together with an action $\gamma: \mathsf{F}\bE \to \bE$.  Note that, given an $\mathcal S$-bimodule $\bE$, the \emph{free wheeled prop generated by} $\bE$ is the $\mathcal S$-bimodule $\mathsf{F}\bE$ with structure map the monadic multiplication $\mu: \mathsf{F}^2\bE \to \mathsf{F}\bE$. We give examples and an alternative description of wheeled props in Section~\ref{section: categorical description}. 

\begin{remark}
	The reader may have noticed that circuit algebras are described as algebras over an operad and wheeled props are described as algebras over a monad. In general, one can associate, to any operad $\mathsf{O}$, a monad $M_{\mathsf{O}}$ with the property that $\mathsf{O}$-algebras are $M_{\mathsf{O}}$-algebras. It is not the case, however, that all monads come from operads. For full details see \cite[Appendix C]{MR2094071}. 
\end{remark}

%%%%%%%%%%%%%%%%%%%%%%%%%%%%

\section{Equivalence}\label{sec:Eq}

In this section we prove that there is an equivalence of categories between the category of circuit algebras and the category of linear wheeled props. The key observation is that wiring diagrams (Definition~\ref{def:WD}) are in bijection with oriented graphs (Definition~\ref{def: oriented graph}), and under this bijection, wiring digram composition corresponds with graph substitution. This correspondence leads to the equivalence of categories proven in Theorem~\ref{thm:main}.

\subsection{Graphs and Wiring Diagrams} \label{sec:graphs and WD}
The goal of this subsection is to define a (structure respecting) correspondence $$\Phi: \{\text{iso.\ classes of oriented labelled graphs}\} \to \{\text{wiring diagrams}\}.$$
The idea behind $\Phi$ is that the vertices of graphs can be viewed as input circles of wiring diagrams, as illustrated in Figure~\ref{figure:graphtoWD}. The technical details take more work.

\begin{figure}[h]
	\begin{centering}
		\includegraphics[height=4.5cm]{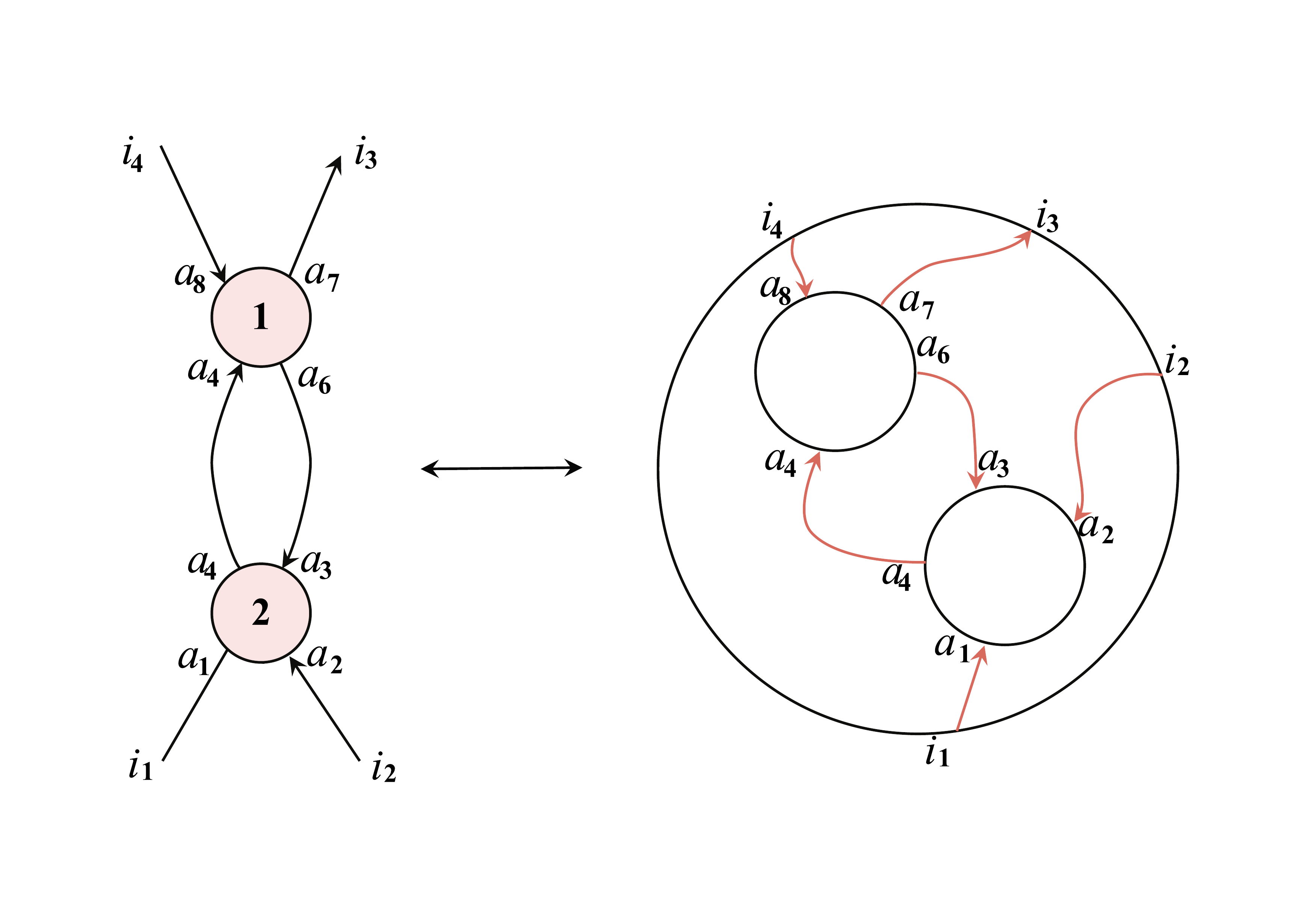}
	\end{centering}
	\caption{The correspondence between graphs and wiring diagrams.}\label{figure:graphtoWD}
\end{figure}

Recall that the data of an oriented graph $G$ consists of a set of flags $\calF(G)$, an ordered partition $\calF(G)= (\sqcup_{\alpha \in V(G)} v_{\alpha}) \sqcup v_{\epsilon}$, an involution $\iota$, a fixed point free involution $\pi$ on the $\iota$-fixed points of $v_{\epsilon}$, an orientation map $\calF_{0} \to \{-1,1\}$, and labelling functions $\lambda: \coprod_{{\alpha}\in V(G)} v_{\alpha} \to \calI$ on the vertices, and $\beta: \partial G \to \calI$ on the boundary.

By Proposition~\ref{rmk:Matching}, a wiring diagram is a triple $D=(\calA, p,l)$ where $\calA$ is the set of finite sets of labels $\{A_{0}^{\tin/\tout},...,A_{r}^{\tin/\tout}\}$,  $p$ is a perfect matching (bijection) between the finite sets 
$\bigsqcup_{i=0}^r 
A^{\text{out}}_i$ and $\bigsqcup_{i=0}^r A^{\text{in}}_i$, and $l\in \Z_{\geq 0}$ is a non-negative integer (the number of circles). 

%\medskip In Construction~\ref{map:Phi} we define $\Phi$ as a map of sets. Construction~\ref{map:Psi} describes its inverse.

\begin{cons}\label{map:Phi}
	We construct a correspondence $\Phi$  which assigns to an isomorphism class of oriented labelled graphs $[G]$ -- represented by a graph $G$ -- a wiring diagram $D_{[G]}=(\calA, p, l)$. % defined as follows:
	\begin{enumerate}
		\item We define the number $l= \frac{1}{2} \#\{\varphi \in v_{\epsilon}: \iota(\varphi)\neq \varphi\}$, where the $\#$ sign denotes the cardinality of the set that follows it. Simply put, $l$ is the number of free loops in $G$. 
		
		\item The vertices $V(G)=\{1,2,...,r\}$ give rise to $\calA=\{A^{\tin/\tout}_{0}, A^{\tin/\tout}_{1},...,A^{\tin/\tout}_{r}\}$ by setting  
		\begin{itemize}
			\item[-] $\tin(G)=A_0^\tout, \; \tout(G)=A_0^\tin.$ \hspace{1mm}(Note: the switch of in/out here is intentional and is due to opposite conventions.)
			\item[-] $ \tin(v_i)=A_i^\tin,\quad \tout(v_i)=A_i^\tout \quad \text{for}\quad i=1,...,r.$
		\end{itemize}
		
		\item The bijection $p$ is built as follows:
		\begin{itemize}
			\item[-] For each label $a \in A_{i}^{\tout/\tin}$, $1\leq i \leq r$,
			there is a unique negative/positive (respectively) flag $\lambda^{-1}(a) \in v_{i}$. If $\lambda^{-1}(a)$ is not fixed by $\iota$, and $\iota\lambda^{-1}(a) \in v_j$ then $p(a):=\lambda \iota\lambda^{-1}(a)\in A_j^{\tin/\tout}$. If $\iota(\lambda^{-1}(a))=\lambda^{-1}(a)$ then $p(a):=\beta \lambda^{-1}(a)\in A_{0}^{\tin/\tout}$. 
			\item[-] For each label $a \in A_{0}^{\tout/\tin}$, there is a unique positive, respectively negative (once again in/out conventions are opposite here) flag $\beta^{-1}(a) \in \partial G$. If $\beta^{-1}(a) \in v_i$ for some $i$, then define $p(a):= \lambda \beta^{-1}(a) \in A_i^{\tin/ \tout}$. If $\beta^{-1}(a) \in v_\epsilon$, then define $p(a):=\beta \pi \beta^{-1} (a) \in A_0^{\tin / \tout}$.
		\end{itemize}
	\end{enumerate}
	
\end{cons}
Since labelled graph isomorphisms preserve the ordered partitions, involutions, orientation and labellings (Definition~\ref{def:graph isomorphism}), the wiring diagram $D_{[G]}$ does not depend on the representative of the isomorphism class $[G]$. Hence, $\Phi$ is well-defined.

Next, we construct an inverse map $$\Psi: \{\text{Wiring diagrams}\} \to \{\text{Iso.\ classes of oriented labelled graphs}\},$$ 
which, intuitively, turns input circles into vertices and removes the output circle. 

\begin{cons}\label{map:Psi}
	Given a wiring diagram  $D=(\{A^{\tout/\tin}_{0},...,A^{\tout/\tin}_{r}\},p,l)$,  we define an isomorphism class of graphs $\Psi(D)=[G_{D}]$ as follows.
	\begin{enumerate}
		\item The set of flags \begin{equation*}\begin{split}\calF(G_D)& :=\{(a, i,\tout/\tin): a \in A^{\tout/\tin}_{i} \text{ for } 1\leq i \leq r\} \quad
		\cup   \\& \{(a,0,\tout/\tin): a\in A^{\tout/\tin}_{0} \text{ and } p(a) \in A^{\tin/\tout}_{0}\}  \cup  \{c_{j}\}_{j=1}^{2l}.\end{split}\end{equation*} Here, ``$(a,i,\tout/\tin)$'' stands for ``$(a,i, \tout)$ and $(a,i,\tin)$''.
		\item The flags $\calF(G)$ are partitioned into vertices $\bigcup_{i=1}^{r} v_i\cup v_{\epsilon}$ with $$v_{i}:=\{(a,i,\tout/\tin): a \in A^{\tout/\tin}_{i}\} \text{ for } 1\leq i \leq r,$$ and $$v_{\epsilon} = \{(a,0,\tout/\tin): a\in A^{\tout/\tin}_{0}, \text{ and } p(a) \in A^{\tin/\tout}_{0}\} \cup \{c_{j}\}_{j=1}^{2l}.$$
		\item The involutions $\iota$ and $\pi$ are set as follows:
		\begin{itemize} 
			\item[-] On the flags in the vertices $v_i$, $i=1,...,r$,  
			\[\iota(a,i, \tout/\tin)=\begin{cases} ( p(a),j, \tin/\tout), \text{ where } p(a) \in A_{j}^{\tin/\tout}, j\neq 0\\(a,i, \tout/\tin), \text{ where } p(a) \in A_{0}^{\tin/\tout}.\end{cases} \] This defines the internal edges of the graph $G_D$, and the boundary edges connected to a vertex. 
			\item[-]If $a\in A^{\tout/\tin}_{0}$ such that $p(a) \in A^{\tin/\tout}_{0}$, then $(a,0,\tout/\tin)$ is an $\iota$-fixed point, and $\pi(a,0,\tout/\tin)=(p(a), 0, \tin/\tout)$. This gives free-floating edges in the graph $G_D$. 
			\item[-] It remains to describe the free floating loops on $G_D$. For $c_j$, where $j$ is odd, $\iota(c_j)=c_{j+1}$ and where $j$ is even, $\iota(c_j)=c_{j-1}$. 
		\end{itemize} 
		\item The labelling functions $\lambda$ and $\beta$ are as follows. For each $(a,i, \tout/\tin) \in v_i \quad (i=1,...,r)$, we set $\lambda(a,i,\tout/\tin)=a$. 
		Furthermore, if $(a,i,\tout/\tin)\in v_i$ is an $\iota$-fixed point, then set $\beta(a,i,\tout/\tin)=p(a)$. For the $\iota$-fixed points of type $(a,0,\tout/\tin)$, define $\beta(a,0,\tout/\tin)=a$.
		
		\item The direction function is defined by $\delta(a,i, \tin)=1$ and $\delta(a,i,\tout)=-1$ where $i=1,...,r$; while $\delta(a,0, \tin)=-1$ and $\delta(a,0,\tout)=1$. Keep in mind that free floating loop flags don't have signs -- in technical notation, $\calF_0=\calF \setminus \{c_j\}_{j=1}^{2l}$.
		
	\end{enumerate}
\end{cons} 

\begin{lemma}\label{lemma: graphs are wiring diagrams}
	The maps $\Phi$ and $\Psi$ are inverse maps of sets, therefore set bijections.
\end{lemma}

\begin{proof}
	It is straightforward to check that $\Phi\circ \Psi$ is the identity map on wiring diagrams. For a graph $G$, the composition $\Psi(\Phi(G))$ renames the flags of G, but retains the labelling, therefore it does not change the isomorphism class of $G$.
\end{proof}

\begin{lemma}\label{lem:CompSub}
	The map $\Psi$ translates wiring diagram composition to graph substitution:
	$$\Psi(D_1 \circ_i D_2)= \Psi(D_1)\big(\Psi(D_2)_i\big).$$ 
\end{lemma}

\begin{proof}
	A straightforward verification.
\end{proof}

A concise way to summarize the above is that the operad structure of wiring diagrams induces the monad structure on graphs, via the map $\Phi$. Formally, this implies that they admit the same  algebras, a statement we unpack in the final proof below.

\subsection{Equivalence of Categories}
We are now ready to prove the main result of this paper:

\begin{thm}\label{thm:main}
	There is an equivalence of categories \[\begin{tikzcd}\mathsf{CA}\arrow[r, "\widetilde\Phi", shift left=.5ex]& \arrow[l, "\widetilde\Psi", shift left=.5ex]\mathsf{wProp}\end{tikzcd}\] between the category of circuit algebras and the category of linear wheeled props. 
\end{thm}

\begin{lemma}\label{lem:CAPerm}
	Every circuit algebra has an underlying $\mathcal{S}$-bimodule. 
\end{lemma} 

\begin{proof} 
	This is simply stating that the collection of vector spaces that comprise a circuit algebra $\mathsf{V}=\{\V[T^\tout;T^\tin]\}$ each carry natural symmetric group actions by $S_{T^\tout}$ on the left and $S_{T^\tin}$ on the right, where $T^{\text{out}}, T^\text{in} \subseteq \calI$. Indeed, this is the case, induced by the action of {\em label permuting wiring diagrams} as explained in Example~\ref{ex:IdAndPerm}.
\end{proof} 

\begin{proof}[Proof of Theorem~\ref{thm:main}]
	%{\em Proof of Theorem~\ref{thm:main}}.
	First, we show that every a circuit algebra $\sV$ admits the structure of a wheeled prop $\widetilde{\Phi}(\sV)=\bV$. By Lemma~\ref{lem:CAPerm}, $\sV$ has an underlying $\calS$-bimodule $\bV=\{V[I;J]\}_{I,J\subseteq \calI}$. To describe a wheeled prop structure on $V$, we need to exhibit a structure map $\gamma: \mathsf{F}\bV \to \bV$. 
	
	Recall from Section~\ref{subsec:WP} that $\mathsf{F}\sV[I;J]= \colim_{[G]\in \calG(I;J)} \bV[G]$. In other words, $\mathsf{F}\sV[I;J]$ is linearly spanned by  isomorphism classes of labelled directed graphs $[G]$ with $\partial G= (I;J)$, where the vertices $v_i \in V(G)$ are decorated with the vector spaces $\sV[\tin(v_i); \tout(v_i)]$. 
	
	The map $\Phi$ from Construction~\ref{map:Phi} assigns to $[G]$ a wiring diagram $D_{[G]}$. By definition of the circuit algebra structure on $\sV$, $D_{[G]}$ induces a linear map  $F_{D_{[G]}}: \bigotimes_{i=1}^r \sV[\tin(v_i); \tout(v_i)] \to \sV[I;J]$. 
	
	The maps $F_{D_{[G]}}$ are natural in $[G]$: any map $[G]\rightarrow [G']$ in $\calG(I;J)$ corresponds to a permutation of vertex order, which is respected by assignment of linear maps in the circuit algebra (Axiom (2) in Definition~\ref{def:CA}) and thus we define the structure map $\gamma$ using the universal property of the colimit, for each pair of label sets $I,J \subseteq \calI$:
	\[\begin{tikzcd} \bV[G] \arrow[d]\arrow[drrr, "F_{D_{[G]}}"] & \\ 
	\mathsf{F}\sV[I,J]=\colim_{[G]\in\calG(I,J)} \bV[G]  \arrow[rrr, swap, "\gamma"] &&& \sV[I;J] \end{tikzcd}\]   
	Thus, $\bV$ is a wheeled prop. It is clear from its construction that  a circuit algebra map $\sV \to \sW$ is automatically also wheeled prop map $\bV \to \bW$ ($\widetilde\Phi$ is natural in $\sV$), and thus we have defined a functor $\widetilde\Phi:\mathsf{CA}\rightarrow \mathsf{wProp}$. 
	%Then the circuit algebra axioms translate directly to the structure map properties, thus every circuit algebra admits a wheeled prop structure.
	
	In the other direction, given a linear wheeled prop $\bW$ we construct a circuit algebra $\widetilde{\Psi}(\bW)=\sW$. First, $\bW$ has an underlying $\calS$-bimodule, which in particular is a collection of vector spaces $\sW[I;J]$, where $I$ and $J$ run over finite subsets of $\calI$. % set this to be the collection of vector spaces for the circuit algebra $W$.
	It remains to construct the action maps $$F_D: \sW[A_1^\tout;A_1^\tin]\otimes ... \otimes \sW[A_r^\tout;A_r^\tin] \to \sW[A_0^\tin;A_0^\tout]$$ for each wiring diagram $D=(\calA, p, l)$. 
	The map $\Psi$ assigns to $D$ an isomorphism class of graphs $[G_D]=\Psi(D)$. We define the action $F_D$ as the restriction of $\gamma: \colim_{[G]\in\calG(I,J)} \bW[G]\rightarrow \sW[I;J]$ to the component $\bW[G_D]$.  
	
	The composition axiom -- Axiom (1) -- of Definition~\ref{def:CA} holds by Lemma~\ref{lem:CompSub} and axiom (\ref{alg over monad}) of an algebra over a monad, as $\mu$ captures wiring diagram composition, and $\gamma$ captures the assignment of linear maps to wiring diagrams. The equivariance -- Axiom (2) -- holds by the naturality of $\Psi$ in $G_D$: input set permutations $D\to D'$ correspond to morphisms $G_D \to G_{D'}$ in $\calG$.  The assignment $\widetilde\Psi: \mathsf{wProp}\rightarrow \mathsf{CA}$ is natural in $\bW$ and thus $\widetilde\Psi$ defines a functor.  The fact that $\widetilde\Phi$ and $\widetilde\Psi$ are inverse functors follows from Lemma~\ref{lemma: graphs are wiring diagrams}. 
\end{proof}

%%%%%%%%%%%%%%%%%%%%%%%%%%%%%%%%%%%%%%%%%%%%%%

\section{Wheeled props as tensor categories}\label{section: categorical description}

In \cite[Chapter V]{MacLane_Props}, Mac Lane introduced \emph{props} as strict symmetric tensor categories whose monoid of objects has a single generator: that is, a symmetric tensor category equipped with a distinguished object $x$ such that every object is a tensor power $x^{\otimes n}$, for some $n\geq 0$. Hence, morphisms in a prop are of the form $f:x^{\otimes n}\rightarrow x^{\otimes m}$. Diagrammatically, such a morphism is illustrated by an $(n,m)$-corolla whose vertex is decorated by $f$ (as on the left in Figure~\ref{fig:propmorphism}). 

Composition of morphisms, also called {\em vertical composition}, is modelled diagrammatically by attaching some of the outputs of one corolla to inputs of another, resulting in directed graphs. Directed graphs are composed the same way. The tensor product of the prop is realised by taking disjoint unions of graphs, and is called {\em horizontal composition}. For examples of both compositions see Figure~\ref{fig:propmorphism}. 

In other words, props are categories in which morphisms are directed graphs, where every edge ``carries'' a copy of the generator $x$. Note that these graphs have no floating loops or closed cycles; a floating edge denotes the identity $\id:x\to x$.  For more details on this point of view, and examples of props, we suggest the survey article~\cite[Section 8]{markl_operads_and_props}. 

\begin{figure}[h]
	\begin{centering}
		\includegraphics[height=4cm]{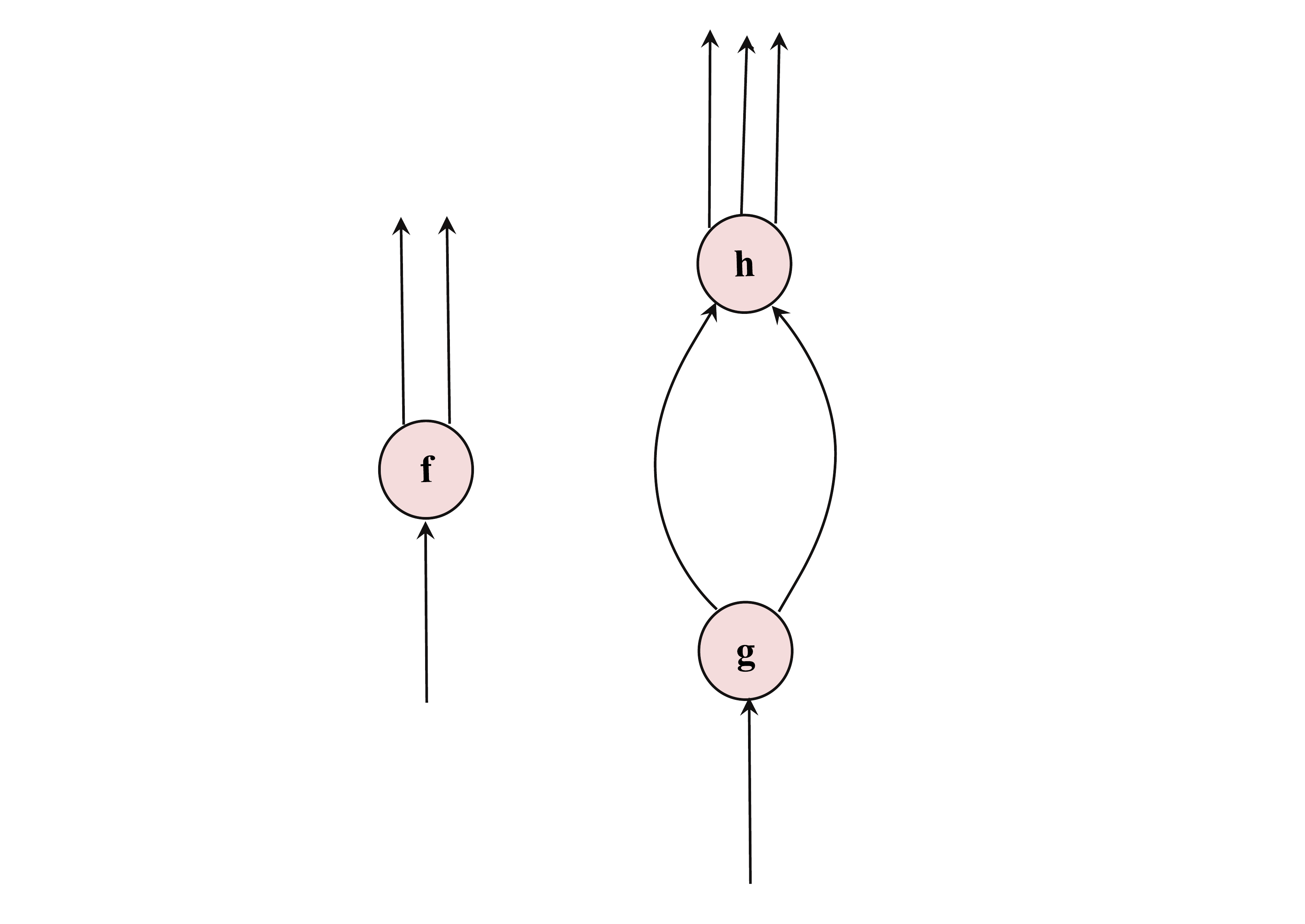}
	\end{centering}
	\caption{The tensor product of morphisms $f$ and $hg$ where $f:x\rightarrow x^{\otimes 2}$ and $hg$ is the composite of morphisms $g:x\rightarrow x^{\otimes 2}$ and $h:x^{\otimes 2}\rightarrow x^{\otimes 3}$  in a prop $\mathsf{P}$. Stacking morphisms next to each other like this is called \emph{horizontal composition}. The composition $hg$ is called \emph{vertical composition}.}\label{fig:propmorphism}
\end{figure}

A wheeled prop is a prop where every object has a dual. This gives rise to a family of linear ``trace'' or ``contraction'' maps \[\begin{tikzcd}\trace^{j}_{i}: x^{* \otimes m}\otimes x^{\otimes n}\arrow[r]& x^{*\otimes m-1}\otimes x^{\otimes n-1}.\end{tikzcd}\] Diagrammatically, contractions are represented by connecting a chosen output of a graph (the $j$th copy of $x$) to a chosen input of the same graph (the $i$th copy of $x^*$), as in Figure~\ref{fig:some morphisms in wheeled prop}.

\begin{figure}[t]
	\begin{centering}
		\includegraphics[height=4cm]{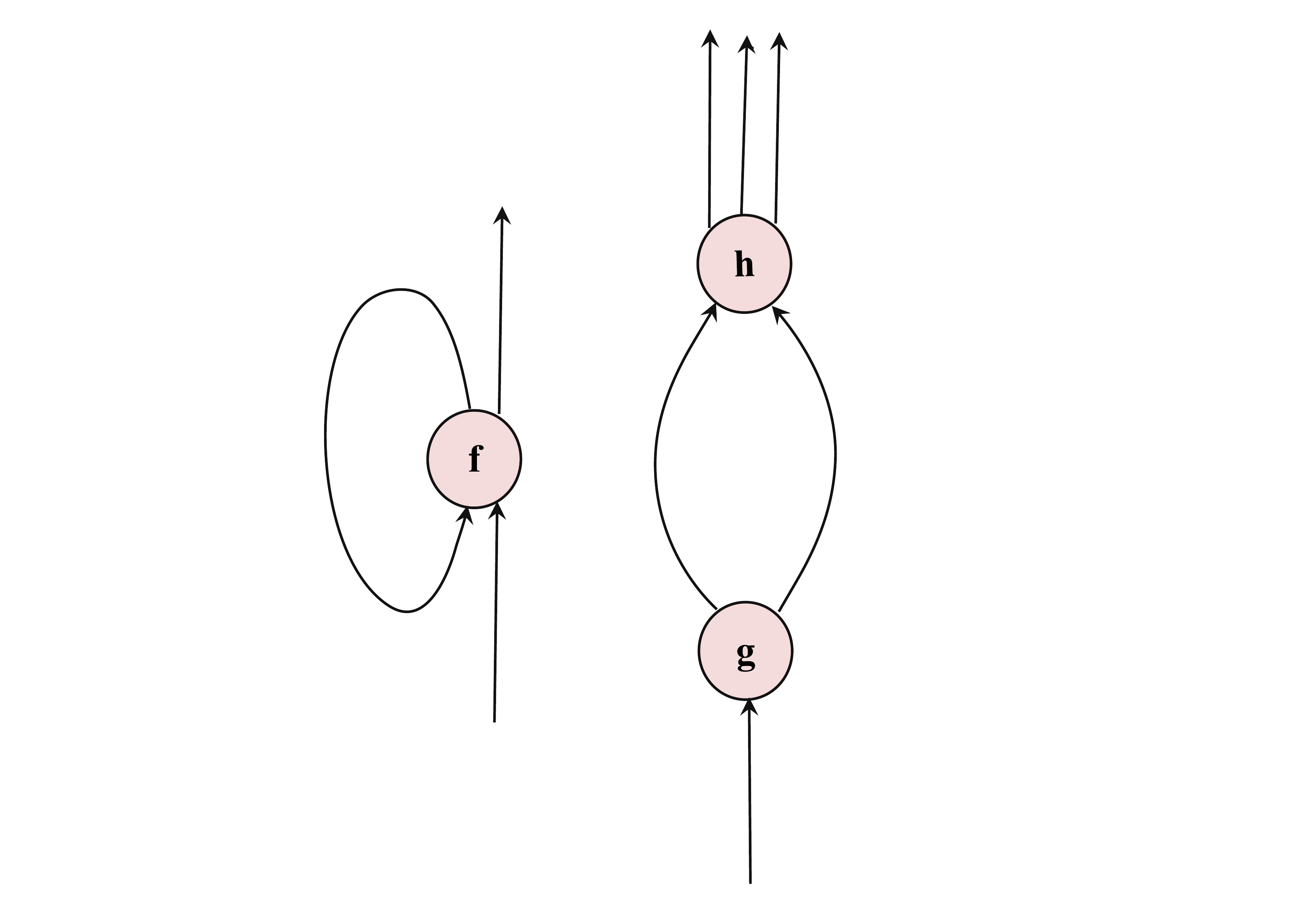}
	\end{centering}
	\caption{The horizontal product $\trace^{1}_{1}f \ast hg$.}\label{fig:some morphisms in wheeled prop}
\end{figure}

We note that in the literature a strict symmetric tensor category with duals is also called a \emph{rigid} symmetric\footnote{Rigid symmetric tensor categories are also called compact closed categories.} tensor category. Saying that such a tensor category has a single generating object is equivalent to saying that the monoid of objects has a single generator $x$. \cite{MR1106898, MR1357057} 

In this section, we present an axiomatic definition of a wheeled prop in line with this categorical view, which is equivalent to the monadic definition given in Definition~\ref{def:WP}. In light of Theorem~\ref{thm:main}, one can equivalently interpret the following as a set of axioms satisfied by circuit algebras. In Section~\ref{section:examples of wheeled props}, we present two prominent examples of wheeled props. Finally, in Section~\ref{subsec:Pivotal} we show that every wheeled prop is, in particular, a strict pivotal category. This gives a fully faithful embedding into the category of pivotal categories, parallel to that between circuit algebras and planar algebras in Proposition~\ref{prop: adjunction PA and CA}.

\medskip

\subsection{Axiomatic definition}\label{subsec:Axiomatic}
\begin{definition}\label{def:baised modular operad} Let $\calI$ denote a fixed countable alphabet. A \textbf{wheeled prop} $\bE:=\left(\bE, \ast,\trace_{i}^{j}\right)$ consists of:
	\begin{enumerate}
		
		\item an $\mathcal{S}$-bimodule $\bE=\{\sE[I ;J]\}$;
		
		\item a \emph{horizontal composition}\[\begin{tikzcd}\ast:\sE[I;J] \otimes \sE[K;L] \arrow[r]& \sE[I\cup K;J\cup L], \end{tikzcd}\]where $I\cap K =\emptyset$ and $J\cap L =\emptyset$;
		
		\item a linear map $1_{\emptyset}:\Bbbk \to \sE[\emptyset;\emptyset]$ called the \emph{empty unit};
		
		\item	 a \emph{contraction} operation\[\label{note:xiij} \begin{tikzcd}\sE[I;J] \arrow[r, "\trace^{j}_{i}"] & \sE[I\setminus\{i\};J\setminus\{j\}],\end{tikzcd}\] for every pair $i\in I$ and $j\in J$;
		
		\item a linear map $$1_i: \Bbbk \to \sE[\{i\};\{i\}],$$ for every $i\in \calI$, called the {\em unit}.
		
	\end{enumerate} This data satisfies a list of axioms that we will present in detail shortly. In particular, the horizontal composition and contractions commute with each other and are associative, equivariant and unital. 
\end{definition} 

\begin{remark}\label{remark: wheeled prop is a prop}
	Wheeled props are, in particular, examples of props (\cite[Example 2.1.1]{mms}). To define vertical composition of morphisms such as the composition of $h$ and $g$ in Figure~\ref{fig:propmorphism}, one combines horizontal compositions and contractions, as in Figure~\ref{fig:dioperadic comp}.  
	
	In fact, for $i\in I,$ $l\in L$, and $I\cap K =\emptyset$ and $J\cap L =\emptyset$, the horizontal composition and contraction operations combine to give an additional \emph{dioperadic composition}\footnote{The ${}_{i}\circ_{l}$ notation means ``identify output $l$ with input $i$.''} denoted ${}_{i}\circ_{l}$, which joins the $l$th output of one graph to the $i$th input of another:
	\[\begin{tikzcd} \sE[I; J] \otimes \sE[K;L] \arrow[r,"{}_{i}\circ_{l}"]\arrow[d,"\ast", swap] & \sE[I\setminus\{i\}\cup K;J\cup L\setminus\{l\}] \\
	\sE[I\cup K;J\cup L]\arrow[ur, "\trace_i^l", swap]&\end{tikzcd}.\] Vertical composition can then be obtained by a horizontal composition followed by iterated contractions. As an example, the morphism $hg$ from Figure~\ref{fig:some morphisms in wheeled prop} is obtained by first taking a horizontal composition of $h\in x^{* \otimes 2}\otimes x^{\otimes 2}$ and $g\in x^{*\otimes 2}\otimes x^{\otimes 2}$ and then applying contractions $\trace_{i_3}^{j_4}$ and $\trace_{i_4}^{j_2}$. See Figure~\ref{fig:dioperadic comp}.
	\begin{figure}[h]
		\begin{centering}
			\includegraphics[height=4.5cm]{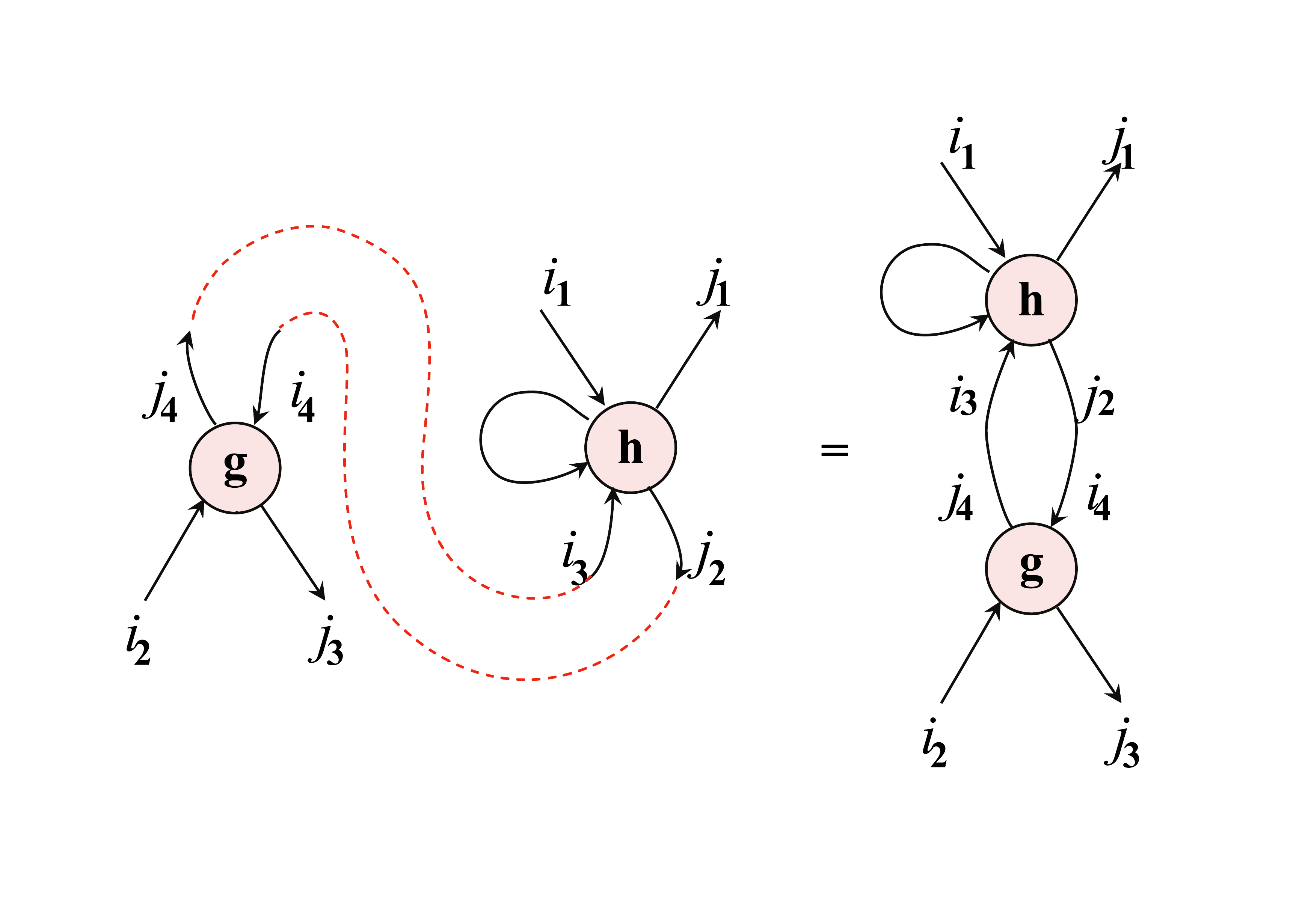}
		\end{centering}
		\caption{The vertical composition $hg$ depicted on the right is horizontal composition followed by two contractions: $\trace^{j_2}_{i_4}\trace^{j_4}_{i_3}(g\ast h)$.}\label{fig:dioperadic comp}
	\end{figure} 
	
\end{remark}

The following is a comprehensive list of axioms satisfied by the horizontal composition, contractions and the units.  Throughout, by an abuse of notation we write ``$\sigma I$'' to indicate that the set $I$ has been ``permuted'' by $\sigma\in\mathcal{S}_{I}$. Of course, $I$ is a set and $\sigma$ is a bijection so as sets $\sigma I=I$, however this notation is useful in practical examples with a naturally ordered alphabet, e.g. when $\calI=\Z_{\geq 0}$. For example, the pair $(\sigma,\tau)\in\calS_{I}\times\calS_{J}$ acts on the vector space $\sE[I;J]$ and we write $(\sigma,\tau):\sE[I;J] \to \sE[\sigma I;J\tau]$.

\medskip
\textbf{H1:} The horizontal composition is \emph{associative} in the sense that the following square commutes: \[\begin{tikzcd} \sE[I;J]\otimes \sE[K;L]\otimes\sE[M;N] \arrow[r, "\ast\otimes id"] \arrow[d, swap, "id \otimes \ast"] & \sE[I\cup K;J\cup L]\otimes\sE[M;N]\arrow[d, "\ast"] \\ 
\sE[I;J]\otimes \sE[K\cup M;L\cup N] \arrow[r, "\ast"] & \sE[I\cup K\cup M;J\cup L\cup N]. \end{tikzcd}\] 

\textbf{H2:} The horizontal composition is \emph{bi-equivariant}. Explicitly, for any two pairs of disjoint finite subsets $I,K$ and $J,L$ of $\calI$, the following square commutes:
\[ \begin{tikzcd} \sE[I;J]\otimes \sE[K;L] \arrow[r, "\ast"]\arrow[d,swap, "(\sigma_1;\tau_1)\otimes (\sigma_2;\tau_2)"]& \sE[I\cup K;J\cup L]\arrow[d, "(\sigma_1\cup\sigma_2;\tau_1\cup\tau_2)"] \\
\sE[\sigma_1I;J\tau_1]\otimes \sE[\sigma_2K;L\tau_2] \arrow[r, "\ast"] & \sE[\sigma_1I\cup\sigma_2K;J\tau_1\cup L\tau_2].\end{tikzcd}\]  
The notation $\sigma_1\cup\sigma_2$ refers to the element of $\calS_{I\cup K}$ which acts as $\sigma_1$ on $I$ and $\sigma_2$ on $K$. Similarly,  $\tau_1\cup\tau_2\in \calS_{J\cup L}$ acts by $\tau_1$ on $J$ and $\tau_2$ on $L$. 

\medskip 
\textbf{H3:}
Horizontal composition is \emph{symmetric}. Let  $\beta$ denote the block permutation $\beta=(12)\langle I,K\rangle \in\mathcal{S}_{I\cup K}$, which swaps the blocks $I$ and $K$ of $I\cup K$. In the same notation, $\gamma=(12)\langle J,L\rangle\in\mathcal{S}_{J\cup L}$ swaps the blocks $J$ and $L$ in $J\cup L$. Then, for any two pairs of disjoint subsets $I,K$ and $J,L$ of $\calI$, the following diagram commutes:
\[ \begin{tikzcd} \sE[I;J]\otimes \sE[K;L] \arrow[r, "\ast"]\arrow[d,swap, "\otimes \text{-swap}"]& \sE[I\cup K;J\cup L]\arrow[d, "(\beta;\gamma)"] \\
\sE[K;L]\otimes \sE[I;J] \arrow[r, "\ast"] & \sE[K\cup I;L\cup J].\end{tikzcd}\]  

\medskip 	
\textbf{H4:} The empty unit $1_{\emptyset}$ is a two-sided unit for the horizontal composition:

\[\begin{tikzcd}
& \arrow[dr,swap,"\cong", swap]\sE[I;J]\arrow[dl,"\cong", swap]\arrow[ddd,"="] &\\
\Bbbk\otimes \sE[I;J]\arrow[d,"1_\emptyset \otimes \id", swap] && \sE[I;J]\otimes \Bbbk\arrow[d,"\id \otimes1_\emptyset"]\\
\sE[\emptyset;\emptyset]\otimes \sE[I;J]\arrow[dr,"\ast", swap] && \sE[I;J]\otimes\sE[\emptyset;\emptyset]\arrow[dl, "\ast"]\\
&\sE[I;J].&
\end{tikzcd}\]

\begin{remark}\label{rmk:WPTensorCat} A wheeled prop $\bE$ is, in particular, a prop, and thus, a symmetric tensor category with a single generating object. The axioms \textbf{H1} -- \textbf{H4} are the axioms that govern the symmetric tensor product on $\bE$. The next set of axioms shows that contractions in $\bE$ are also bi-equivariant, and commute with each other and the horizontal composition, providing the remainder of the rigid symmetric tensor category structure on $\bE$. The unit $1_i$ is the unit for vertical composition -- i.e. composition of morphisms in the tensor category -- which arises as a horizontal composition followed by a contraction. \end{remark}

%%%%%%%%%%%%%%%%%%%%%%%%%%%%%%%%%%%%%%%%%%%%%%%%%%%%%%%%%%%%%%%%

\begin{remark}\label{rmk:Relabelling} In order to precisely state the bi-equivariance axiom for contractions, one needs to establish that any pair of {\em relabelling} bijections $(f,g):(I\times J)\rightarrow (I'\times J')$ induce natural isomorphisms of the vector spaces $\begin{tikzcd} R_{f;g}:\sE[I;J]\arrow[r]&\sE[I';J']\end{tikzcd}$. We leave it as an exercise to the reader to construct these relabelling isomorphisms from the wheeled prop structure.
\end{remark}

\textbf{C1:} Contraction is \emph{bi-equivariant}: for any pair of non-empty label sets $I,J\subseteq \calI$, the following diagram commutes:
\[\begin{tikzcd}
\sE[I;J] \arrow[rr, "\trace_{i}^{j}"] \arrow[d, "(\sigma;\tau)", swap] && \sE[I\setminus\{i\};J\setminus\{j\}] 
\arrow[d, "R_{\sigma |_{I\setminus\{i\}}; \tau |_{J\setminus \{j\}}}"] \\
\sE[\sigma I;J\tau] \arrow[rr, "\trace^{\tau^{-1}(j)}_{\sigma(i)}"]	&& 
\sE[\sigma I\setminus\{i\};J\tau\setminus\{j\}]
\end{tikzcd}\] 
Here $\sigma\in\mathcal{S}_{I}$, $\tau\in\mathcal{S}_{J}$, and $\sigma |_{I\setminus\{i\}}$ and $\tau |_{J \setminus\{j\}}$ are restrictions of the permutations -- note that in general these are no longer permutations, but relabellings, and $R_{\sigma |_{I\setminus\{i\}}; \tau |_{J\setminus \{j\}}}$ is the induced isomorphism as in Remark~\ref{rmk:Relabelling}. 

\textbf{C2:} Contraction maps commute: given any labelling sets $I,J\subseteq \calI$ with $|I|\geq 2$, $|J|\geq 2$, $i\neq k \in I$ and $j\neq l \in J$, then the operations $\trace_{i}^{j}$ and $\trace_{k}^{l}$ commute:

\[\begin{tikzcd}
\sE[I;J] \arrow[r, "\trace_{i}^{j}"] \arrow[d, "\trace_{k}^{l}" swap]	& \sE[I\setminus\{i\};J\setminus\{j\}]\arrow[d, "\trace_{k}^{l}"] \\
\sE[I\setminus\{k\};J\setminus\{l\}] \arrow[r, "\trace_{i}^{j}" swap] & \sE[I\setminus\{i,k\};J\setminus\{j,l\}].\end{tikzcd}\]

\textbf{HC1:} Horizontal composition and contraction maps commute with one another: for any pairs of disjoint subsets $I,K$ and $J,L$ of $\calI$, and any chosen  $i\in I$, $j \in J$, $k \in K$ and $l\in L$, the following two squares commute.

\[\begin{tikzcd}\sE[I;J]\otimes\sE[K;L] \arrow[r, "\ast"] \arrow[d, swap, "\trace_{i}^{j}\otimes id"]& \sE[I\cup K;J\cup L] \arrow[d, "\trace_{i}^{j}"] \\
\sE[I\setminus\{i\};J\setminus\{j\}]\otimes\sE[K;L] \arrow[r, "\ast"] & \sE[I\setminus\{i\}\cup K;J\setminus\{j\}\cup L]
\end{tikzcd}\]

\[\begin{tikzcd}\sE[I;J]\otimes\sE[K;L] \arrow[r, "\ast"] \arrow[d, swap, "id \otimes \trace_{k}^{l}"]& \sE[I\cup K;J\cup L] \arrow[d, "\trace_{k}^{l}"] \\
\sE[I;J]\otimes\sE[K\setminus\{k\};L\setminus\{l\}] \arrow[r, "\ast"] & \sE[I\cup K\setminus\{k\};J\cup L\setminus\{l\}].
\end{tikzcd}\]

\textbf{HC2:} The units $1_i$ are the units for the dioperadic compositions, which are themselves compositions of horizontal compositions and contractions as defined in Remark~\ref{remark: wheeled prop is a prop}.  Specifically, for every pair of sets $I,J\subseteq \calI$ and labels $i \in I, j \in J$, the following diagrams commute. 

\begin{minipage}{.5\linewidth}
	\[\begin{tikzcd} \Bbbk\otimes \sE[I;J]\arrow[d, "1_i\otimes id", swap] \arrow[r, "\cong"]& \sE[I;J]\\
	\sE[\{i\};\{i\}]\otimes \sE[I;J] \arrow[ur, "\trace^i_i\circ \ast", swap]& \end{tikzcd}\]
\end{minipage}
\noindent\begin{minipage}{.5\linewidth}
	\[\begin{tikzcd} \sE[I;J]\otimes \Bbbk\arrow[d, "id\otimes 1_j", swap] \arrow[r, "\cong"]& \sE[I;J]\\
	\sE[I;J]\otimes \sE[\{j\};\{j\}] \arrow[ur, "\trace^j_j\circ \ast", swap]& \end{tikzcd}\]
\end{minipage}

\medskip

In the context of this definition, morphisms $\begin{tikzcd}\bE \arrow[r, "f"] & \mathbf{E}'\end{tikzcd}$ of wheeled props are tensor functors which respect the contractions.

The axiomatic Definition~\ref{def:baised modular operad} is equivalent to Definition~\ref{def:WP}. The key to understanding this is the translation between graph substitution, and the horizontal composition and contraction operations. Observe that any connected graph $G$ -- which is not a free floating loop or edge -- can be constructed from iterated substitution of {\em elementary directed graphs}: graphs which have either one or two vertices, and one or zero internal edges (edges where both flags are part of a vertex).
An $\mathbf{E}$-decorated graph with one vertex represents a composition of contraction operations, and a graph with two vertices a dioperadic composition -- itself a combination of horizontal composition with contractions -- with possibly additional contractions. Figure~\ref{fig:some morphisms in wheeled prop} shows an example of an elementary graph with one vertex, as well as a graph with two vertices obtained from two elementary graphs. 

Thus, a graph $G$ without floating loops and edges represents a sequence\footnote{There is a corresponding statement for circuit algebras, stating that all wiring diagrams are generated via compositions from ``elementary'' wiring diagrams, which realise disjoint unions, dioperadic compositions and contractions.} of iterated contractions and horizontal compositions. The empty unit is represented by the empty graph; the unit by a floating edge, and the floating loop represents a contraction applied to the unit. It is non-trivial to check that the axioms above are equivalent to the algebra structure over the monad of graph substitutions. In the literature this is often called an equivalence of the unbiased definition (monadic) and the biased definition (axiomatic). A full proof of this equivalence can be found, for example, in \cite[11.9.3, Corollary 11.35]{yj15}.

\subsection{Examples}\label{section:examples of wheeled props}
Wheeled props arise in the literature in a range of contexts, for example
as natural wheeled extensions of the associative and commutative operads in \cite{mms}, and have applications in geometry and physics. We recommend the survey article~\cite{merkulov_survey} for full details.

In this section we present an example that in the authors' opinion illuminates some of the structure encoded in a wheeled prop: namely, a wheeled prop whose {\em category of algebras} is the category of semisimple Lie algebras. An {\em algebra over a wheeled prop} $\bW$ is a morphism from $\bW$ to an {\em endomorphism} wheeled prop $\End (\sE)$, defined in Example~\ref{ex:endo}. The term ``algebra'' is somewhat confusing: based on the definition, algebras over a wheeled prop may be more intuitively named {\em representations} of the wheeled prop.

\begin{remark} In fact, the statement we prove below is stronger than simply describing the wheeled prop for semisimple Lie algebras. In \cite{MR1671737}, Kapranov constructs a prop from any operad $\mathsf{P}$, by adjoining a module of $\mathsf{P}$-algebra forms to the prop generated by $\mathsf{P}$. We will show below that for $\mathsf{P}=\mathsf{Lie}$ (the operad for Lie algebras), considering Kapranov's prop as a wheeled prop, a finite dimensional algebra over it that satisfies non-degeneracy conditions for the $\mathsf{P}$-algebra forms is a semisimple Lie algebra.  We take no credit for originality of this construction -- Kapranov's construction preceded the definition of wheeled prop in \cite{mms} by several years.
	
\end{remark} 

The contraction operations in a wheeled prop can be seen as a generalized trace operation.  We begin with the definition of {\em endomorphism} wheeled props, which makes this precise, as there the contraction maps are the standard trace maps of linear algebra. Given $\Bbbk$-vector spaces $\mathsf{U}$, $\sV$, and $\sW$, a linear map $f: \sV\otimes \mathsf{U} \rightarrow \sW \otimes \mathsf{U}$ is given by $f(v_i \otimes u_j)=\Sigma_{k,m}\alpha_{ij}^{km} w_k \otimes u_m$, where $v_i$, $u_j$ and $w_k$ run over a chosen basis for $\sV$, $\mathsf{U}$ and $\sW$, respectively. Recall that the \emph{trace of $f$ with respect to $\mathsf{U}$} is given by $\trace^{\mathsf{u}}_{\mathsf{u}}f(v_i)=\Sigma_{j,k} \alpha_{ij}^{kj} w_k$. If $\sV$ and $\sW$ are one dimensional, this formula reduces to the trace of the matrix of the linear map $f:\mathsf{U}\rightarrow\mathsf{U}$.

\begin{example}\label{ex:endo}
	For simplicity, set the alphabet $\calI=\mathbb{Z}_{\geq 0}$ to be non-negative integers, and use label sets $\underline{n}=\{1,2,...,n\}$.  Fix a finite dimensional $\Bbbk$-vector space $\sE$ and denote its linear dual by $\sE^*$.  We define a family of vector spaces,  for $(n,m)\in\mathbb{Z}_{\geq0}^{2}$: $$\End(\sE)[\underline{n};\underline{m}]: =\Hom_{\Bbbk}(\sE^{\otimes n},\sE^{\otimes m})\cong (\sE^*)^{\otimes n}\otimes \sE^{\otimes m}.$$   Linear maps can be pre- and post-composed with actions of the symmetric groups $\calS_m$ and $\calS_n$ which permute the tensor factors, making the collection $\operatorname{End}(\sE)=\{\operatorname{End}(\sE)[\underline{n};\underline{m}]\}$ into an $\mathcal{S}$-bimodule.

	Using abbreviated notation, write  $$(\phi_1\otimes\ldots\otimes\phi_n)\otimes (w_1\otimes\ldots\otimes w_m)\in(\sE^*)^{\otimes n}\otimes \sE^{\otimes m}\cong \End(\sE)
	[\underline{n};\underline{m}]$$ as $\phi\otimes w$ for short. The horizontal composition
	\[\begin{tikzcd}\ast:\End(\sE)[\underline{n};
	\underline{m}]\otimes \End(\sE)[\underline{k};\underline{l}]\arrow[r]& \End(\sE)[\underline{n +k}; \underline{m+l}]\end{tikzcd}\] is defined as concatenation $(\phi\otimes w)\ast (\phi'\otimes w') := (\phi\otimes\phi')\otimes (w\otimes w')$, and extended linearly.  
	
	Using the same notation, the contraction maps are defined by $$\trace^{j}_{i}(\phi\otimes w):=\phi_i(w_j)\cdot\big((\phi_1\otimes\dots\otimes \phi_{i-1}\otimes\phi_{i+1}\otimes\dots\otimes \phi_{n})\otimes (w_1\otimes\dots\otimes w_{j-1}\otimes w_{j+1}\otimes\dots\otimes w_{m})\big)$$ for any $1\leq i\leq n$ and $1\leq j\leq m$. In other words, following the standard definition of trace above, the contraction operation \[\trace_{i}^{j}:\operatorname{End}(\sE)[\underline{n};\underline{m}]\rightarrow\operatorname{End}(\sE)[\underline{n-1};\underline{m-1}]\] applied to a linear map $\phi\otimes w$ in $\End(\sE)[\underline{n};\underline{m}]$ given by $\trace_{i}^{j}(\phi\otimes w)$ is the classical trace detailed above with respect to the $i$th copy of $E^*$ and the $j$th copy of $E$, using the isomorphism $\mathsf{E}\cong \mathsf{E}^*$ specified by the choice of basis.
	See \cite[Example 2.1.1]{mms} for full details. 
	
	\begin{figure}[h]
		\includegraphics[height=4cm]{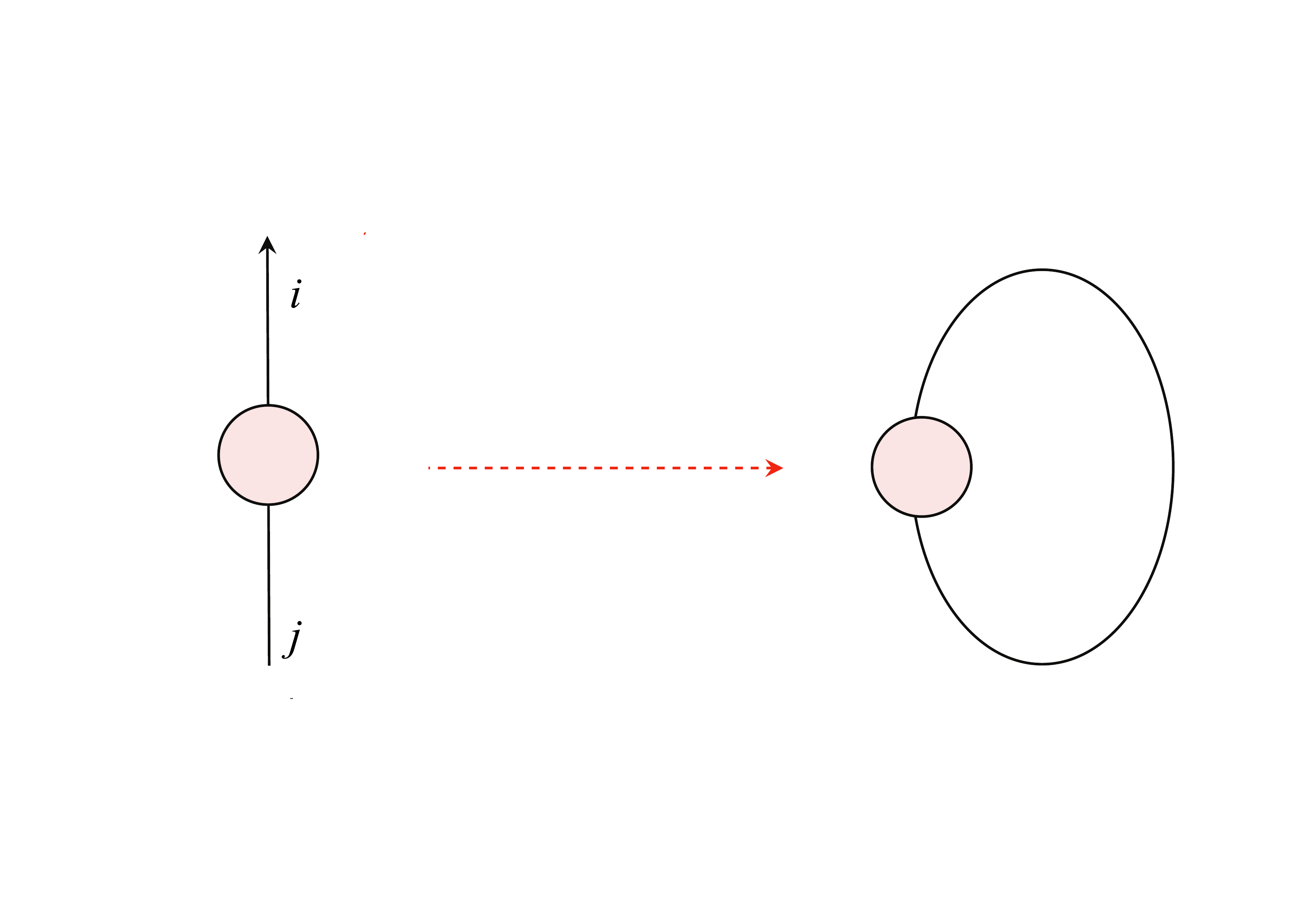}
		\caption{A graphical interpretation of trace. }\label{fig:trace}
	\end{figure}
	
\end{example}

\begin{example}
	The main example for this section is the wheeled prop $\mathsf{Lie}^w$ for semisimple Lie algebras. We assemble the vector spaces $\mathsf{Lie}^w[n;m]$ from two ingredients: the vector spaces $\mathsf{Lie}(n)$ generated by free Lie words on $n$ letters, and formal traces $\trace(p)$ for a Lie word $p$.
	
	Let $\mathsf{Lie}(n)$ denote the $\Bbbk$-vector space spanned by all the Lie words in the free Lie algebra generated by letters $x_1,\ldots, x_n$, with each letter $x_i$ appearing exactly once. Diagrammatically, such Lie words are represented by directed trivalent graphs with $n$ inputs labelled $x_1,\ldots, x_n$, and $1$ output; satisfying that every trivalent vertex has two inputs and one output; and these graphs are considered modulo the antisymmetry and Jacobi relations of Figure~\ref{fig:Lie_Relations}. For the reader familiar with operads, these are the spaces that make up the arity $n$-operations of the operad $\mathsf{Lie}.$

	\begin{figure}%[h]
		\centering
		\begin{subfigure}{.5\textwidth}
			\centering
			\includegraphics[height=4cm]{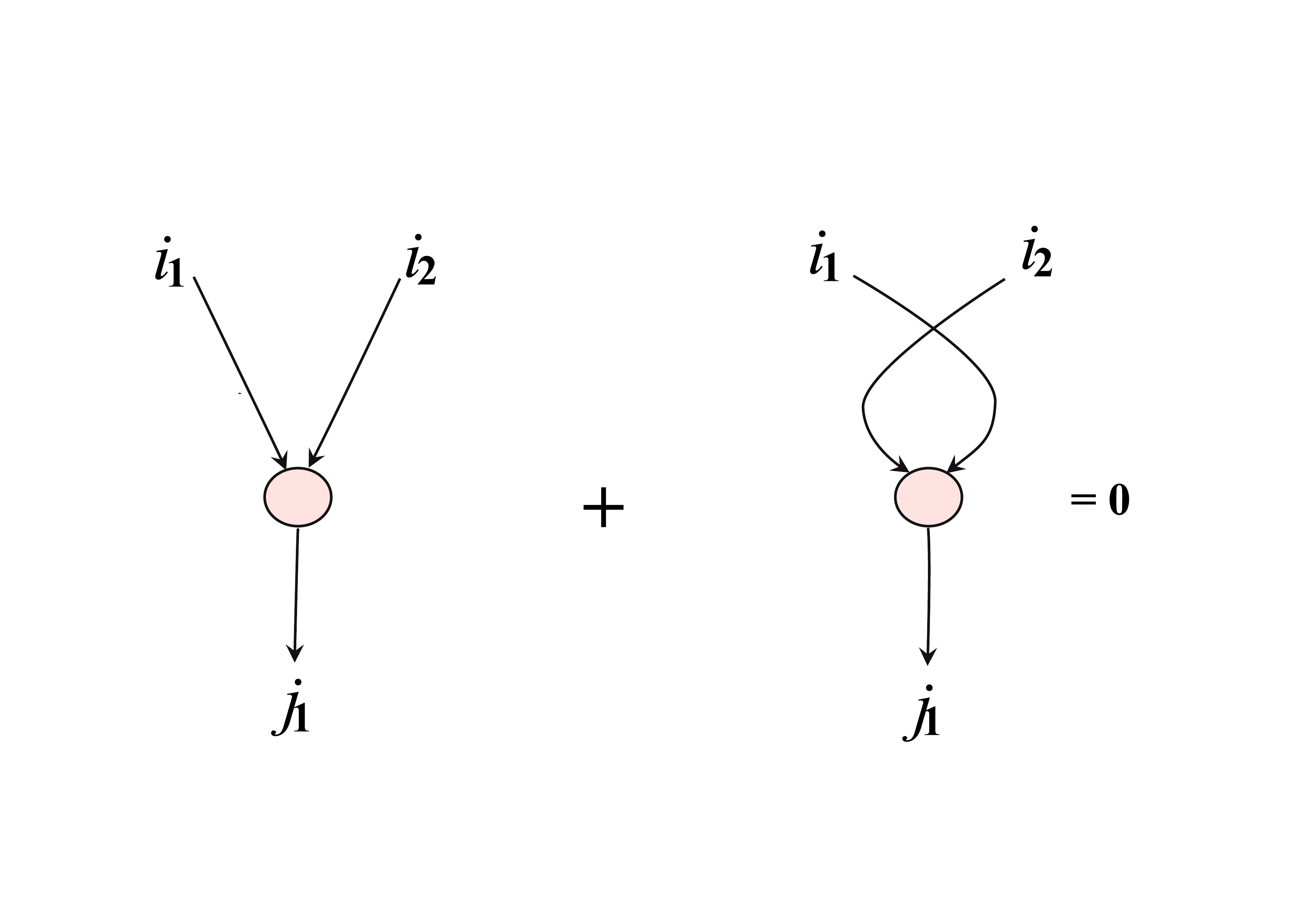}
		\end{subfigure}%
		\begin{subfigure}{.5\textwidth}
			\centering
			\includegraphics[height=4.5cm]{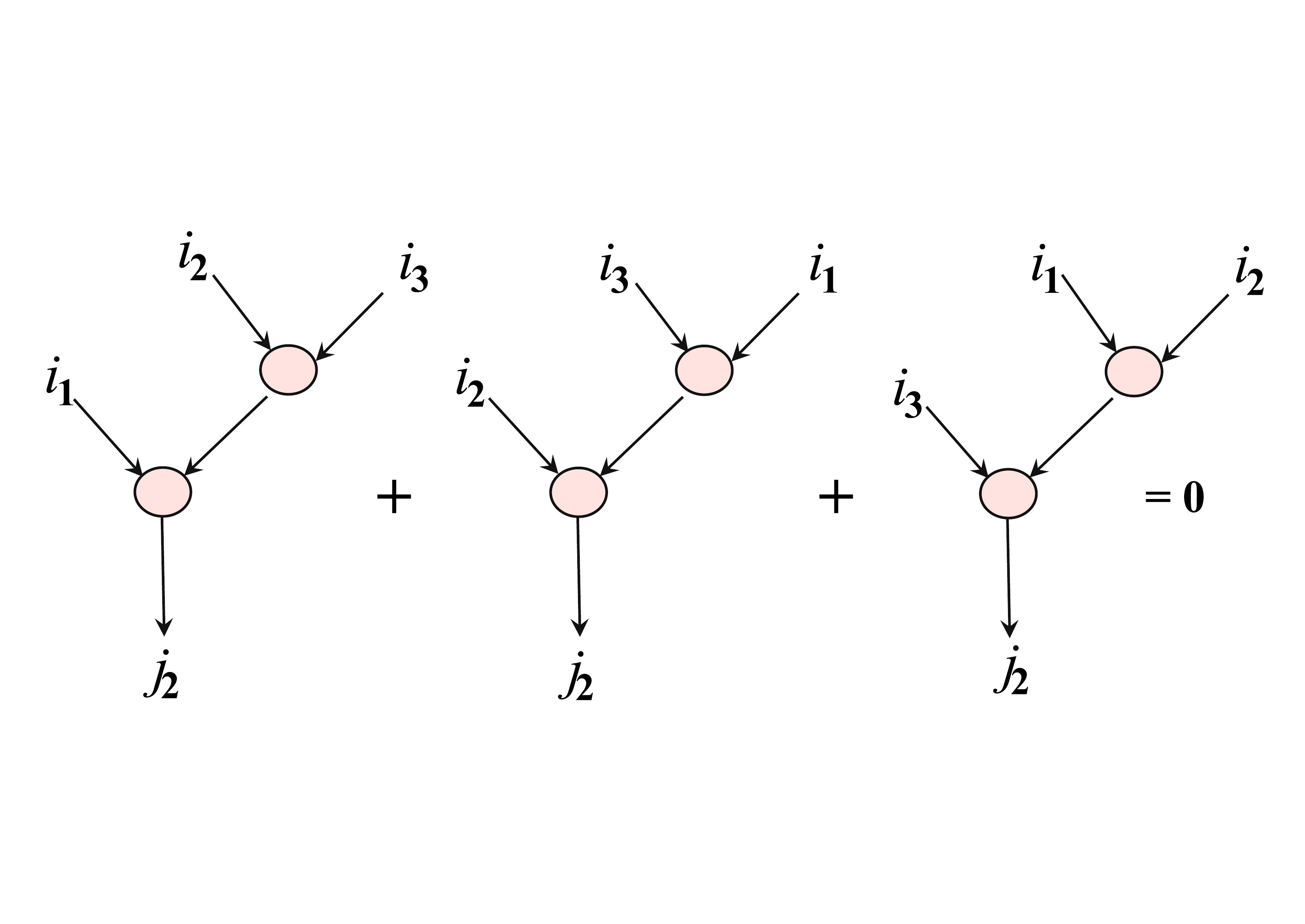}
		\end{subfigure}
		\caption{ The antisymmetry and Jacobi relations, to be understood locally, that is, $i_1$, $i_2$, $i_3$, $j_1$ and $j_2$ each denote trivalent graphs. }\label{fig:Lie_Relations}
	\end{figure}  
	
	The vector space $\mathsf{Lie}(n)$ admits a natural $\calS_n$ action, by permuting the letters $x_i$. The prop\footnote{Precisely, this is the operad for Lie algebras, considered as a prop as in \cite{BV}.} $\mathsf{Lie}$, whose algebras are Lie algebras, is freely generated -- using horizontal and vertical compositions -- by setting $\mathsf{Lie}[n,1]:= \mathsf{Lie}(n)$. We will not describe the prop $\mathsf{Lie}$ in detail, but instead adjoin formal trace operations to obtain the wheeled prop $\mathsf{Lie}^w$.
	
	Denote by $\trace(p)$ the result of identifying the single output of a Lie word $p \in \mathsf{Lie}(n+1)$ with its last input, as in Figure~\ref{fig:trace(p)}. Denote by $\mathsf{Lie}[n;0]$, for $n\geq 0$, the vector space formally spanned by the symbols $\trace(p)$ for $p\in \mathsf{Lie}(n+1)$. The symmetric group $\calS_n$ acts on $\mathsf{Lie}[n;0]$ by permuting the first $n$ letters of $p$. The symbols $\trace(p)$ satisfy the following relations:
	
	\begin{enumerate} 
		\item For $\sigma\in \calS_n$, $\trace(\sigma p)=\sigma\trace(p)$, where $\sigma p$ is the action of $\sigma$ on $p$ via the standard embedding $\calS_n \hookrightarrow \calS_{n+1}$.
		\item For $p \in \mathsf{Lie(n+1)}$ and $q \in \mathsf{Lie(m+1)}$, write $p\circ_{i} q\in \mathsf{Lie}(n+m+1)$ for the dioperadic composition given by gluing the unique output of $q$ to the $i$th input of $p$ (This would be denoted $p\hspace{-0.1cm}\phantom{\circ}_{i}{\circ}_1 q$  in Remark~\ref{remark: wheeled prop is a prop}, we have dropped the indexing of the output of $q$, as it is unique). Then, $\trace(p \circ_{i} q)=\trace(p)\circ_{i} q$, whenever $p\in\mathsf{Lie}(n+1)$, $q\in\mathsf{Lie}(m+1)$ and $i\not=n+1$. 
		\item Finally, the trace operations are cyclically symmetric: $\trace(p \circ_{n+1} q)= \beta\trace(q \circ_{m+1} p)$, where $\beta$ is the block transposition of $[1,...,n]$ and $[1,...,m]$, as shown in Figure~\ref{fig:Rel3}. \end{enumerate}
	
	\begin{figure}%[ht]
		\centering
		\includegraphics[height=4cm]{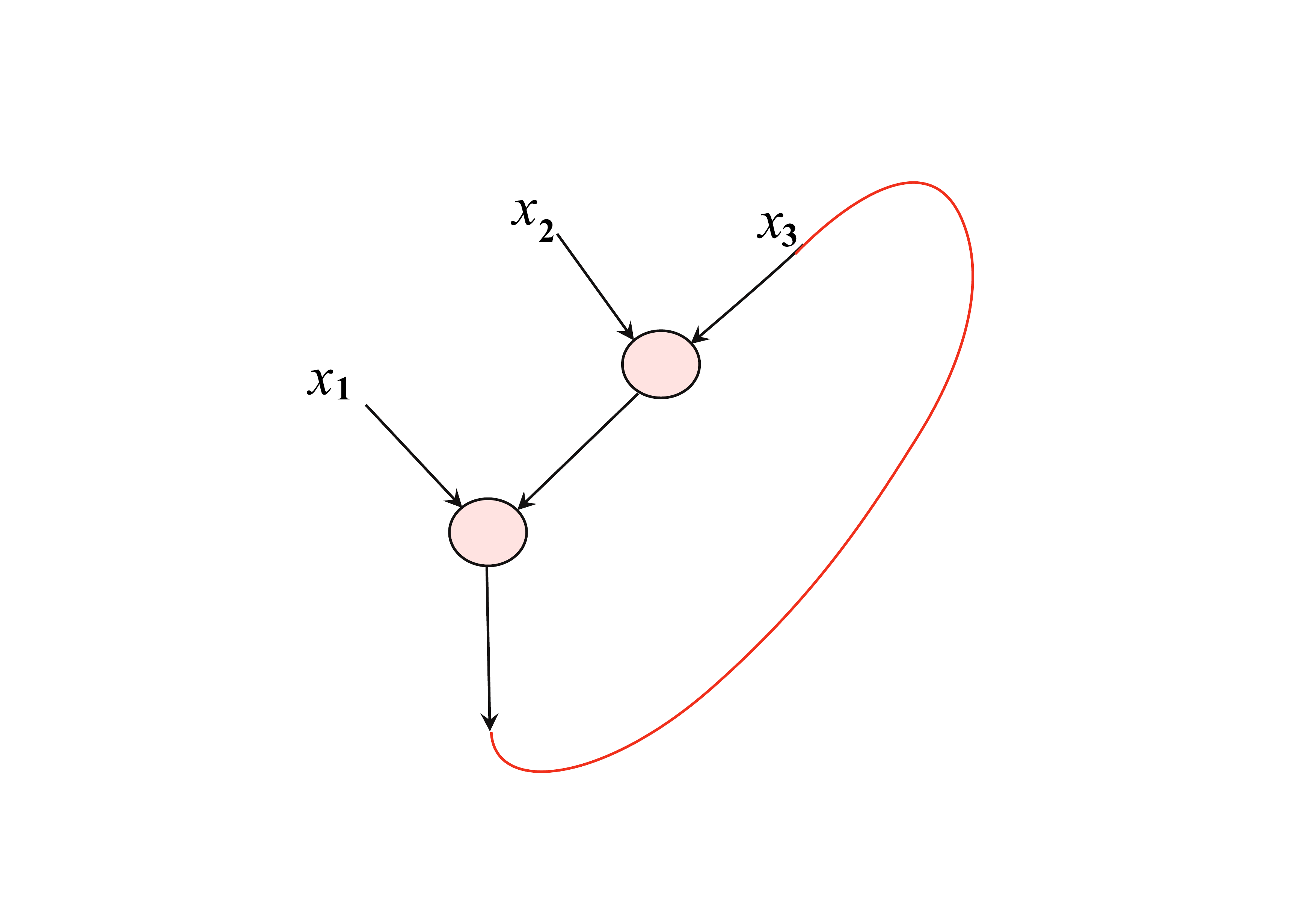}
		\caption{ An element $\trace(p)$ in $\mathsf{Lie}[2;0]$ }\label{fig:Lie_Relations}\label{fig:trace(p)}
	\end{figure}
	
	\begin{figure}
		\centering
		\includegraphics[height=5cm]{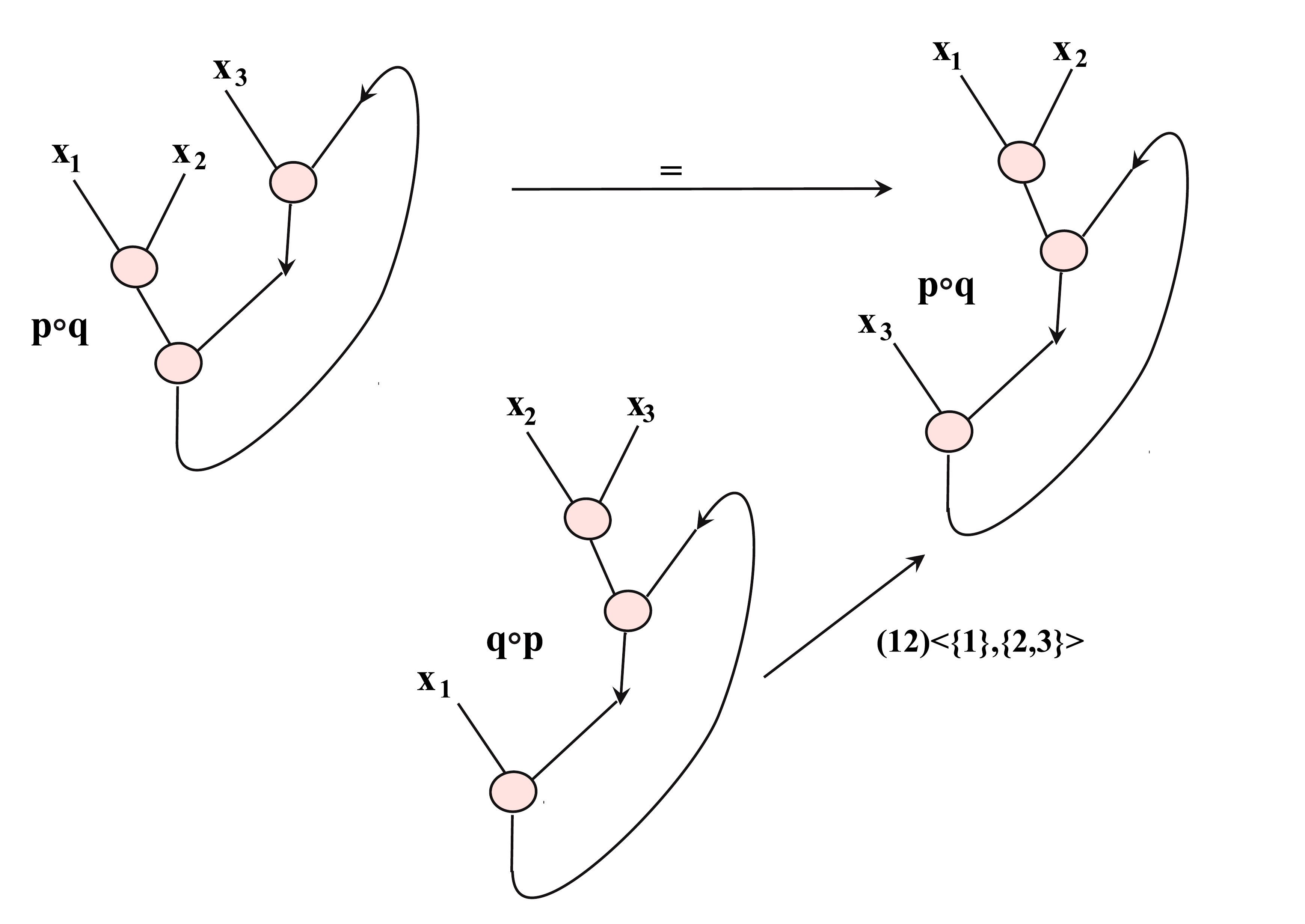}
		\caption{An example of relation $(3)$ for $\mathsf{Lie}[n;0]$.}\label{fig:Rel3}
	\end{figure} 
	
	Now we're ready to define $\mathsf{Lie}^w=\{\mathsf{Lie}^w[n;m]\}$, by first setting
	\[\mathsf{Lie}^w[n;m]=\bigoplus_{A_1\cup\dots\cup A_m\cup B_1\cup\dots\cup B_r} \bigotimes_{i}\mathsf{Lie}(a_i)  \otimes \bigotimes_{j} \mathsf{Lie}[b_{j};0],\] where the sets $A_1\cup\dots\cup A_m\cup B_1\cup\dots\cup B_r$ run over all partitions of the set $\underline{n}$, the numbers $a_i=|A_i|$, $b_j=|B_j|$ denote their cardinalities, and the words in $\mathsf{Lie(a_i)}$ are on the letters $x_\alpha$ for $\alpha \in A_i$. Diagrammatically, these spaces are spanned by graphs whose connected components are the Lie graphs in $\mathsf{Lie}(n)$ or trace graphs $\trace(p)$. The permutation group $\calS_n$ acts by permuting the input labels $x_\alpha$ of the graphs. There is also a right action by $\calS_m$, where $\tau\in \calS_m$ acts as $\tau^{-1}$ on the sets $A_i$ and tensor factors $\mathsf{Lie}(a_i)$.

	Horizontal composition on $\mathsf{Lie}^w$ is given by concatenation of tensor factors, that is, disjoint union of graphs. Contraction on $\mathsf{Lie}^w$ is given by identifying an output of a graph in $\mathsf{Lie}^w[n;m]$ with an input. This is either a dioperadic composition (joining two separate connected components), or a permuted trace symbol (connecting the output of a connected component to one of its own inputs). Next, we need to check that this structure satisfies the axioms of a wheeled prop.
	
	The horizontal composition axioms are easy to verify. Essentially, the contraction axioms follow from the relations (1), (2) and (3) imposed on $\mathsf{Lie}[n;0]$. For example, the equivariance axiom {\bf C1} is true by the relation (1) for traces.
	
	To verify the commutativity of contraction maps, {\bf C2}, one needs to analyse different cases depending on whether the two contraction maps are of the ``dioperadic composition'' or the ``trace'' type: 
	\begin{itemize}
		\item For two dioperadic compositions -- that is, if the two contractions involve three or four separate connected components -- the axiom clearly holds.
		The same is true for a dioperadic composition with a trace map on a separate component; and for two trace maps on separate components. 
		\item For a dioperadic composition between components, and a permuted trace on one of those components, the axiom {\bf C2} holds by the relation (2).
		\item Given two dioperadic compositions between the same pair of Lie graph components, the first is performed as a dioperadic composition, while the second is a permuted trace map on the resulting graph. If each outgoing edge is connected to the last incoming edge of the other graph, then the axiom {\bf C2} follows directly from property (3), see Figure~\ref{fig:Rel3} to visualize this. If the outgoing edges are connected to other incoming edges, then one applies antisymmetry permutations to reduce this to the earlier scenario. 
	\end{itemize}
	
	We leave it as an exercise to the reader to verify the remaining axioms. To summarize, the relations (1), (2) and (3) make $\mathsf{Lie}^w$ a wheeled prop. Note that the converse is also true: the wheeled prop axioms force the relations (1), (2) and (3). In other words, $\mathsf{Lie}^w$ can equivalently be described as the wheeled prop generated by the Lie word $[x_1,x_2]$, denoted $\YGraph$; using horizontal compositions, contractions and units; and modulo the Anti-Symmetry and Jacobi relations:
	\[\mathsf{Lie}^w=\text{WP}\left\langle \YGraph \, | \, \text{Anti-Symmetry}, \text{Jacobi}\right\rangle\] 
	
	In effect, the symbols in $\mathsf{Lie}[n;0]$ parametrise generalized Killing forms $\kappa_n:=\trace([x_1,[x_2,...[x_n, x_{n+1}]...]).$ Indeed, Proposition 3.4.4 \cite{MR1671737} shows that, as vector spaces, $\mathsf{Lie}[n;0]$ is isomorphic to the vector space whose basis is given by the non-cyclic permutations of the $\kappa_n.$

	An \emph{algebra} over the wheeled prop $\mathsf{Lie}^w$ is, by definition, a morphism of wheeled props $\alpha:\mathsf{Lie}^w\rightarrow \End(\sE)$. As such, every algebra is determined by where it sends the wheeled prop generator. The image of $\YGraph$ picks out a bracket in $\End(\sE)[\underline{2};\underline{1}]=\Hom_{\Bbbk}(\sE^{\otimes 2},\sE)$. This is subject to the Anti-Symmetry and Jacobi relations which hold in $\mathsf{Lie}^w$. 
	
	Moreover, the elements $\trace(p)\in \mathsf{Lie}[n;0]$ (which are themselves obtained from $\YGraph$ using horizontal compositions and contractions) are sent to the Killing forms $\kappa_n = \trace([x_1,[x_2,...[x_n, x_{n+1}]...])$ in $\End(\sE^{\otimes n},\Bbbk)$. Note that relation $(3)$ guarantees that this is a symmetric bilinear form. If the target vector space $\sE$ is a finite-dimensional vector space, then the $\kappa_n$'s are the Killing forms $$x_1\otimes\ldots \otimes x_n\mapsto \text{tr}(\text{ad}(x_1)\ldots\text{ad}(x_n)).$$ In this finite dimensional case it follows that an algebra $\alpha: \mathsf{Lie}^w\rightarrow \End(\sE)$ that takes $\trace([x_1,[x_2,x_3]])$ to a non-degenerate form makes $\sE$ into a semisimple Lie algebra.
\end{example}

\subsection{Wheeled props and pivotal categories}\label{subsec:Pivotal}
A \emph{pivotal category} is a particular kind of tensor category with a notion of dual. That is, a tensor category $\mathcal{C}$, equipped with a (strict) contravariant functor of monoidal categories $(-)^{*}$, with $(-)^{**}=\text{id}_{\mathcal{C}}$, and a family of maps $\epsilon_{c}:c\otimes c^*\rightarrow I$ for each $c\in\mathcal{C}$ (here $I$ is the unit object of $\mathcal{C}$), which satisfy axioms $P1,P2$ and $P3$ of  \cite[Definition 1.3]{MR1020583} (a more general version can be found in \cite{MR1250465}). 

As mentioned in the introduction, the category of planar algebras is equivalent to the category of pivotal categories with a symmetrically self-dual generator \cite{MR2559686}. It is a straightforward exercise to check that a wheeled prop is, in particular, a pivotal category with a single generator; since we couldn't find any statement of this fact in the literature we include a proof sketch here. There is no expectation however for the opposite direction to hold, i.e. not every pivotal category is a wheeled prop. 

\begin{prop}\label{prop: wheeled props are pivcat}
	There exists a fully faithful embedding $\rho: \mathsf{wProp}\hookrightarrow \mathsf{PivCat}$ from the category of linear wheeled props to the category of linear pivotal categories with a single generator. 
\end{prop}

\begin{proof}
	Let $\mathsf{A}$ be a linear wheeled prop with generating object $x$ (as in Remark~\ref{rmk:WPTensorCat}). We will show directly that $\mathsf{A}$ is a pivotal category. Objects in $\mathsf{A}$ are generated by a single object $x$ -- the ``colour'' of all outgoing edges of a graph. The {\em dual} of $x$ colours incoming edges, and $(x^{\otimes n})^* =(x^*)^{\otimes n}$.  Since the symmetric tensor product $\otimes$ on $\mathsf{A}$ is horizontal composition,  it is clear that $(-)^*$ is a tensor contravariant functor on $\mathsf{A}$. The trace map $\epsilon_x:x\otimes x^*\rightarrow x^{0}=I$ is given by the contraction $\epsilon_x=\trace_{1}^{1}$, and can be generalised to all objects by iterated applications of contractions, e.g. $\epsilon_{x^{\otimes n}}: x^{\otimes n}\otimes (x^{\otimes n})^*\rightarrow I$. 
	
	The axioms $P1, P2$ and $P3$ of \cite{MR1020583} now follow from the axioms of a wheeled prop. As just one example, the axiom $P1$ states that for any $x^{\otimes k}, x^{\otimes m}$ in $x^{\otimes n}$ the following diagram commutes: 
	\[\begin{tikzcd} 
	(x^{\otimes k}\otimes x^{\otimes m})\otimes (x^{\otimes m*}\otimes x^{\otimes k*})\arrow[ddd] \arrow[r]& x^{\otimes k}\otimes (x^{\otimes m}\otimes (x^{\otimes m*}\otimes x^{\otimes k*}))\arrow[r]& x^{\otimes k}\otimes( (x^{\otimes m}\otimes x^{\otimes m*})\otimes x^{\otimes k*}) \arrow[d]\\
	&& x^{\otimes k}\otimes (I\otimes x^{\otimes k*}) \arrow[d]\\
	&&x^{\otimes k}\otimes x^{\otimes k*}\arrow[d]\\
	(x^{\otimes k}\otimes x^{\otimes m})\otimes (x^{\otimes k}\otimes x^{\otimes m})^* \arrow[rr]&& I
	\end{tikzcd}\] One can check that this indeed holds in any wheeled prop $\mathsf{A}$, by axiom \textbf{H1} and iterated applications of the contraction operations. 
	The remaining axioms follow in a similar manner. 
	
	As a morphism $f:\mathsf{A}\rightarrow \mathsf{B}$ between wheeled props is, in particular, a monoidal functor between rigid tensor categories, it is also a morphism between pivotal categories. It follows that the category of $\mathsf{wProp}$ is subcategory of $\mathsf{PivCat}$ and the natural inclusion gives a fully faithful embedding $\mathsf{wProp}\hookrightarrow \mathsf{PivCat}$. 
\end{proof} 

\begin{remark}
	Pivotal categories are strictly more general than wheeled props and there is no claim that the functor $\rho$ is an equivalence. For formal reasons, there exists an adjoint to $\rho$ (similiar to \cite[Proposition 5.2]{MR1357057}). This leads one to question if and when one can assemble the equivalence between circuit algebras and linear wheeled props with the classification of planar algebras in terms of pivotal categories into a commutative diagram:
	\[\begin{tikzcd} \mathsf{CA} \arrow[d, shift left = 1, hook]\arrow[r, "\cong"] & \mathsf{wProp} \arrow[l]\arrow[d, shift right=1, hook]\\
	\mathsf{PA}\arrow[r, "\cong"] \arrow[u, shift left=1] & \mathsf{PivCat} \arrow[l]\arrow[u, shift right=1]
	\end{tikzcd}\] This problem remains open and will be pursued in future work. Moreover, the category of pivotal categories is naturally a $2$-category and one would like to see the equivalences promoted to equivalences of $2$-categories. It is not known if wheeled props have a natural $2$-category structure, though it is likely, as it is already known that props admit the structure of a $2$-monoid (See~\cite[Section 8]{MR2559644} or \cite{MR2116328}). 
	
\end{remark}

\bibliographystyle{amsalpha}
\bibliography{circuit}

\providecommand{\bysame}{\leavevmode\hbox to3em{\hrulefill}\thinspace}
\providecommand{\MR}{\relax\ifhmode\unskip\space\fi MR }
\providecommand{\MRhref}[2]{%
  \href{http://www.ams.org/mathscinet-getitem?mr=#1}{#2}
}
\providecommand{\href}[2]{#2}
\begin{thebibliography}{{Mer}10b}

\bibitem[{Bar}05]{BN:KHTangles}
D.~{Bar-Natan}, \emph{Khovanov's homology for tangles and cobordisms}, Geometry
  and Topology \textbf{9} (2005), no.~33, 1443--1499.

\bibitem[BD16]{BND:WKO1}
D.~{Bar-Natan} and Z.~{Dancso}, \emph{Finite type invariants of w-knotted
  objects i: braids, knots and the {A}lexander polynomial}, Algebraic and
  Geometric Topology \textbf{16} (2016), no.~2, 1063--1133.

\bibitem[BD17]{BND:WKO2}
D.~{Bar-Natan} and Z.~Dancso, \emph{Finite type invariants of w-knotted objects
  ii: tangles, foams and the kashiwara-vergne conjecture}, Mathematische
  Annalen \textbf{367} (2017), no.~3-4, 1517--1586.

\bibitem[BHP12]{BHP12}
A.~{Brothier}, M.~{Hartglass}, and D.~{Penneys}, \emph{Rigid {$C^*$}-tensor
  categories of bimodules over interpolated free group factors}, J. Math. Phys.
  \textbf{53} (2012), no.~12, 123525, 43. \MR{3405915}

\bibitem[BM07]{bm_resolutions}
C.~{Berger} and I.~{Moerdijk}, \emph{Resolution of coloured operads and
  rectification of homotopy algebras}, 31--58. \MR{2342815}

\bibitem[{Bro}19]{Brochier}
A.~{Brochier}, \emph{Virtual tangles and fiber functors}, J. Knot Theory
  Ramifications \textbf{28} (2019), no.~7, 1950044, 17. \MR{3975573}

\bibitem[BV73]{BV}
J.~M. {Boardman} and R.~M. Vogt, \emph{Homotopy invariant algebraic structures
  on topological spaces}, Lecture Notes in Mathematics, Vol. 347,
  Springer-Verlag, Berlin-New York, 1973. \MR{0420609}

\bibitem[CFP20]{CFP20}
P.~J. {Clavier}, L.~Foissy, and S.~Paycha, \emph{Props of graphs and
  generalised traces}, 2020, preprint: arXiv:2005.02115.

\bibitem[{Del}90]{MR1106898}
P.~{Deligne}, \emph{Cat\'{e}gories tannakiennes}, The {G}rothendieck
  {F}estschrift, {V}ol. {II}, Progr. Math., vol.~87, Birkh\"{a}user Boston,
  Boston, MA, 1990, pp.~111--195. \MR{1106898}

\bibitem[DF18]{DF}
C.~{Damiani} and V.~{Florens}, \emph{Alexander invariants of ribbon tangles and
  planar algebras}, J. Math. Soc. Japan \textbf{70} (2018), no.~3, 1063--1084.
  \MR{3830799}

\bibitem[DK05]{DK:Virtual}
H.~{Dye} and L.~H. Kauffman, \emph{Virtual knot diagrams and the
  {W}itten--{R}eshetikin--{T}uraev invariant}, Journal of knot theory and its
  ramifications \textbf{14} (2005), no.~8, 1045--1075.

\bibitem[DM19]{DM19}
H.~{Derksen} and V.~Makam, \emph{Invariant theory and wheeled props}, 2019,
  preprint: arXiv:1909.00443.

\bibitem[FY89]{MR1020583}
P.~J. {Freyd} and D.~N. Yetter, \emph{Braided compact closed categories with
  applications to low-dimensional topology}, Adv. Math. \textbf{77} (1989),
  no.~2, 156--182. \MR{1020583}

\bibitem[{Hal}16]{Hal16}
I.~{Halacheva}, \emph{Alexander type invariants of tangles}, arXiv:1611.09280,
  2016.

\bibitem[HP17]{MR3624399}
M.~{Hartglass} and D.~Penneys, \emph{{${\rm C}^*$}-algebras from planar
  algebras {I}: {C}anonical {${\rm C}^*$}-algebras associated to a planar
  algebra}, Trans. Amer. Math. Soc. \textbf{369} (2017), no.~6, 3977--4019.
  \MR{3624399}

\bibitem[HPT]{HPT16}
A.~{Henriques}, D.~Penneys, and J.~Tener, \emph{Planar algebras in braided
  tensor categories, 2016}, arXiv preprint arXiv:1607.06041.

\bibitem[{Jon}99]{Jones:PA}
V.~F.~R. {Jones}, \emph{Planar algebras i}, arXiv:math.QA/9909027, 1999.

\bibitem[JS93]{MR1250465}
A.~{Joyal} and R.~Street, \emph{Braided tensor categories}, Adv. Math.
  \textbf{102} (1993), no.~1, 20--78. \MR{1250465}

\bibitem[JSV96]{MR1357057}
A.~{Joyal}, R.~Street, and D.~Verity, \emph{Traced monoidal categories}, Math.
  Proc. Cambridge Philos. Soc. \textbf{119} (1996), no.~3, 447--468.
  \MR{1357057}

\bibitem[JY09]{MR2559644}
M.~W. {Johnson} and D.~Yau, \emph{On homotopy invariance for algebras over
  colored {PROP}s}, J. Homotopy Relat. Struct. \textbf{4} (2009), no.~1,
  275--315. \MR{2559644}

\bibitem[{Kap}99]{MR1671737}
M.~{Kapranov}, \emph{Rozansky-{W}itten invariants via {A}tiyah classes},
  Compositio Math. \textbf{115} (1999), no.~1, 71--113. \MR{1671737}

\bibitem[{Kup}03]{Kup}
G.~{Kuperberg}, \emph{What is a virtual link?}, Algebraic and Geometric
  Topology \textbf{3} (2003), no.~1, 587--591.

\bibitem[{Lac}04]{MR2116328}
S.~{Lack}, \emph{Composing {PROPS}}, Theory Appl. Categ. \textbf{13} (2004),
  No. 9, 147--163. \MR{2116328}

\bibitem[{Lei}04]{MR2094071}
T.~{Leinster}, \emph{Higher operads, higher categories}, London Mathematical
  Society Lecture Note Series, vol. 298, Cambridge University Press, Cambridge,
  2004. \MR{2094071}

\bibitem[{Mac}65]{MacLane_Props}
S.~{Mac Lane}, \emph{Categorical algebra}, Bull. Amer. Math. Soc. \textbf{71}
  (1965), 40--106. \MR{171826}

\bibitem[{Mar}08]{markl_operads_and_props}
M.~{Markl}, \emph{Operads and {PROP}s}, Handbook of algebra. {V}ol. 5, Handb.
  Algebr., vol.~5, Elsevier/North-Holland, Amsterdam, 2008, pp.~87--140.
  \MR{2523450}

\bibitem[{Mer}10a]{MR2648322}
S.~A. {Merkulov}, \emph{Wheeled props in algebra, geometry and quantization},
  European {C}ongress of {M}athematics, Eur. Math. Soc., Z\"{u}rich, 2010,
  pp.~83--114. \MR{2648322}

\bibitem[{Mer}10b]{merkulov_survey}
\bysame, \emph{Wheeled props in algebra, geometry and quantization}, European
  {C}ongress of {M}athematics, Eur. Math. Soc., Z\"{u}rich, 2010, pp.~83--114.
  \MR{2648322}

\bibitem[{Mer}11]{MR2762550}
\bysame, \emph{Permutahedra, {HKR} isomorphism and polydifferential
  {G}erstenhaber-{S}chack complex}, Higher structures in geometry and physics,
  Progr. Math., vol. 287, Birkh\"{a}user/Springer, New York, 2011,
  pp.~293--314. \MR{2762550}

\bibitem[{Mer}16]{MR3451536}
\bysame, \emph{Formality theorem for quantizations of {L}ie bialgebras}, Lett.
  Math. Phys. \textbf{106} (2016), no.~2, 169--195. \MR{3451536}

\bibitem[MMS09]{mms}
M.~{Markl}, S.~Merkulov, and S.~Shadrin, \emph{Wheeled {PROP}s, graph complexes
  and the master equation}, J. Pure Appl. Algebra \textbf{213} (2009), no.~4,
  496--535. \MR{2483835}

\bibitem[MPS10]{MR2559686}
S.~{Morrison}, E.~{Peters}, and N.~{Snyder}, \emph{Skein theory for the
  {$D_{2n}$} planar algebras}, J. Pure Appl. Algebra \textbf{214} (2010),
  no.~2, 117--139. \MR{2559686}

\bibitem[{Pol}10]{polyak2010alexanderconway}
M.~{Polyak}, \emph{Alexander-{C}onway invariants of tangles}, 2010.

\bibitem[{Tub}14]{Tub}
D.~{Tubbenhauer}, \emph{Virtual {K}hovanov homology using cobordisms}, J. Knot
  Theory Ramifications \textbf{23} (2014), no.~9, 1450046, 91. \MR{3268982}

\bibitem[YJ15]{yj15}
D.~{Yau} and M.~W. Johnson, \emph{A {F}oundation for {PROP}s, {A}lgebras, and
  {M}odules}, Mathematical Surveys and Monographs, vol. 203, American
  Mathematical Society, Providence, RI, 2015. \MR{3329226}

\end{thebibliography}
\end{document}